\documentclass[10pt]{article}

\usepackage{amssymb}
\usepackage{amsmath}
\usepackage{theorem}
\usepackage{epsfig}
\usepackage{verbatim}
\usepackage{graphicx}
\usepackage{subfigure}
\usepackage{cancel}
\usepackage{epstopdf}
\usepackage{a4wide}
\usepackage{pdflscape}
\usepackage{color}
\usepackage{comment}
\newtheorem{thm}{Theorem}[section]
\newtheorem{proposition}[thm]{Proposition}

\newtheorem{defi}[thm]{Definition}
\newtheorem{rem}[thm]{Remark}

\newtheorem{lemma}[thm]{Lemma}

\newcommand{\R}{\mathbb{R}}
\newcommand{\ep}{\varepsilon}
\newcommand{\pa}{\partial}

\newcommand{\xb}{\textbf{x}}
\newcommand{\yb}{\textbf{y}}
\newcommand{\fb}{\textbf{f}}

\newcommand{\tb}{\textbf{t}}
\newcommand{\nb}{\textbf{n}}
\newcommand{\eb}{\textbf{e}}

\newcommand{\zb}{\textbf{z}}
\newcommand{\ub}{\textbf{u}}
\newcommand{\mb}{\textbf{m}}

\newcommand{\rhos}{\rho^\sharp}

\newcommand{\hs}{h^\sharp}

\newcommand{\sign}{\text{sign}}
\newcommand{\lar}{\langle A \rangle}
\newcommand{\ok}{\overline{k}}
\newcommand{\dl}{\left(y^2+\theta^2\right)}
\newcommand{\dlx}{\left(y^2+\theta_{lin}^2\right)}

\newcommand{\Op}{\textrm{Op}}
\newcommand{\Hw}{\mathcal{S}_{1,1}}

\numberwithin{equation}{section}

\allowdisplaybreaks[4]

\title{Mixing solutions for the Muskat problem}
\author{A. Castro, D. C\'ordoba \& D. Faraco}

\newenvironment{proof}{\begin{trivlist} \item[] {\em Proof:}}{\hfill $\Box$
                       \end{trivlist}}

\makeatletter
\renewcommand*\l@section{\@dottedtocline{1}{0em}{1.5em}}
\renewcommand*\l@subsection{\@dottedtocline{2}{1.5em}{2.3em}}
\renewcommand*\l@subsubsection{\@dottedtocline{3}{3.8em}{3.7em}}
\makeatother

\begin{document}

\maketitle

\begin{abstract}
We prove the existence of mixing solutions of the incompressible porous media equation for all Muskat type $H^5$ initial data in the fully unstable regime.
The proof combines convex integration, contour dynamics and a basic calculus for non smooth semiclassical type pseudodifferential operators which is developed.
\end{abstract}

\tableofcontents

\section{Introduction and the main theorem}
The dynamics of an incompressible fluid in an homogeneous and isotropic porous media is modeled by the following system
\begin{align}
\pa_t\rho + \ub\cdot \nabla \rho = 0 &\quad \text{in $\Omega$}\label{mcl}\\
\nabla\cdot \ub = 0 & \quad\text{in $\Omega$}\label{ic}\\
\frac{\nu}{\kappa} \ub = -\nabla p -\rho\textbf{g}\label{dl}&\quad \text{in $\Omega$},
\end{align}
where $\rho$ is the density, $\ub$ is the incompressible velocity field, $p$ is the pressure, $\nu$ is the  viscosity, $\kappa$ is the permeability of the media and $\textbf{g}$ is the gravity. The first equation represents the mass conservation law, equation \eqref{ic} the incompressibility of the fluid and equation \eqref{dl} is Darcy's law, which relates the velocity of the fluid with the forces acting on it. In this paper we will consider $\Omega=\R^2$.  As usual, we will refer to the system \eqref{mcl}, \eqref{ic} and \eqref{dl} as the IPM system.

The Muskat problem deals with two incompressible and immiscible fluids in a porous media with different constant densities $\rho^+$ and $\rho^-$ and  different constant viscosities. In this work
we will focus on the case in which both fluids have the same viscosity. Then  one can obtain the following system of equations from IPM
\begin{align}
\nabla \cdot \ub =0 &\quad \text{in $\Omega^\pm(t)$}\label{uno}\\
\nabla^\perp \cdot \ub =0 &\quad \text{in $\Omega^\pm(t)$}\\
\frac{\nu}{\kappa}(\ub^+-\ub^-)\cdot \tb = -g(\rho^+-\rho^-)(0,1)\cdot \tb  & \quad \text{on $\Gamma(t)$}\\
(\ub^+-\ub^-)\cdot \nb = 0  &\quad \text{in $\Gamma(t)$}\\
\pa_t \textbf{X}(\xb,t)= \ub(\textbf{X}(\xb,t),t)&\quad \text{in $\Omega^+(0)$}\\
\Omega^+(t)=\textbf{X}(\Omega^+(0),t), &\label{seis}
\end{align}
where $\ub^\pm$ is the restriction of the velocity to the interface, $\Gamma(t)=  \pa\Omega^+(t) \cap \pa\Omega^-(t)$, between both fluids , $\nb$ is the normal unit vector to $\Gamma(t)$ pointing out of $\Omega^+$, $\tb$ is a unit tangential vector to $\Gamma(t)$, $\Omega^\pm$ is the domain occupied by the fluid with density $\rho^\pm$ and therefore $\Omega^-=\R^2\setminus \overline{\Omega^+}$. Without any loss of generality we will take from now on $g=\nu=\kappa=1$.

The same system of  equations governs an interface separating two
fluids trapped between two closely spaced parallel vertical plates (a "Helle Shaw cell"). See \cite{safman}.

We also assume that $\Omega^+(0)$ is open and  simple connected, that there exist a constant $C$ such that $\{\xb=(x_1,\,x_2)\in\R^2\,:\, x_2<C\}\subset \Omega^+(0)$ (the fluid with density $\rho^+$ is below) and that the interface $\Gamma(0)$ is asymptotically flat  at  infinity with $\lim_{x_1\to -\infty}x_2=\lim_{x_1\to \infty}x_2=0$ for $\xb\in \Gamma(0)$.  This type of initial data will be called {\it of Muskat type}.

In this situation one can find an equation for the interface between the two fluids. Indeed, if we take the parametrization
$$\Gamma(t)=\{\zb(s,t)=(z_1(s,t),\, z_2(s,t))\in \R^2\},$$
the curve $\zb(s,t)$ must satisfy from \eqref{uno},..., \eqref{seis} (see \cite{RTB} and \cite{dp})
\begin{align}\label{muskat}
\pa_t\zb(s,t)=\frac{\rho^+-\rho^-}{2\pi}P.V.\int_{-\infty}^\infty\frac{z_1(s,t)-z_1(s',t)}{|\zb(s,t)-\zb(s',t)|^2}(\pa_s \zb(s,t)-\pa_s \zb(s',t))ds',
\end{align}
where $P.V.$ denotes the principal value integral.  At the same time the solutions of the Muskat equation \eqref{muskat} provide weak solutions of the IPM system.

The behaviour of the equation \eqref{muskat} strongly depends on the order of the densities $\rho^+$ and $\rho^-$. The problem is locally well posed in Sobolev spaces, $H^3$ (see \cite{dp}), if the interface is a graph and $\rho^+>\rho^-$, i.e., in the stable regime (see also \cite{vicol} and \cite{granero} for improvements of the regularity). Otherwise, we are in the unstable regime and the problem is ill-posed in $H^4$. This is a consequence of the instant analyticity proved in \cite{RTB} in the stable case (see also \cite{dp} for ill-posedness in $H^3$ for an small initial data).

 This contrast between the stable and unstable case is easy to believe  since  $\textbf{F}(s,t)=\pa^{4}_s \zb(s,t)$ satisfies that
\begin{align*}
\pa_t \textbf{F}=-\sigma(s,t)\Lambda \textbf{F}+ a(s,t)\pa_s \textbf{F}+ \textbf{R}(s,t),
\end{align*}
where $\Lambda=(-\Delta)^\frac{1}{2}$ , $a(s,t)$ and $\textbf{R}$ are lower order terms and the Rayleigh-Taylor function $\sigma(s,t)$ reads
\begin{align*}
\sigma(s,t)=(\rho^+-\rho^-)\frac{\pa_s z_1(s,t)}{|\pa_s \zb(s,t)|^2}.
\end{align*}
A quick analogy with the heat equation indicates that for  $\sigma(s,t)$  positive everywhere  the problem is well-possed (we are in the stable case). If $\sigma(s,t)$ is negative
the equation resembles a  backwards heat equation in this region and therefore  instabilities arise.

However,
in the present paper, we show that  there exists  weak solutions to the IPM system starting with an initial data of Muskat type in the fully unstable regime, i.e., $\rho^+<\rho^-$ and $\pa_s z_1(s,0)>0$ everywhere. The initial interface will have Sobolev regularity and in addition these solutions will have the following structure: there will exist domains $\Omega^{\pm}(t)$ where the density will be equal to $\rho^\pm$ and a  mixing domain $\Omega_{mix}(t)$ such that for every space-time ball contained in the mixing area
the density will take both  values $\rho^+$ and $\rho^-$.  We will call these solutions mixing solutions (see the forthcoming definition~\ref{mixing}). In figures \ref{figuresin} and \ref{mixingsolution} we present the main features of this kind of solutions.

\begin{figure}[ht]\label{figuresin}
\centering
\includegraphics[width=10cm,height=8cm]{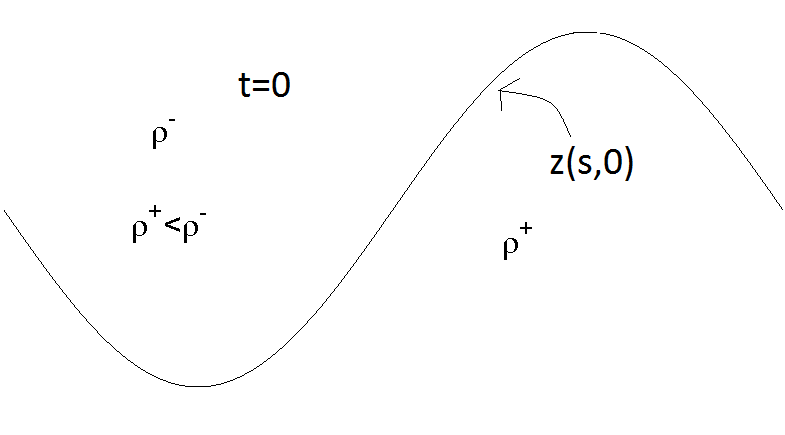}
\caption{A Muskat type initial data in fully unstable regime.}
\end{figure}

\begin{figure}[ht]
\centering
\includegraphics[width=10cm,height=8cm]{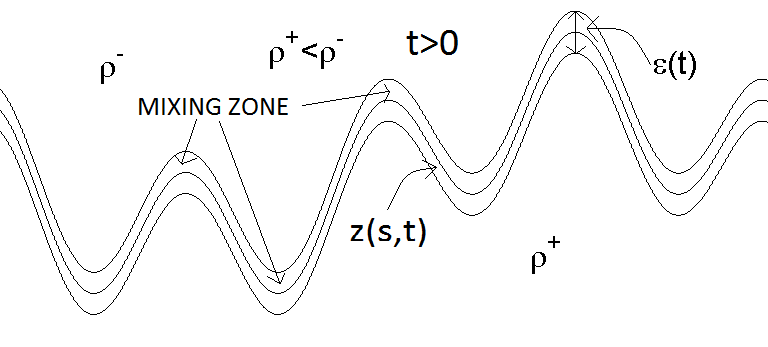}
\caption{A mixing solution a time $t>0$ starting in the configuration of fig. \ref{figuresin}.}
\label{mixingsolution}
\end{figure}

\begin{thm}\label{thmmain}Let $\Gamma(0)=\{ (x,f_0(x))\in \R^2\}$ with $f_0\in H^5$. Let us suppose that $\rho^+<\rho^-$. Then there exist infinitely many "mixing solutions" starting with the inital data of Muskat type given by $\Gamma(0)$ (in the fully unstable regime) for the IPM system.
\end{thm}
\begin{rem}
The existence
of such mixing solutions was predicted   by Otto in \cite{Otto}.  In this pioneering paper, Otto  discretizes the problem and present a relaxation in the context of Wasserstein metric,
which yields the existence of a "relaxed" solution in the case of a flat initial interface. It is a very interesting question whether it is possible to extend this approach to cover Theorem~\ref{thmmain}.
We would like to emphasize that the initial interface has Sobolev regularity, thus the Muskat problem is ill-possed in the Hadamard sense (see for example \cite{dp}). Therefore the creation of a mixing zone provides a mechanism to solve the IPM system in a situation where solutions of Muskat are not known.
\end{rem}

\begin{rem}Notice that these "mixing solutions" do not change the values that the density initially takes and that in any space-time ball $B \subset \Omega_{mix}(t)\times (0,T)$, $\rho$ takes both values, i.e there is total mixing. In fact,  a more refined version of convex integration recently  presented in the recent manuscript \cite{CastroFaracoMengual}, it is proved that there is mixing in space balls.
\end{rem}

The method of the proof is based on the adaptation of the method of convex integration
for the incompressible Euler equation in the Tartar framework developed recently by De Lellis and Sz\'ekelyhidi
(see \cite{tres}, \cite{cuatro}, \cite{cinco}, \cite{seis}, \cite{siete}, \cite{ocho}, \cite{nueve}, \cite{diez}, \cite{dieciseis} and \cite{vortex}  for the incompressible Euler and  for another equations \cite{doschi}, \cite{treschi}, \cite{cuatrochi}, \cite{seischi} and \cite{kim}).

Very briefly,  the version of convex integration used initially by De Lellis and Sz\'ekelyhidi  understands a nonlinear PDE,
$F(\rho,u)=0 $ as a combination of a linear system $L(\rho,u,q)=0$ and a pointwise constraint $(\rho,u,q) \in K$ where $K$ is a convenient set of states and $q$ is an artificial new variable. Then $L$ gives rises to a wave cone $\Lambda$ and the geometry of the
$\Lambda$ hull of $K$, $K^\Lambda$, rules whether the convex integration method will yield
solutions. An h-principle holds in this context: if for a given initial data there exists an evolution which belongs to $K^\Lambda$, called a subsolution, then one finds infinitely many weak  solutions.

 For the case of the IPM system, in \cite{CFG}, the authors initiated this analysis and used a version of the convex integration method which avoids the computation of $\Lambda$ hulls based on T4 configurations, key in other applications  of convex integration, e.g.  to the (lack of) regularity of elliptic systems \cite{MullerSverakAnnals, KirchheimMullerSverak,Kirchheimthesis}. Keeping the discussion imprecise, their criteria amounts to say that $(0,0)$ must be in the convex hull of $\Lambda \cap K$ in  a stable way.  Shvydkoy extended this approach  to a general family of active scalars, where the velocity is an even singular integral operator, in  \cite{roman}.  Recently, in \cite{VI}, Isett and Vicol using more subtle versions of convex integration  show the existence of weak solution for IPM with $C^\alpha-$regularity.  All of these  solutions, change the range of the modulus of the density. We remark that the solutions in theorem~\ref{thmmain} do not change the values of the density.

  Sz\'ekelyhidi refined the result of \cite{CFG} in \cite{laslo} computing explicitly the $\Lambda$-hull for the case of IPM. Notice that this increases the number of subsolutions (and thus the solutions available). In fact, Sz\'ekelyhidi showed that for the case of a flat interface in the unstable regime there exists a subsolution and thus  proved theorem \ref{thmmain} in this case.

\vspace{0.5cm}

The main contribution of this work is  a new way  to construct  such subsolutions, inspired by previous studies in contour dynamics, which  we believe of interest in related problems. Let us describe it briefly. The mixing zone (that is where the subsolution is not a solution) will be
a neighborhood of size $2\ep(x,t)$ of a suitable curve $(x,f(x,t))$ evolving in time according to a suitable  evolution equation. We call this curve
the {\it pseudointerface}.

Namely, if $\xb(x,\lambda)=(x,\lambda+f(x,t))$ we will
declare the mixing zone $\Omega_{mix}$ to be
$$\Omega_{mix}=\{ \xb\in \R^2 \,:\, \xb=\xb(x,\lambda)\, \quad\text{for}\quad (x,\lambda)\in (-\infty,\infty)\times (-\ep(x,t),\ep(x,t))\}.$$
Inside the mixing zone, the density of our subsolution will be simply  $\rho=\frac{\lambda}{\ep(x,t)}$.

Notice that the width of the mixing zone is variable, and it will grow linearly in time as $\ep(x,t)=c(x,t) t$, where $c(x,t)$, $1\leq  c(x,t) <2$,  is essentially  an arbitrary smooth function (technical assumptions will be made in theorem \ref{existencialocal}).

The case of constant  $c(x,t)=c$ is technically easier but we have preferred to deal with the variable growth case as it is more useful for further application and it shows the flexibility of the method.

Let us observe, that at the boundary of the mixing zone, the subsolution must become a solution ($|\rho|=1$). Our choice of the  subsolution
 imposes  that $f(x,t)$ must satisfy the  following non linear and non local equation,
\begin{align}
\pa_t f(x,t)= & \mathcal{M}u(x,t)\nonumber\\
f(x,0)= & f_0(x),\label{zmedio}
\end{align}
where
\begin{align*}
\mathcal{M}u(x,t)=-\frac{1}{2}\int_{-1}^1\frac{1}{\pi}\int_{-\infty}^\infty \frac{1}{2}\int_{-1}^1\frac{(x-y)\left(\pa_xf(x)-\pa_yf(y)+\lambda\ep(x)-\lambda'\ep(y)\right)}{(x-y)^2+(\ep(x)\lambda-\ep(y)\lambda'+f(x)-f(y))^2}d\lambda' dy d\lambda.
\end{align*}
Here $\mathcal{M}u$ can be understood as a suitable double
average of the velocity in the Muskat case.

 It turns out that it is possible but rather difficult to
obtain uniform estimates on $t$ for the operator $\mathcal{M}u$ in order to obtain existence for this system.  The situation is reminiscent to that of the Muskat problem but it is different as, on one hand, the kernel is not so singular but, on the other hand, we need to obtain estimates which are independent of $\ep$ (notice that for $\ep=0$ the problem is ill-posed). The first  difficulty is to quasi-linearize the operator $\mathcal{M}u$. This quasi-linearization is inspired by that one for the classical Muskat equation \ref{muskat} (see for example \cite{dp}). However, even in the case of constant $\ep$, some new difficulties arise and to deal with them we need to use different tools e.g., pseudodifferential theory.
The presence of variable width $\ep(x,t)$ introduces additional technical complications. Since the proof is long and delicate but the result is believable we postpone the proof to the appendix \ref{ordenmasaltosection} and \ref{apendice4} where we have introduced ad hoc notation which should make the proofs nice to follow.

In turn,  the needed a priori estimate boils down to understanding the evolution of the following equation for $F(x,t)=\pa^5_x f(x,t)$
\begin{equation} \label{Evolution}\partial_t F(x,t)=\int K(x,y,t) \partial_x F(x-y,t)dy+a(x,t)  \partial_x F(y,t)+G(x,t),\end{equation}
for a suitable kernel $K:\R\times \mathbb{R}\times \R^+ \to \R$, where $G(x,t)$  is a lower order term and $a(x,t)$  are functions with a lower number of derivatives.  The important fact in equation \eqref{Evolution} is that the kernel $K$ is order zero at time $t=0$, and yields a $(-\Delta)^\frac{1}{2}-$term with the wrong sign.  However $K$ is of $(-1)$-order  for any $t>0$ and yields a bounded term but with a blowing up norm  $\sim\frac{1}{t}$.

At the beginning of section \ref{apriori} we explain with a toy model, where the $x-$dependence of $K$ is frozen, that this behaviour forces a loss of at least one derivative with respect to the initial data. Semiclassical analysis \cite{Zworski} studies how the behaviour of smooth symbols $p(x,\hbar \xi)$ is like that of Fourier multipliers up to factors of
$\hbar$, $\hbar$ standing for the Planck constant. Our symbols $p(x,t,t\xi)$ can be interpreted as semiclassical with the time playing the role of $\hbar$  but they are not smooth. Thus, in order to deal with the full system, we produce a basic calculus of semiclassical type of pseudodifferential operators with limited smoothness, e.g.,  composition of such symbols or a suitable G\aa rding inequality. The results are pretty general and perhaps of its own interest.

Once that we define such a pseudointerface and the corresponding mixing zone,  we can
find the corresponding density $\rho$ and velocity  $\ub$ and show that they belong
to the suitable $\Lambda$ hull for small time, yielding then a subsolution.  Given the subsolution,  convex integration applies to create infinitely many  weak solutions,  though an additional observation is needed to obtain the mixing property (see section \ref{Hprinciple}).

The method of the proof seems robust to prove existence of weak solutions in a number of free boundary problems in an unstable regime.  For further recents developments of this circle of ideas,  see e.g \cite{ACFpreprint}, \cite{Mengualpreprint}, \cite{MengualSzekelyhidipreprint}, \cite{NoisseteSzekelyhidipreprint}. As it was remarked by Otto and Sz\'ekelyhidi (\cite{Otto} and \cite{laslo})
the underlying subsolution seems to capture  relevant observed properties of the solution as it is the growing rate of the mixing zone, the fingering phenomena (see the numerics in \cite{Castro})  or the volume proportion of the mixing (This has been recently quantified in \cite{CastroFaracoMengual}).

 It seems to us that the creation of a mixing zone in the lines
of this work, might end up in to a canonical way of turning ill-posed problems into solvable ones, at the price of loosing uniqueness at least at the microscopic level  (this line of thought has been already expressed in \cite{Otto} and
\cite{laslo}).   We emphasize that subsolutions as such are also highly non unique (e.g. see  the recent \cite{FS}  for an elegant proof of existence of subsolutions with piecewise constant densities).  In the case of the flat interface the relaxation solution obtained by Otto can be characterized as the unique entropy solution \cite{Otto} of a  concrete scalar conservation law, the one who linearly interpolates between the heavier and lighter fluid  and as the subsolution who maximizes the  speed of growth of the  mixing zone (\cite{laslo}).  Perhaps the most challenging open question after this work to obtain such nicely agreeing selection criteria for subsolutions  in the case of an arbitrary interface.
At the end of the paper we add a remark showing that surprisingly the mixing solutions are also present  in the stable regime in the case of straight interfaces except in the horizontal case. Let us remark that this have been extended to not flat interfaces in \cite{FS}

The paper is organized as follows: In section~\ref{mixingsubsolutions} we introduce the rigorous definition of mixing solutions and subsolution. In section~\ref{Hprinciple} we explain how the convex integration theory allow us to obtain a mixing solution from a subsolution. Section~\ref{subsolutionsection} is divided in two parts. In the first part, subsection 4.1, we construct a subsolution for the IPM system assuming the existence of the pseudointerface, ie.  solution for the equation \eqref{zmedio}. In the second part, section 4.2, we will show the existence of solutions for the equation \eqref{zmedio}. As discussed before, the proof requires some pseudodifferential  estimates
for non smooth symbols which might be of its own interest so we have gathered them in section~\ref{Semiclassical}. First we present the results which are general and then
those more related to our specific symbols, though it would not be difficult to extrapolate general theorems from  the later, as in the case of G\aa rding inequality.

In section~\ref{stable} we show how to construct mixing solutions in the stable regime. Finally in the appendix we prove the quasilinearization estimates as well as compute the symbols and their estimates.
\subsection{Notation}\label{notation}
We close the introduction by fixing some notation as it varies quite a lot in the literature. When no confusion arises  will use $L^2,H^k$ to
denote $L^2(\mathbb{R}), H^k(\mathbb{R})$ and $S$ denotes the Schwarz class.
 Given a symbol $p(x,\xi)$  we define a pseudodifferential operator $\Op(p)$
by
\[ \Op(p)(f)(x)= \int e^{2\pi i x\xi} p(x,\xi) \hat{f}(\xi) d\xi, \]
for $f\in S$.

In the case that the symbol $p=p(\xi)$,  depends only on the frequency variable, i.e., $p$ is a Fourier multiplier, we denote the operator by $P$ (the capital letter).
We will use the following  notation  to estimate commutators, correlation of differential operators  and the skew symmetric part of an operator.

\[\begin{aligned}
&[\Op(p_1),\Op(p_2)]=\Op(p_1)\circ \Op(p_2)-\Op(p_2)\circ \Op(p_1), \\
& \mathfrak{C}(p_1,p_2)=\Op(p_1)\circ \Op(p_2)-\Op(p_1\cdot p_2),
\\ &\Op(p)^{skew}=\Op(p)-\Op(p)^T,\end{aligned}\]
where $\Op(p)^T$ is the adjoint respect to the standard $L^2$ product.  For smoothness of the symbols we use the norms
\[ \|p\|_{\alpha,\beta}=\sup_{x,\xi,\alpha'\le \alpha,\beta' \le \beta} |\partial_x^{\alpha'}\partial_\xi^{\beta'}p(x,\xi)|, \]
where the derivatives are taken in the distributional sense.
Finally we will say that $p(x,\xi) \in \mathcal{S}_{\alpha,\beta}$,  if
\[ \|p\|_{\alpha,\beta} <\infty. \]

The symbols $|||f|||$ and $\lar$ will denote some polynomial function evaluated in $||f||_{H^4}$
and $||A||_{H^3}$ and we recall that $$A=\pa_x f$$
along the paper. In particular  both $|||f|||$ and $\lar$ will depend on $c$ as well  but we will not make such dependence explicit
as it is harmless for the apriori estimates (for a $c(x,t)$ as in the statement of theorem \ref{existencialocal}).

\section{The concepts of mixing solution and subsolution}\label{mixingsubsolutions}

Following \cite{laslo} we rigorously define the concept of "mixing solution" in the statement of theorem~\ref{thmmain}.  We would like our solutions to mix in every ball of the domain and thus we incorporate this into the
definition. Firstly, since we are working in unbounded domains, we give a definition of weak solution in which we prescribed the behaviour of the density  at $\infty$. In the following     $\mathcal{R}_i$, with $i=1,2$ are   the Riesz transform and $\mathcal{BS}$ is the Biot-Savart convolution. Recall that for a smooth function $f$ these operators  admit the kernel representations,
\begin{align*}
\mathcal{R}_if(\xb)=\frac{1}{2\pi}P.V.\int_{\R^2}\frac{x_i-y_i}{|\xb-\yb|^3}f(\yb)d\yb, &&
\mathcal{BS}f(\xb)=\frac{1}{2\pi}\int_{\R^2}\frac{(\xb-\yb)^\perp}{|\xb-\yb|^2}f(\yb)d\yb.
\end{align*}

\begin{defi}\label{weaksolution} Let  $T>0$ and  $\rho_0 \in  L^\infty(\R^2).$
The density $\rho(\xb,t) \in L^\infty(\R^2\times [0,T])$ and the velocity $\ub(\xb,t)  \in L^\infty(\R^2\times [0,T])$ are a weak solution of the IPM system with initial data $\rho_0$ and  if and only if  the weak equation
\begin{align*}
\int_{0}^T\int_{\R^2} \rho\left(\pa_t \varphi +\ub\cdot\nabla\varphi\right)dxdt=\int_{\R^2} \varphi(\xb,0)\rho_0(\xb)dx
\end{align*}
holds for all $\varphi \in C_c^\infty([0,T)\times\R^2)$,
and
\begin{align}\label{biotsavartmodi}
\ub(\xb)= \mathcal{BS}(-\pa_{x_1}\rho).
\end{align}
 \end{defi}

Notice that we have interpreted  the incompressibility of the velocity field and Darcy's law  with \eqref{biotsavartmodi}. In fact, for  $\rho\in C^\infty_c(\R^2), $ the equations
\begin{equation}\label{divcurl}
\begin{aligned}
\nabla\cdot \ub=&0\\
\nabla^\perp \cdot \ub= & -\pa_{x_1}\rho.
\end{aligned}
\end{equation}
together with the condition that  $\ub$ vanishes at infinity (the boundary condition) are equivalent to
\begin{align*}
\ub(\xb)=\mathcal{BS}(-\pa_{x_1}\rho)= \left(\mathcal{R}_2\mathcal{R}_1\rho,-\mathcal{R}_1\mathcal{R}_1\rho\right).
\end{align*}
 Thus, they are  consistent with  definition \ref{weaksolution}.
Definition \ref{weaksolution} extends the concept of solution of the system \eqref{divcurl} plus vanishing boundary condition for densities which do not necessarily vanish at infinity.
Notice that incompressibility and Darcy's law are automatically satisfied by our solution in the weak sense. That is,
\begin{align*}
\int_{\R^2}\ub\cdot \nabla \varphi d\xb =&0\\
\int_{\R^2}\ub\cdot \nabla^\perp \varphi d\xb =& -\int_{\R^2} \rho \pa_{x_1}\varphi d\xb,
\end{align*}
for all $\varphi\in C^\infty_c(\R^2)$.
\begin{defi}\label{mixing}
The density $\rho(\xb,t)$ and the velocity $\ub(\xb,t)$ are a "mixing solution" of the IPM system if they are a weak solution and also there exist, for every $t\in [0,T]$, open simply connected  domains $\Omega^{\pm}(t)$ and $\Omega_{mix}(t)$ with
$\overline{\Omega^+}\cup\overline{\Omega^-}\cup\Omega_{mix}=\R^2$  such that, for almost every $(x,t)\in \R^2\times [0,T]$, the following holds:

 \begin{align*}
\rho(x,t)=\left\{\begin{array}{cc}\rho^{\pm} & \text{in $\Omega^\pm(t)$}\\
 (\rho-\rho^+)(\rho-\rho^-)=0& \text{in $\Omega_{mix}(t)$}\end{array}\right..
\end{align*}

For every  $r>0, x \in \R^2, 0<t<T$ $B((x,t),\,r) \subset \cup_{0<t<T} \Omega_{mix}(t)$ it holds that
\[ \int_B (\rho-\rho^+) \int_B(\rho-\rho^-) \neq 0. \]
\end{defi}

For sake of simplicity and without any loss of generality we will fix the values of the density to be
\begin{align}\rho^{\pm}=\mp 1.\label{densi}\end{align}
The concept of subsolution is rooted in the Tartar framework understanding a non linear PDE as a linear PDE plus a non linear constraint.  In our context the linear constraint is given by
\[ K=\{(\rho,\ub,\mb) \in \R \times \R^2 \times \R^2:  \mb=\rho \ub, |\rho|=1 \}.\]
As observed by Sz\'ekelyhidi the set $K$ contains unbounded velocities which is slightly
unpleasant. Thus for a given $M>1$ we define
\[ K_M=\{(\rho,\ub,\mb) \in \R \times \R^2 \times \R^2:  \mb=\rho \ub, |\rho|=1, |\ub|\le M \}.\]

 Subsolutions arise as a relaxation of the nonlinear constraint. In the framework of the IPM system the relaxation is given by the mixing hull, the $\Lambda$ lamination hull for the associated wave cone  $\Lambda$ (see \cite{CFG,laslo} for a description of $\Lambda$). In \cite{laslo}, the author computed the laminations hulls of $K$ and $K_M$. We take them as definitions.

\begin{defi} We defined the mixing hulls for IPM by

\begin{align}\label{convex1}
K^{\Lambda}=&\left\{ (\rho,\ub,\mb) \in \R \times \R^2 \times \R^2\, :\,
 \left|\mb-\rho\ub+\frac{1}{2}\left(0,1-\rho^2\right)\right|  <  \left(\frac{1}{2}\left(1-\rho^2\right)\right)  \right\}.
\end{align}

For a given $M>1$, the $M$-mixing hull $K_M^{\Lambda}$ are the elements in
$K^\Lambda$ which additionally satisfy that

\begin{align}
&\left|2\ub+(0,\rho)\right|^2<  M^2-(1-\rho^2)\label{convex2}\\
& \left|\mb-\ub-\frac{1}{2}(0,1-\rho)\right|< \frac{M}{2} (1-\rho)\label{convex3}\\
&\left|\mb+\ub+\frac{1}{2}(0,1+\rho)\right|<  \frac{M}{2}(1+\rho) \label{convex4}.
\end{align}
\end{defi}

\begin{rem} Let us clarify the differences between our notation and the notation in \cite{laslo}. We are using same notation as in \cite{laslo} in section 4, but with $v$  there replaced by $\ub$  here. The concept of M-subsolution arises in section 2, proposition 2.5 in \cite{laslo}. To translate this proposition to our language one has to replace $u$ there by $2\ub+(0,\rho)$ and $m$ there by $\mb+\frac{1}{2}(0,1)$ (notice that in \cite{laslo} $m$, in section 2, pass to $m+\frac{1}{2}(0,1)$ in section 4).
\end{rem}

\begin{defi}\label{subsolution} Let $M>1$ and $T>0$.
We will say that $(\rho,\ub,\mb)\in L^\infty(\R^2\times [0,T])\times L^\infty(\R^2\times [0,T])\times L^\infty(\R^2\times [0,T])$,  is a M-subsolution of the IPM system if there exist open simply connected  domains $\Omega^{\pm}(t)$ and $\Omega_{mix}(t)$ with
$\overline{\Omega^+}\cup\overline{\Omega^-}\cup\Omega_{mix}=\R^2$ and such that the following holds:
\begin{itemize}
\item [(No mixing)]The density satisfies
\begin{align*}
\rho(\xb,t)=\mp 1\quad \text{in $\Omega^\pm(t)$}.
\end{align*}

\item [(linear constraint)]In $\R^2\times [0,T]$ $(\rho,\ub,\mb)$ satisfy the equations
\begin{align}
\pa_t\rho+ \nabla \cdot \mb =&0\nonumber \\
\rho(x,0)= & \rho_0 \nonumber \\ \label{biotsavart}
\ub(\xb)=&\mathcal{BS}(-\pa_{x_1}\rho)\equiv\frac{1}{2\pi}\int_{\R^2}\frac{(\xb-\yb)^\perp}{|\xb-\yb|^2}(-\pa_{y_1}\rho(\yb))d\yb,
\end{align}
in a weak sense.
\item [(Relaxation)]   $(\rho,\ub,\mb) \in K^{\Lambda}_M$  in $\Omega_{mix}(t)\times (0,T)$ and $(\rho,\ub,\mb) \in \overline{K}^{\Lambda}_M$  in $\R^2\times (0,T)$.
\item [(Continuity)] $(\rho,\ub,\mb)$ is continuous in $\Omega_{mix}(t)\times (0,T)$.
 \end{itemize}

\end{defi}

\begin{rem}Along the text we will typically speak about subsolution (rather than M-subsolution) and we only  make explicit the constant $M$ when it is needed.
\end{rem}

\section{H-principle: Subsolutions yield weak solutions}\label{Hprinciple}
In this section we follow  \cite{laslo} to find that to prove theorem \ref{thmmain} is enough to show the existence of a M-subsolution, for some $M>1$, $(\overline{\rho},\overline{\ub}, \overline{\mb})$.  Since $L^\infty(\mathbb{R}^2) \subset L^2(d\mu)$
with $d\mu = \frac{dx}{(1+|x|)^3}$,we will work with $L^2(d\tilde{\mu})$,  where $d\tilde{\mu}=d\mu dt$ as the auxiliar space.

 Associated to a M-subsolution $(\overline{\rho},\overline{\ub}, \overline{\mb})$ in $[0,T]$,
we define a set $X_0$.
\begin{align*}
&X_0=\left\{ (\rho,\ub,\mb)\in L^\infty(\R^2\times [0,T]\times L^\infty(\R^2\times[0,T]) \times L^\infty (\R^2\times [0,T]) \quad :\right.\\
& (\rho,\ub,\mb)=(\overline{\rho},\overline{\ub}, \overline{\mb})\quad \text{a. e. in $\R^2\setminus \overline{\Omega}_{mix}$},\\
&\left. \text{and $(\rho,\ub,\mb)$ is a subsolution} \right\}.
\end{align*}

This set is not empty since $(\overline{\rho},\overline{\ub}, \overline{\mb}) \in X_0$.
\begin{lemma}Let $(\overline{\rho},\overline{\ub}, \overline{\mb})$ be a M-subsolution. Then the space $X_0$ is bounded in $L^2(d\tilde{\mu})$.
\end{lemma}
\begin{proof} Let $(\rho,\ub,\mb)\in X_0$. Then $||\rho||_{L^\infty}\le1$ and $||u||_{L^\infty}\leq C(M)$, so that for a fixed time  $||\rho||_{L^2(d\mu)}$, $||u||_{L^2(d\mu)}\leq C(M)$.
 Similarly
 \begin{align*}
 ||\mb||_{L^2(d\mu)}\leq & \left|\left|\mb-\rho\ub+\frac{1}{2}\left(0,1-\rho^2\right)\right|\right|_{L^2(d\mu)}+
 \left|\left|\rho\ub-\frac{1}{2}\left(0,1-\rho^2\right)\right|\right|_{L^2(d\mu)}.
 \end{align*}
 Thus $||\mb||_{L^2(d\mu)}$ is bounded thanks to \eqref{convex1} and to  $||\rho||_{L^2(d\mu)}$, $||\ub||_{L^2(d\mu)}\leq C(M)$.  The claim follows by integrating respect to time in
 $[0,T]$.

\end{proof}
 Since $X_0$ is bounded in $L^2(d\tilde{\mu})$ and the weak topology of this space is metrizable, we can consider the space $X$ given by  closure of $X_0$ under this metric.

 We will prove the following theorem,

\begin{thm} \label{mixing residual}If $X_0$ is not empty the set of mixing  solutions of
 IPM with $\rho_0$ as initial data is residual in $X$. Here $\rho_0$ is the subsolution at time $t=0$.\end{thm}

The general framework of convex integration applies  easily to our setting. For the sake of simplicity  we will follow  the Appendix from \cite{laslo} with an slight modification.
We consider the unbounded domain $\mathbb{R}^3$ ($\R^2$ in space and $\R$ in time), $z:\Omega \to \mathbb{R}^5$  and  a bounded set $K \subset  \mathbb{R}^5$ such that

\begin{equation}\label{LS}
\sum_{i=1}^d A_i \partial_i z=0,
\end{equation}

\begin{equation}
z \in K.
\end{equation}

Assumptions:

\begin{itemize}

\item [H1] The wave cone.  There exists a closed  cone $\Lambda \subset \mathbb{R}^5$ such that for every $\overline{z}  \in \Lambda$ and for every ball $B\in \R^3$ there exists a sequence $z_j \in C_c^\infty
(B,\mathbb{R}^5)$ such that
\begin{itemize}
\item [i)]$\textrm{dist}(z_j,[-\overline{z},\overline{z}])\to 0$ uniformly,
\item [ii)]$z_j \to 0$ weakly 0 in $L^2(d\tilde{\mu})$ weakly,
\item [iii)] $\int |z_j|^2 d\tilde{\mu}\ge \frac12 |\overline{z}|^2$.
\end{itemize}

\item [H2] The $\Lambda$ convex hull. There exist an open set $U$ with $U\cap K=\emptyset$ and a continuous convex and increasing nonnegative function $\phi$ with $\phi(0)=0$ that for every $z \in U$   $z+t\overline{z} \in U$ for $|t|\le \phi(dist(z,K))$
\item [H3] Subsolutions. There exists a set $X_0 \subset L^2(d\tilde{\mu})$ that  is a bounded subset of  $L^2(d\tilde{\mu})$ which is perturbable in a fixed subdomain $\mathcal{U} \subset \Omega$
such that any $z \in X_0$ that
satisfies $z(y) \in U$ and if $w_j \in C^\infty_c(\mathcal{U},\R^5)$ is the approximating sequence from $[H1]$ and
$z+w_j \in U$  then $z+w_j \in X_0$.
\end{itemize}
In the case of the IPM equation with the constraints $|\rho|=1,|u| \le M$ both $\Lambda$, $K_M$ and $K_M^\Lambda$  has been extensively studied in \cite{CFG,laslo}. We take $\mathcal{U}=\Omega_{mix}(t)\times (0,T)$. The property [H2] for $K_M^\Lambda$ was proved in \cite[Proposition 3.3]{laslo}. For the property [H1]  we use the sequence $z_j$ as constructed for example in \cite[Lemma 3.3]{CFG}.
Our  Property [H1i)] is stated in the first property stated in that lemma. For property [H1ii)] notice  that  we know  from \cite[Lemma 3.3]{CFG} that $z_j \to 0$  weakly  star topology of $L^\infty$. However, $z_j$  is uniformly bounded in $L^\infty$  and compactly supported and thus uniformly bounded in $L^2(d\tilde{\mu})$. Thus the weak star convergence  implies also weak star convergence in  $L^2(d\tilde{\mu})$. Our property  [H1iii)] requires some work as $\mu$ does not scale uniformly.
However as proved  for example in \cite[Lemma 3.3]{CFG}, in addition to the properties listed in \cite[H1]{laslo}  it holds  that for a $\Lambda$ segment $\overline{z}$ the approximating
sequence satisfies also that,

\[\lim_{j \to \infty} | (x,t) \in B:   |z_j(x,t)| \neq \pm \overline{z}| =0,\]
and by absolute continuity  it holds that
\[\lim_{j \to \infty} \mu\{ (x,t) \in B:   |z_j(x,t)| \neq \pm \overline{z}\} =0.\]

Thus by choosing $j$ large enough $iii)$ also holds.

We skip the proof of the following lemma as it is identical to \cite[Lemma 5.2]{laslo}

\begin{lemma}\label{perturbation}Let $z \in X_0$ with $\int_{\Omega_{mix}(t)\times[0,T]} F(z((x,t)))d\tilde{\mu}\ge \ep>0$. For all $\eta>0$ there exists $\tilde{z} \in X_0$ with $d_{X}(z,\tilde{z})\le \eta$
and
$$ \int_{\Omega_{mix}(t)\times[0,T]} |z-\tilde{z}|^2 d\tilde{\mu} \ge \delta. $$
Here $\delta=\delta(\ep)$.
\end{lemma}

{\it Proof of theorem~\ref{mixing residual}.}

Firstly, as in the  proof of  \cite [theorem 5.1]{laslo} lemma~\ref{perturbation} implies that the set of bounded solutions to IPM is residual in $X$.
The proof works in the same way since due to the fact that $\mu(\mathbb{R}^2)<\infty$,  convolutions with a standard mollification kernel are continuous from
$L^2(d\tilde{\mu} ,w)$ to $L^2(d\tilde{\mu} )$ and thus the Identity is a Baire one map, with  a residual set of points of continuity.  That is  the set of $(\rho,\ub,\mb) \in X$ which belong to
$K_M$ a.e. $(x,t) \in \mathbb{R}^2 \times [0,T]$ is residual in $X$. This is precisely the set
of weak solutions to IPM with the Muskat initial data.

 It remains  to show the mixing property:

Choose  $B((x,t),r) \subset \cup_{0<t<T} \Omega_{\text{mix}}(t)$. Declare
\[X_{B,\pm 1}=\{ (\rho, u, m) \in X: \int_B (\pm 1-\rho)=0\}. \]
Then  $X_{B,\pm 1} \subset X$ is closed by the definition of weak convergence and since $X_{B,\pm 1}\cap X_0 =\emptyset$ (for states in $X_0$, $|\rho|<1$)  and $X_{B,\pm 1} \subset \overline{X_0}$. Thus, $X_{B,\pm 1}$  has empty interior. Therefore  $X\setminus  X_{B,\pm 1}$ is residual.
Since intersection of residual sets is residual, it follows that
\[ \{ X\setminus \cup_i X_{B_i,\pm 1}: B_i=B(x_i,t_i,r_i) \subset \cup_{0<t<T} \Omega(t),
x_i \in \mathbb{Q}^2,t_i \in \mathbb{Q},r_i \in Q\}\]
with $\mathbb{Q}$ the rationals is residual. By density of rationals  elements in  $X\setminus \cup_i X_{B_i,\pm 1}$ satisfy the mixing property and thus the set of mixing solutions is residual in $X$ with respect to the weak topology.

\hfill $\Box$

\begin{rem} We introduce the measure $\mu$ to deal with the unboundedness of the domain.  However we could have followed instead \cite{siete} and consider for capital $N \nearrow
\infty$  $I_N: X \mapsto \mathbb{R}$ defined by
$I_N: \int_{B(0,N)\times[0,T]} (|\rho|^2-1) dxdt$. By convexity of the $L^2$ norm it follows that  $I_N$ is  lower semicontinuous respect to the weak star topology of $L^\infty(X)$. Thus it is a Baire one map with a residual set of points of continuity.  By our lemma~\ref{perturbation}
if $z$ is  a point of continuity  of $I_N$ in $X$ $I_N(z)=0$.
 Since elements of $X$ such that  $\rho(x,t)=1$ correspond to weak solutions to IPM and intersection of residual sets is residual the theorem follows.
\end{rem}

\begin{rem} The proof presented above only yields weak solutions to the IPM system such that $|\rho(x,t)|=1$ for a.e. $t \in [0,T]$. However (see the proof of \cite[Lemma 3.3]{CFG})  for every $\overline{z}=(\overline{\rho},\overline{u},\overline{m}) \in \Lambda$ with $\overline{\rho} \neq 0$ there exists
$(\xi , \xi_t) \in \mathbb{R}_x^2\times \mathbb{R}_t, \xi \neq 0$ such that
\[\begin{aligned}
&D^2(h( (\xi,\xi_t) \cdot (x,t)))= h''( (\xi,\xi_t) \cdot (x,t))\left(\begin{array}{cc} \overline{\rho} -\overline{u_2} & \overline{u_1} \\  \overline{u_1} & \overline{\rho}+\overline{u_2} \end{array}\right)
\\ & \partial_t \nabla h( (\xi,\xi_t) \cdot (x,t))+\nabla^{\perp}h'( (\xi,\xi_t) \cdot (x,t))= h''( (\xi,\xi_t)\cdot (x,t))\overline{m}. \end{aligned} \]
This is the analogous of \cite[Proposition 4]{siete} .
Thus one imitates the proof in \cite[Proposition 2]{siete}  and obtain weak solutions to the IPM systems such that

\[ |\rho(x,t)|=1 \]
for every $t$.  We skip the details since there is no essential difference. Also following \cite{CastroFaracoMengual} the mixing property can be proven at every time slice.

\end{rem}

{\it Proof of theorem~\ref{thmmain}.}

We start with a given initial data of Muskat type $f_0\in H^{5}$, with   $1\leq c(x,t)<2$ satisfying hypothesis of theorem \ref{existencialocal}.  By theorem~\ref{existencialocal}  there exists a time $T^*(f_0)>0$ and a function $f\in C([0,T^*(f_0)], H^4(\R))$, such that $(\ep(x,t)=c(x,t)t,f(x,t))$ solve the equation \eqref{zmedio}. By theorem \ref{sub} there exists a M-subsolution in $[0,T(f_0,M,c)]$, with $T(f_0,M,c)\leq T^*(f_0)$, and therefore we can define the space $X_0$ associated to this subsolution and apply theorem \ref{mixing residual}.

\hfill $\Box$

\section{Constructing a subsolution for the IPM system}\label{subsolutionsection}

     This section is  divided in two parts and its purpose is to show the existence of a subsolution. In the first part we will find a subsolution for the IPM system in the sense of definition \ref{subsolution}  assuming that there exist a solution for the equation \eqref{zmedio}.  We next state such existence theorem with the precise conditions on the
 speed of opening $c$.

\begin{thm}\label{existencialocal}Let  $f_0(x)\in H^{5}(\R)$  and  $c(x,t) \in C^\infty(\R\times \R^+) $
  and such that either there exist constants $c_\infty\in \R$ and $\kappa>0$ such that $1+\kappa \le c(x,t) \le 2 $ and
$c(x,t)-c_\infty \in C^1\left([0,\infty); H^6(\R)\right)$, with $$\sup_{t\in \R^+}\left(||c(\cdot,t)-c_\infty||_{H^6(\R)},
+||\pa_t c(\cdot,t)||_{H^6(\R)}\right)\leq C$$ or  $c(x,t)=1$. Then there exists a time $T>0$ and
$$f(x,t)\in C([0,T], H^4(\R))\cap C^1([0,T],H^3(\R)),$$ solving the equation \eqref{zmedio} with $\ep(x,t)=c(x,t)t$.
\end{thm}

\begin{rem}
The condition $c \ge 1+\kappa$ could be replaced by $c\ge 1$ plus technical conditions on the zeros of $c-1$ and the behaviour of $c$ at $\pm \infty$.
 This would only affect the proof
of lemma  \ref{symbolp_+}  which would be less neat. We have preferred to keep the statement of the theorem easy.  In order to deal with
low speed of opening $c<1$  different pseudodifferential machinery  is needed to deal with equation \eqref{zmedio}, thus we have not pursued the issue here. The $H^6$-condition is needed in the proof of \ref{ordenmasalto}.
Finally, in our proof we have prescribed  $\rho=\frac{\lambda}{\ep}$
as it is simplest continuous function and agrees with the entropy and maximal mixing solution in  the case of the flat interface. Other choices might be of interest though the proof would be technically  different as the velocity would change. We have not explored this later aspect.  {The fact that the maximum growth of the mixing zone is linear on $t$ seems to be intrinsically related to the problem and it is coherent with Darcy's law and the flat interface case. Our proof quickly breaks if we want to have sublinear growth (see equation \eqref{gamma}).}
\end{rem}

\begin{rem} Let us explain  why we need to ask five derivatives on the initial data. Firstly, in order to perform energy estimates we need to  quasi-linearize the equation as in lemmas \ref{ordenmasalto} and \ref{descomposicion}, where $l.o.t$ are defined in \ref{l.o.t}. The number of derivatives that we take it is enough to get the estimate in \ref{l.o.t} for the $l.o.t.$. We do not claim that the regularity can not be improved to get solutions to \ref{zmedio} with an initial data $f_0\in H^k$ and $k<5$.  The quasi-linearization of \ref{zmedio} in that case would be much more complicated. Secondly, in order to deal with the higher order terms in lemma \ref{descomposicion}, we need some regularity in the pseudodifferential operators that arise in section \ref{apriori}. The regularity of these operators is linked to that of the solution $f$. It turns that, again,  the number of derivatives we take suffices for our purposes.
\end{rem}

\subsection{Constructing a subsolution. Part 1}\label{subsolution1}

This section is dedicated to the proof of the following theorem.
\begin{thm}\label{sub}Let us assume that $f$, with $f(x,t)\in C^1([0,T]\times \R)$, solves the equation \eqref{zmedio}, with $c(x,t)$ as in theorem \ref{existencialocal}. Then there exists a M-subsolution of the IPM system for $t\in [0,T]$, $T$ small enough depending on $f_0(x)$, and for some $M$.
\end{thm}

We start by defining the mixing zone.
 For $x\in (-\infty,\infty)$ and $-\ep(x,t)<\lambda<\ep(x,t)$ we define the change of coordinates
\begin{align*}
\xb(x,\lambda)=(x,\, \lambda+ f(x,t)).
\end{align*}

We define the set $\Omega_{mix} \subset \R^2$  as follows
\begin{align}\label{zonademezcla}\Omega_{mix}=\{ \xb\in \R^2 \,:\, \xb=\xb(x,\lambda)\, \quad\text{for}\quad (x,\lambda)\in (-\infty,\infty)\times (-\ep(x,t),\ep(x,t))\}.\end{align}

Recall that, in $\Omega_{mix}$, our subsolutions $(\rho,\mb,\ub)$ should solve
\begin{align}
\pa_t\rho+\nabla\cdot \mb=&0\label{primera}\\
\ub= & \mathcal{BS}(-\pa_{x_1}\rho)\label{segunda}.
\end{align}

We prescribe $\mb$ to be of the form
 $$\mb=\rho\ub-(0,\alpha)\left(1-\rho^2\right),$$
 where $\alpha$ will be chosen later.
Then the transport equation  \eqref{primera} reads
\begin{align}\label{primerabis}
\pa_t\rho+ \ub\cdot \nabla \rho = \nabla \cdot \left((0,\alpha)\left(1-\rho^2\right)\right).
\end{align}

On the other hand we need $(\rho,\ub,\mb)\in \text{int } K^\Lambda$, in \eqref{convex1}, which is equivalent to
\begin{align} &\left(\alpha-\frac{1}{2}\right)^2<\left(\frac{1}{2}\right)^2\label{convex12},\\
&\rho^2<1.\label{convex22}
\end{align}

In fact, we need $(\rho,\ub,\mb)\in \text{int}K^\Lambda_M$, but we will take care of this later.

\subsubsection{The equations in $(x,\lambda)$-coordinates and the choices of $\rho$ and $\mb$}

Next we write the equation \eqref{primerabis} in $(x,\lambda)-$ coordinates. Let
 $g:\Omega_{mix} \to \R$ be a smooth function. We will denote
\begin{equation}
g^\sharp(x,\lambda)=g(\xb(x,\lambda)).
\end{equation}

Let us analyze the mixing error in these new coordinates. Set
\begin{equation}
\textbf{E}^\sharp(x,\lambda)=\left(1-\rhos {}^2\right)(0,\alpha^\sharp),
\end{equation}
which we split as,

$$\textbf{E}^\sharp= \overline{\fb}^\sharp+\eb^\sharp,$$
with
\begin{align*}
\overline{\fb}^\sharp=\left(1-\rhos {}^2\right)\left(0,\alpha^\sharp-\frac{1}{2}\right), &&
\eb^\sharp=\frac{1}{2}(0,1)\left(1-\rhos {}^2\right).
\end{align*}

We will define the density in the mixing zone to be \begin{align}\label{quienesrho} \rhos(x,\lambda)=\frac{\lambda}{\ep(x,t)}\end{align} and it will simplify the calculation to call $h^\sharp=\left(\alpha^\sharp-\frac{1}{2}\right)\left(1-\left(\rhos\right)^2\right)$.
Then $\rhos$  produces a density $\rho(\xb)$ satisfying the condition \eqref{convex22} in $\Omega_{mix}$. In addition $\rhos(\pm\ep)=\pm 1$ thus
\begin{align}\label{rho}
\rho(\xb)=\left\{\begin{array}{cc}\mp 1 & \text{in $\overline{\Omega^\pm}$}\\ \rho(\xb) & \text{in $\Omega_{mix}$}\end{array}\right.,
\end{align}
where $\Omega^+$ is the open domain below $\Omega_{mix}$ and $\Omega^-$ is the open domain above $\Omega_{mix}$, is  a continuous function in $\R^2$.

After, these choices, the next lemma describes the necessary conditions to be a subsolution.

\begin{lemma}\label{h} Let  $\rhos=\frac{\lambda}{\ep}$ and $\mb^\sharp=\rhos\ub^\sharp-\hs(0,1)-\eb^\sharp$ with $\eb^\sharp=\frac{1}{2}(0,1)(1-\rhos{}^2)$. Then, $\rho$, $\ub$ and $\mb$ satisfy the equation \eqref{primera} if and only if
\begin{align}\label{ecuacionparah}
\pa_\lambda \hs = \frac{\lambda}{\ep^2}+\pa_t\rhos+\frac{1}{\ep}\left(\ub^\sharp\cdot (-\pa_x f-\lambda \frac{\pa_x\ep}{\ep},1)-\pa_tf\right).
\end{align}
In addition if $$\hs=\gamma^\sharp(1-\rhos{}^2)$$ the inclusion \eqref{convex12} reads
$$|\gamma^\sharp{}|<\frac12.$$
\end{lemma}

\begin{proof} Since $(\pa_{x_1}\rho)(x,\lambda+f(x,t))= \rhos_x(x,\lambda)-\pa_x f(x,t)\pa_\lambda\rhos$ and $(\pa_{x_2}\rho)(x,\lambda+f(x,t))=\pa_\lambda\rhos(x,\lambda)$ we have that \begin{align}\label{1n} \ub(x,\lambda+f(x,t))\cdot \left(\nabla\rho\right)(x,\lambda+f(x,t))=\ub^\sharp(x,\lambda)\cdot\left(\pa_x\rhos(x,\lambda)-\pa_x f(x,t)\pa_\lambda\rhos,\,\pa_\lambda\rhos(x,\lambda)\right).\end{align}
Also
\begin{align}\label{2n}
\pa_t\rhos(x,\lambda)=&(\pa_t\rho)(x,\lambda+f(x,t))+(\pa_{x_2}\rho)(x,\lambda+f(x,t))\pa_t f \nonumber\\ =&(\pa_t\rho)(x,\lambda+f(x,t))+\pa_\lambda\rhos(x,\lambda)\pa_t f(x,t).
\end{align}

In addition,
\begin{align}\label{3n}
\nabla\cdot(\overline{\fb}+\eb)(x,\lambda+f(x,t))=\pa_\lambda h^\sharp(x,\lambda)+\pa_\lambda \eb^\sharp(x,\lambda)=\pa_\lambda h^\sharp-\rhos\pa_\lambda\rhos.
\end{align}

Evaluating \eqref{primerabis} at $\xb=(x,\lambda+f(x,t))$, putting together \eqref{1n}, \eqref{2n} and \eqref{3n} and taking into account \eqref{quienesrho} yields \eqref{ecuacionparah}.

Finally, if we define $$\gamma^\sharp=\frac{1}{1-\rhos{}^2}\hs,$$
the condition \eqref{convex12} reads
$$\gamma^\sharp{}^2<\left(\frac{1}{2}\right)^2.$$
\end{proof}

From lemma \ref{h} we have that in order to prove theorem  \ref{sub}, it is enough to show that $\gamma^\sharp{}^2<\frac12$ with $\gamma^\sharp$ given by
\begin{align}\label{gamma}
\gamma^\sharp\left(1-\rhos{}^2\right)=&\int_{-\ep}^\lambda \frac{\lambda'}{\ep^2}-\lambda\frac{\pa_t\ep}{\ep^2}+\frac{1}{\ep}\left(\ub^\sharp(x,\lambda)\cdot (-\pa_xf-\lambda\frac{\pa_x\ep}{\ep},1)-\pa_tf \right)d\lambda'\\
=&-\frac{(1-\ep_t)}{2}(1-\rhos{}^2)+\int_{-1}^{\rhos}\left(\ub^\sharp(x,\ep(x)\lambda')\cdot (-\pa_x f(x)-\pa_x\ep\lambda',1)-f_t\right)d\lambda'\nonumber,
\end{align}
$\ub$ given by the Biot-Savart law and
 $\rho(\xb)$ by \eqref{rho} and \eqref{quienesrho}.

\subsubsection{The velocity $\ub$ and the equation for the pseudointerface}
The velocity $\ub$ is given by the expression
\begin{align*}
\ub(\xb)=-\frac{1}{2\pi}\int_{\R^2}\frac{(\xb-\yb)^\perp}{|\xb-\yb|^2}\pa_{x_1}\rho(\yb)d\yb=
-\frac{1}{2\pi}\int_{\Omega_{mix}}\frac{(\xb-\yb)^\perp}{|\xb-\yb|^2}\pa_{x_1}\rho(\yb)d\yb.
\end{align*}
Then a change of coordinates yields
\begin{align}\label{velocidad}
\ub(\xb)&=-\frac{1}{\pi}\int_{-\infty}^\infty\frac{1}{2\ep(y)}\int_{-\ep(y)}^{\ep(y)} \frac{(\xb-\xb(y,\lambda'))^\perp}{|\xb-\xb(y,\lambda')|^2}\left(\pa_y f(y)+\lambda'\frac{\pa_y\ep(y)}{\ep(y)}\right)d\lambda'dy\nonumber\\
&=-\frac{1}{\pi}\int_{-\infty}^\infty\frac{1}{2}\int_{-1}^{1} \frac{(\xb-\xb(y,\ep(y)\lambda'))^\perp}{|\xb-\xb(y,\ep(y)\lambda')|^2}\left(\pa_y f(y)+\lambda'\pa_y\ep(y)\right)d\lambda'dy.
\end{align}

Next we will modify this expression  since it will help in the  proof of the local existence for the equation \eqref{zmedio}. This idea  has  been already introduced in \cite{dp}.
First we notice that
\begin{align*}
\frac{1}{2}\pa_{y}\log\left(|\xb-\xb(y,\ep(y)\lambda')|^2\right)=&-\frac{(x_1-y)}
{|\xb-\xb(y,\ep(y)\lambda')|^2}\\
&-\frac{x_2-\ep(y)\lambda'-f(y))(\pa_{y}f(y)+\lambda'\pa_y\ep(y))}
{|\xb-\xb(y,\ep(y)\lambda')|^2}.
\end{align*}
Thus since the integral of the left hand side is null (in the sense of the principal value) we can also write \eqref{velocidad} in the most
convenient form,
\begin{align}\label{velo}
\ub(\xb)=\frac{1}{\pi}P.V.\int_{-\infty}^\infty
\frac{1}{2}\int_{-1}^1 \frac{x_1-y}{|\xb-\xb(y,\ep(y)\lambda')|^2}(1,\pa_yf(y)+\lambda'\pa_y\ep(y))d\lambda'dy.
\end{align}

As we prove in the following lemma this velocity $\ub$ is in $L^\infty(\R^2)$.

\begin{lemma}\label{veloLinfty} Let $\ub$ be like in expression \eqref{velo} with $f \in H^4$ and $c$ as in theorem \eqref{existencialocal}. Then $\ub\in L^\infty(\R^2)$ and
$$||\ub(\cdot,t)||_{L^\infty(\R^2)}\leq P(||f||_{H^4})$$
for some smooth function $P$.
\end{lemma}
\begin{proof} The proof of this result is left to  appendix \ref{eftv}.
\end{proof}

We turn back to our equation \eqref{gamma}. It says that the evolution is governed by the following modified velocity.

\begin{align*}
u_c^\sharp(x,\lambda)&\equiv \ub^\sharp(x,\ep(x)\lambda)\cdot(-\pa_x f(x,t)-\pa_x\ep(x) \lambda,\,1)
\\&=\frac{1}{2\pi} P. V. \int_{\R}\int_{-1}^1
\frac{x-y}{(x-y)^2+(\ep(x)\lambda-\ep(y)\lambda'+f(x)-f(y))^2}\\&\qquad\qquad\qquad\quad\times (\pa_y f(y)-\pa_xf(x)+\pa_y\ep(y)\lambda'-\pa_x\ep(x)\lambda)d\lambda'dy,
\end{align*}

where the principal value is taken at infinity. Now, notice that by at $|\lambda|=1$, the left hand side of \eqref{gamma}  is $0$. Therefore a continuous solution
must satisfy that

\begin{align}
\pa_tf=\mathcal{M}u(x,t)=\frac{1}{2}\int_{-1}^1 u_c^\sharp(x,\lambda)d\lambda,
\end{align}
which is what motivates \eqref{zmedio}. Of course, the specific aspect of the kernel is prescribed by our ansatz for  $\rho$.

 Then, \eqref{gamma} reads

\begin{align}\label{gamma2}
&\gamma^\sharp\left(1-\rhos{}^2\right)=-\frac{(1-\ep_t)}{2}(1-\rhos{}^2)+\int_{-1}^{\rhos}\left(u^\sharp_c(x,\lambda')-f_t\right)d\lambda',
\end{align}

{\it Proof of theorem~\ref{sub}.
} We have already constructed a candidate to be  the subsolution. This candidate is given by $\left(\rho, \ub, \mb\right)$
 with $\rho^\sharp=\frac{\lambda}{\ep}$, $\ub$ as in \eqref{velocidad}, $\mb= \rho\ub-\gamma (1-\rho^2)(0,1)-\eb$, $\eb=\frac{1}{2}(0,1)(1-\rho{}^2)$, $\gamma^\sharp(s,\lambda)=\gamma(\xb(s,\lambda))$ and $\gamma^\sharp$ as in \eqref{gamma2}. Next, we show that  $|\gamma^\sharp{}|<\frac12$, as stated in lemma \ref{h}. Notice that \eqref{gamma2} yields,

 \[\gamma^\sharp=-\frac{(1-\ep_t)}{2}+\frac{1}{ 1-\rhos{}^2 } \int_{-1}^{\rhos}\left(u^\sharp_c(x,\lambda')-f_t\right)d\lambda', \]

We first focus on the first term on the right hand side of this equation.  Notice that
$|1-\pa_t\ep|\leq |1-c(x,t)|+ |\pa_t c(x,t)|t$.  Therefore, our choice of  $1\leq c(x,t)<2$ (see statement of theorem \ref{existencialocal}) implies that $|1-\pa_t\ep|<1$ for small enough time. Then to finish the proof it  is enough to prove that the second term in \eqref{gamma2} is as small as we want by making $t$ small.  This  term is problematic because the factor $\frac{1}{(1-{\rho^\sharp}^2)}$. However we will find a cancelation in order to control it by continuity.

 Here it is where we will use the relation between $\ep$ and $f$. First we will deal with the  part of $\Omega_{mix}$ which lies below the pseudointerface, i.e  $-\ep<\lambda<0$. We need to make
 small the term

\begin{align*}
&\left|\frac{1}{1-\rhos{}^2}\int_{-1}^{\rhos}(u_c^\sharp(x,\lambda')-\pa_t f) d\lambda'\right|
\\ &\leq C \frac{1}{1-\rhos}\sup_{x\in \R}  \sup_{-1<\lambda<0}\left|u_c^\sharp(x,\lambda)-\pa_t f(x,t)\right| \\ & \leq  C \sup_{x \in \R}  \sup_{-1<\lambda<0}\left|u^\sharp_c(x,\lambda)-\pa_t f(x,t)\right|. \\
\end{align*}
Here notice that  $\rhos<0$.

Then we see that, since

\begin{align*}u_c^\sharp(x,\lambda)-\pa_t f(x,t)=\frac{1}{2}\int_{-1}^{1}\left(u^\sharp_c(x,\lambda)-u^\sharp_c(x,\lambda')\right)d\lambda'.
\end{align*}
 Lemma \ref{falta}, where it is proven  that  $|u_c^\sharp(x,\lambda)-u_c^\sharp(x,\lambda')|=O(t)$  uniformly in $x$,  implies that  this term is as small as we want by taking $t$ small.

To  deal with the upper part of $\Omega_{mix}$ we use that our choice of pseudointerface,  \eqref{zmedio}, makes the situation rather symmetric. Indeed, it follows from   \eqref{zmedio}  that
\begin{align*}
\int_{-1}^{\rho^\sharp} (u_c^\sharp(x,\lambda')-\pa_t f) d\lambda'=
&-\int_{\rho^\sharp}^1 (u_c^\sharp(x,\lambda')-\pa_t f)  d\lambda'+\underbrace{\int_{-1}^1
(u_c^\sharp(x,\lambda')-\pa_t f)\ d\lambda'}_{=0}\\=
&-\int_{\rho^\sharp}^1 (u_c^\sharp(x,\lambda')-\pa_t f(x,t)) d\lambda'.
\end{align*}

Thus, the term  $\frac{1}{1-\rhos{}^2}|\int_{\lambda}^\ep\left(u^\sharp_c-\pa_tf\right) d\lambda'|,$ can be made arbitrarily small  by taking  $t$ small
as well. Hence we have proven that  there exists $T>0$, depending on $f_0$ and $c(x,t)$, such that $|\gamma^\sharp(x,\lambda,t)|<\frac12$
for $ (x,\lambda,t)\in  \mathbb{R}\times (-\ep(x,t),\ep(x,t))\times [0,T]$ as  desired.

Recall that lemma~\ref{veloLinfty} implies that $\ub \in L^\infty(\mathbb{R}^2 \times [0,T])$.

In order to conclude the proof of theorem \ref{sub} we need to check that $(\rho, \ub, \mb)$ is continuous in $(0,T)\times \Omega_{mix}(0,t)$ and that also satisfies \eqref{convex2}, \eqref{convex3} and \eqref{convex4}, for some $M>1$. The continuity is a consequence of that $\rho(\xb,t)$ is a Lipschitz function in $(0,T)\times \Omega_{mix}(t)$. Furthermore, if
$$M>8\left(||\ub||_{L^\infty}+1\right),$$
since $|\rho|\leq 1$ is easy to check \eqref{convex2}. In addition, in order to satisfy condition \eqref{convex3} we proceed as follows:
\begin{align*}
&|\mb-\ub-\frac{1}{2}(0,1-\rho)|=|\mb-\rho \ub -\frac{1}{2}(1-\rho^2)+\frac{1}{2}(1-\rho^2)-(1-\rho)\ub-\frac{1}{2}(0,1-\rho)|\\
&\leq \left((1+\rho)+|\ub|+\frac{1}{2}\right)(1-\rho),
\end{align*}
where we have used \eqref{convex1}. Then we see that \eqref{convex3} is satisfied. To check \eqref{convex4} we follows similar steps that for \eqref{convex3}.
\subsection{Constructing a subsolution. Part 2.}\label{subsolution2}
The bulk of the proof is  to show  energy estimates for \eqref{Evolution}. Before starting with the computation we will present a toy model to explain the strategy of the proof. Let us  consider the following equation
\begin{align}\label{toymodel1}
\pa_t f= & \left(\frac{1}{1+ct|\xi|}\right)^{\check{}} * \Lambda f\quad \text{in $\R\times \R^+$}\\
f(x,0)= & f^0(x)\nonumber,
\end{align}
where $1\leq c <2$.  In the Fourier side this equation reads
$$\pa_t \hat{f}(\xi)=\frac{|\xi|}{1+ct|\xi|}\hat{f}(\xi),$$
which can be solved explicitly. Indeed, the solutions are given by
\begin{align}\label{solc}
\hat{f}(\xi)=(1+ct|\xi|)^\frac{1}{c}\hat{f}_0(\xi).
\end{align}
From \eqref{solc} we see that the solution to \eqref{toymodel1} loses $\frac{1}{c}$-derivatives with respect to the initial data. Equation \eqref{zmedio} has a similar behaviour to \eqref{toymodel1} but there is no chance to find explicit solutions. Instead of that we will use energy estimates in the same way that the following energy estimate for \eqref{toymodel1}. We compute the time derivative of $\left|\left|\frac{\hat{f}(\xi)}{1+t|\xi|}\right|\right|_{L^2}$ to obtain that
\begin{align*}
\frac{1}{2}\pa_t \left|\left|\frac{\hat{f}(\xi)}{1+t|\xi|}\right|\right|^2_{L^2}\leq \int_{\R} \overline{\hat{f}(\xi)}\left(-\frac{|\xi|}{(1+t|\xi|)^2} +\frac{\pa_t \hat{f}(\xi)}{1+t|\xi|}\right)d\xi=\int_{\R}\left|\frac{\hat{f}}{1+t|\xi|}\right|^2|\xi|\left(\frac{-1}{1+t|\xi|}+\frac{1}{1+ct|\xi|}\right)d\xi,
\end{align*}
and since $\frac{1}{1+t|\xi|}\geq \frac{1}{1+ct|\xi|}\geq 0$ for $c\geq 1$ we can conclude that
\begin{align*}
\frac{1}{2}\pa_t \left|\left|\frac{\hat{f}(\xi)}{1+t|\xi|}\right|\right|^2_{L^2}\leq 0,
\end{align*}
and therefore $\left|\left|\frac{\hat{f}(\xi)}{1+t|\xi|}\right|\right|_{L^2}\leq ||f_0||_{L^2}$.

The same analysis for $\pa_x^2f$ yields the estimate
\begin{align}\label{aquella}\frac{1}{2}\pa_t \left|\left|\frac{\widehat{\pa_x^2 f}(\xi)}{1+t|\xi|}\right|\right|^2_{L^2}\leq 0.
\end{align}
 In addition, it is easy to see that $\pa_t ||f||_{L^2}^2\leq \frac{1}{2}||f||_{L^2}^2+\frac{1}{2}||\pa_x f||^2_{L^2}$. Furthermore,  $|\widehat{\pa_x f}(\xi)|$ is less or equal than $|\hat{f}(\xi)|$ for $|\xi|\leq 1$ and is less or equal than $\frac{2}{1+t|\xi|}|\widehat{\pa^{2}_x f}(\xi)|$ for $|\xi|>1$ and $t<1$ (this is just because $1\leq \frac{2|\xi|}{1+t|\xi|}$ in this range). Therefore, we have that $||\pa_xf ||^2_{L^2}\leq ||f||_{L^2}^2+ 4\left|\left|\frac{\widehat{\pa^2f}}{1+t|\xi|}\right|\right|^2_{L^2}$ and

\begin{align*}
\frac{1}{2}\pa_t ||f||_{L^2}^2\leq ||f||_{L^2}^2+ 2\left|\left|\frac{\widehat{\pa^2f}}{1+t|\xi|}\right|\right|^2_{L^2},
\end{align*}
which allows us to get, together with \eqref{aquella} that
\begin{align*}
||f||_{L^2}^2+\left|\left|\frac{\widehat{\pa_x^2f}}{1+t|\xi|}\right|\right|^2_{L^2}\leq \left(||f_0||_{L^2}^2+||\pa_x^2f_0||_{L^2}^2\right)e^{Ct},
\end{align*}
for any $t\leq 1$. Thus, for $t\leq 1$, we also control $||f||_{H^1}$, by losing one derivative with respect to the initial data, i.e.
\begin{align*}
||f(t)||_{H^1}\leq C||f_0||_{H^2},\quad \text{for $t<1$}.
\end{align*}

This strategy is flexible enough to be applied to the full system \eqref{zmedio} with the price of paying more derivatives with respect to the initial data than we actually need.
In \eqref{zmedio} we are dealing with pseudodifferential operators but arguing semiclassically we will show that they  behave as Fourier multipliers up to factors of $t$. This
is the content of the following sections.

\subsubsection{First manipulations of the equation and of the mean velocity $\mathcal{M}u$} \label{mu2}

In order to obtain energy estimates for the equation \eqref{zmedio} we need to take $5$ derivatives with respect to $x$ in  both sides of the equation.  We describe   $\pa_x^5 \mathcal{M}u$ as the sum of a main term  and lower order terms.
Since we expect to lose one derivative respect to initial data (e.g by the toy model) we will work with the Fourier multipliers,
\[\widehat{D^{-1}f}(\xi)=\frac{1}{1+2 \pi i t\xi}\hat{f}(\xi), \qquad \widehat{\mathcal{D} f}(\xi)=(1+2 \pi i t\xi)\hat{f}(\xi).\]
Notice that when $t=0$, $\mathcal{D}^{-1}$ equals to the identity and therefore it is not smoothing.

\begin{defi}\label{l.o.t} We say that a function $G(x,t):\R \times [0,T]$ is a lower order term,  $l.o.t.$, if and only if
$$||\mathcal{D}^{-1}G||_{L^2}\leq C\left(||f||_{H^4}+||\pa^5_x \mathcal{D}^{-1} f||_{L^2}\right),$$
for some smooth function $C:\R \to \R$.
\end{defi}
\begin{lemma}\label{ordenmasalto}Let $f\in H^{6}$ and $\ep=c(x,t)t$ with $c$ as in the statement of theorem \ref{existencialocal} and $0<t<1$. Then
\begin{align*}
\pa_x^{5} \mathcal{M}u=-\int_{\R} \Delta \pa^{6}_xf(x,x-y)K(x,x-y)dy + l.o.t.,
\end{align*}
where
\begin{align*}
K(x,x-y)=\frac{1}{4\pi}\int_{-1}^1\int_{-1}^1\frac{y}{y^2+(\Delta f (x,x-y)+\ep(x)\lambda-\ep(x-y)\lambda')^2}d\lambda d\lambda',
\end{align*}
and  $l.o.t.$ defined as in \ref{l.o.t}.
\end{lemma}
\begin{proof} The proof is left to appendix \ref{ordenmasaltosection}.\end{proof}

We still need to simplify the kernel $K(x,y)$ (which depends on $f$ in a nonlinear way). Actually we can linearize it as  the next lemma shows.
\begin{lemma} \label{descomposicion}Let $f\in H^{6}$ and $\ep=c(x,t)t$ with $c$ as in the statement of theorem \ref{existencialocal} and $0<t<1$. Then
\begin{align*}
\pa^{5}_x \mathcal{M}u=\int_{\R} \pa^{6}_x f(x-y)K^{c(x),\pa_xc(x)}_{\pa_xf(x)}(y)dy + a(x,t)\pa^{6}_x f(x)+l.o.t,
\end{align*}
where

\begin{align*}
K^{c(x),\pa_x c(x)}_{\pa_x f(x)}(y)=\frac{1}{4\pi}\int_{-1}^1\int_{-1}^1 \int_{\R}\frac{y}{y^2+(\pa_xf(x)y+\pa_xc(x)ty\lambda'+c(x)t(\lambda-\lambda'))^2} d\lambda d\lambda',
\end{align*}
\begin{align*}
a(x,t)\equiv - P.V.\int_{\R} K(x,y)dy,
\end{align*}
and $l.o.t.$ defined as in \ref{l.o.t}.
\end{lemma}

\begin{proof} This lemma is proven in appendix \ref{apendice4}.\end{proof}
We will deal with the equation mostly on the Fourier side. In order to show the relation with the toy model
in the following lemma we present the Fourier transform of \begin{align}\label{kcca}K^{c,c'}_A(y)=\frac{1}{4\pi}\int_{-1}^1\int_{-1}^1\frac{y}{y^2+(Ay+c'ty\lambda'+c t(\lambda-\lambda'))^2}d\lambda d\lambda'.\end{align}
Notice that to compute the Fourier transform $A,c,c'$ are taking as constants. In the application they are functions of $x$
but not of $y$.
\begin{lemma}\label{fka}
Let $K^{c,c'}_{A}$ as in \eqref{kcca} with $A,\,c'\in \R$ and $c>0$. Then its Fourier transform is given by
\begin{align}\label{fourierkcca}
\hat{K}^{c,c'}_{A}(\xi)=\frac{-i\text{sign}(\xi)}{4\cdot 2\pi c t |\xi|}\int_{-1}^{1}
\left(2-e^{2\pi\sigma_{\lambda'}ct|\xi|(-1-\lambda')(A_{\lambda'}i\text{sign}(\xi)+1)}
-e^{2\pi\sigma_{\lambda'}|\xi|(1-\lambda')(A_{\lambda'}i\text{sign}(\xi)-1)}\right)  d\lambda'.
\end{align}
where $\sigma_{\lambda'}=\frac{1}{1+A_{\lambda'}^2}$ and $A_{\lambda'}=A+c't\lambda'$.

In addition
\begin{align*}
\hat{K}^{c,0}_{A}(\xi)=\frac{-i\sign(\xi)}{2\pi c t|\xi|}\left(1+\frac{1}{4\pi c t |\xi|}\left(e^{-4\pi\sigma c t |\xi|}\left(\cos(4\pi\sigma A c t |\xi|)-A\sin(4\pi\sigma A c t |\xi|)\right)-1\right)\right),
\end{align*}
where $\sigma=\frac{1}{1+A^2}$.
\end{lemma}
\begin{proof} This lemma will be proven in appendix \ref{apendice5}, lemma~\ref{fkcca}.\end{proof}

In spite of its behaviour, a careful Taylor expansion $\hat{K}^{c,0}_{A}$ at zero (using $\sigma=\frac{1}{1+A^2}$) shows that it is bounded. On the other hand for large semiclassical frequencies $t|\xi|$ behaves like
 $\frac{-i \sign(\xi)}{2\pi ct |\xi|}$. These  two observations suggested the toy model from the beginning of the section.

 The next lemma describes more precisely the growth of   $\hat{K}^{c,0}_{A}$. It is dramatic to frame ourselves in  the realm  of positive symbols and to guess the correct energy estimate

\begin{lemma}\label{Upbound}The following estimate holds for every $(x,\xi,t) \in \R\times\R \times \R^+$ and $1\leq c\leq 2$,
$$\left|2\pi i \sign(\xi)\hat{K}^{c,0}_{A}(\xi)\right|\leq \frac{\frac{1}{c}}{1+t|\xi|}+\frac{2|A|+5+8\pi}{1+(t|\xi|)^2}.$$
\end{lemma}
\begin{proof} The proof of this lemma can be found in  appendix \ref{apendice5}, lemma~\ref{cotakc0}
\end{proof}

\subsubsection{A priori energy estimates for the quasi-linear equation}\label{apriori}

Lemma~\ref{descomposicion} says that     if $f$ is an smooth solution of \eqref{zmedio}  and we call $F(x,t)=\pa^{5}_xf(x,t)$ and $A(x,t)=\pa_xf(x,t)$ it holds that
\begin{align}\label{F1}
\pa_t F(x,t)= \int_{\R}K^{c(x),\pa_x c(x)}_{A(x)}(x-y) \pa_x F(y)dy+a(x)\pa_x F(x)+G(x),
\end{align}
where $G(x)$ is a $l.o.t$.
Let us write the equation closer to the spirit of pseudodifferential operators.
We will define the operation
\begin{align*}
K^{c(x),\pa_xc(x)}_{A(x)}\otimes f(x)=\int_{\R}K^{c(x),\pa_x c(x)}_{A(x)}(x-y) f(y)dy,
\end{align*}
in such a way that the equation \eqref{F1} reads as
\begin{align}\label{ecuF}
\pa_t F= K^{c(x),\pa_x c(x)}_{A(x)}\otimes \pa_xF(x)+ a(x)\pa_x F(x)+G(x).
\end{align}

Notice that the pseudoconvolution $\otimes$ can be alternatively expressed as,
\[  K \otimes f(x)=\Op(p)f(x), \]
where $K$ is the Schwarz kernel of the symbol $p$, i.e., $$p(x,\xi)=\int_{\R}  e^{-2\pi i y \xi} K(x,y)dy.$$

\paragraph{Definition of Symbols}

The upper bound in lemma \ref{Upbound} motivates  the definition of the following
pseudodifferential operator $\mathcal{J}^{-1}$. First we define the function $\varphi\,:\, \R^+\to \R^+$ in the following way
\begin{align}\label{varphi}
\varphi(\tau)=\frac{1}{1+\tau}+\frac{B}{1+\tau^2},
\end{align}
where $B$ is a constant that just depends on $|| f_0||_{H^5}$, $f_0$ being the initial data in \eqref{zmedio}. It suffices to take   \begin{align}\label{Bdefi} B=200 || f_0||_{H^5}+200.\end{align}
 Next we define the multiplier $j^{-1}(\xi)=e^{-\int_0^{t|\xi|}\varphi(\tau)d\tau}$
which satisfies
\begin{equation}
\partial_t  j^{-1}=-|\xi| \varphi(t|\xi|)j^{-1}.
\end{equation}
Hence, the corresponding operator $\mathcal{J}^{-1}$ of degree $-1$ is given
by the expression
\begin{align*}
\widehat{\mathcal{J}^{-1}f}(\xi)=e^{-\int_0^{t|\xi|}\varphi(\tau)d\tau}\hat{f}(\xi).
\end{align*}

Here we remark that since $\frac{1}{C}<j^{-1}(t|\xi|)(1+t|\xi|)\le C$,  $\mathcal{J}^{-1}$ is comparable to $\mathcal{D}^{-1}$
meaning that

$$ \frac{1}{C}||\mathcal{J}^{-1}f||_{L^2\to L^2}\leq  ||\mathcal{D}^{-1}f||_{L^2\to L^2}\leq C ||\mathcal{J}^{-1}f||_{L^2\to L^2},$$ where $C$ just depend on $B$.

If we read the right hand side of \eqref{ecuF} as an operator on $F$, the main part is  described by the symbol
$$p_{main}(x,\xi)\equiv 2\pi i \xi \hat{K}^{c,c'}_{A}.$$ This is a bounded symbol in $\xi$ and $x$, but its $L^\infty$ norm blows as $t^{-1}$. This is problematic to get an   uniform in time apriori estimate. Next we explain the strategy to deal with this issue. Firstly, we introduce  a suitable decomposition of $p_{main}$. The symbols $p$, $p_b$, $p_{\textrm{good}}$ and $p_+$ will be given by the expressions

\begin{align}\label{simbolos}
p=&2\pi i \xi \hat{K}^{c,0}_{A}, &
p_b=&2\pi i \xi \left(p_{main}-p\right),&
p_{\textrm{good}}=&\frac{1}{|\xi|}p-\varphi(t|\xi|),
 \quad \text{and} \quad  p_+=-(1+|\xi|) p_{\textrm{good}}.\end{align}

We point out that all of these symbols are even in $\xi$ and therefore the corresponding pseudodifferential operator are real valued.

Secondly,  we observe that $Op(p_b)$ is a bounded operator from $L^2$ to $L^2$. Then we observe that the growth of $p$ is controlled by $|\xi|\varphi$. Thus, the  G\aa rding inequality, Lemma~\ref{positive}, allows to control the norm of $Op(p_+)$ from $L^2$ to $L^2$ in terms of the norms $p_+^\frac12$. As expected, these norms blow up as $t^{-\frac12}$.  This is integrable near $0$ and suffices to our purposes.

Hence we are led to study the problem
\begin{equation} \label{pvarphi}\partial_t f=\Op(|\xi|\varphi(|t\xi|)) f. \end{equation}
Integrating the equation \eqref{pvarphi}, as in the toy problem, leads to  $\partial_t\|\mathcal{J}^{-1} F(\cdot,t)\|^2=0$. Thus in the fully nonlinear case there is the hope of the existence of energy estimate for that quantity. Indeed, this is
the case, but a few manipulations show that then  correlation between $\mathcal{J}$ and $p_{main}$   needs to be estimated as well. Happily even if $p_{main}$ blows like $t^{-1}$, this is compensated by the $t$ provided by our non smooth  semiclassical estimates.
Therefore the worst behaviour is given by $p_+^\frac12$.  The following apriori estimate  shows how   these heuristics are made rigorous.

\begin{thm}\label{estimaciondeenergia}Let $f$ be a smooth solution to the equation \eqref{zmedio} and $c$ as in the statement of theorem \ref{existencialocal}. Set $F=\pa^5_x f$. Let $0<T_p<1$ small enough such that $2||\pa_x f(\cdot,t)||_{L^\infty}+5+8\pi\leq \frac{B}{2}$, with $B$ as in \eqref{Bdefi}.  Then, if $t \in [0,T_p]$, it holds that

\[ \partial_t\|\mathcal{J}^{-1} F(\cdot,t)\|_{L^2} \leq  \frac{1}{\sqrt{t}}M\left(||f_0||_{H^5}, \left(||f||_{L^2}+ \|\mathcal{J}^{-1} F\|_{L^2}\right)\right),\]

where  $M$ is an smooth  function $M\, : \R^+\times\R^+\to\R^+$, positive and finite.
\end{thm}

\begin{proof}

Firstly, we recall that by lemma \ref{descomposicion}
\begin{align*}
\pa_t F= Op(p_{main})F+ a\pa_x F+G,
\end{align*}
where $G$ stands for l.o.t. in the sense of definition \ref{l.o.t}.
Secondly, it is crucial for our estimates that if  $t<T_p$, lemma~\ref{Upbound} and the definition of $\varphi$ implies
that $p_+>0$, ($p_+$ is even for all times).

Next we compute the time derivative,  and express it in terms of the  symbols,

\begin{align*}
&\frac{1}{2}\pa_t\int_{\R}|\mathcal{J}^{-1}F(x)|^2 dx=\int_{\R}\mathcal{J}^{-1}F\pa_t \mathcal{J}^{-1} Fdx\\
&=\int_{\R}\mathcal{J}^{-1}F \int_{\R}e^{2\pi i x\xi}\pa_t\left(j^{-1}(t|\xi|)\hat{F}(\xi)\right) d\xi dx=
\int_{\R}\mathcal{J}^{-1}F\int_{\R}e^{2\pi i x\xi}j^{-1}(t|\xi|)\left(-|\xi|\varphi(t|\xi|)+\widehat{\pa_tF}\right)d\xi dx\\
&=\int_{\R}\mathcal{J}^{-1}F\int_{\R}e^{2\pi i x\xi}j^{-1}(t|\xi|)\left(-|\xi|\varphi(t|\xi|)\hat{F}(\xi)+\mathcal{F}[Op(p_{main})F](\xi)\right)d\xi dx\\
&+\int_{\R}\mathcal{J}^{-1}F\int_{\R}e^{2\pi i x\xi}j^{-1}(t|\xi|)\widehat{a\pa_x F}(\xi)d\xi dx\\
&+\int_{\R}\mathcal{J}^{-1}F\mathcal{J}^{-1}Gdx.
\end{align*}

We denote $g=\mathcal{J}^{-1}F$ and we will split the  term  $$ \int_{\R}e^{2\pi i x\xi}j^{-1}(t|\xi|)\left(-|\xi|\varphi(t|\xi|)\hat{F}(\xi)+\mathcal{F}[Op(p_{main})F](\xi)\right)d\xi dx$$
in the following way
\begin{align*}
&\int_{\R}e^{2\pi i x\xi}\left(-|\xi|\varphi(t|\xi|)\hat{g}(\xi)+\mathcal{F}[\mathcal{J}^{-1}\circ Op(p_{main})\circ \mathcal{J} g](\xi)\right)d\xi \\
&=\int_{\R}e^{2\pi i x\xi}\left(-|\xi|\varphi(t|\xi|)+p_{main}(x,\xi)\right)\hat{g}(\xi)d\xi\\
&+  \mathcal{J}^{-1}\circ Op(p_{main})\circ \mathcal{J} g-Op(p_{main})g\\
&= -Op(p_+)g+ Op(\varphi-\frac{1}{|\xi|}p)g+Op(p_b)g+\mathcal{J}^{-1}[Op(p_{main}),\, \mathcal{J}]g,
\end{align*}
where we have just added and subtracted  $Op(p_{main})g$ in the first equality and in the second one we have used the definition of $p_{main}$ and $p_+$.

Then,

\begin{align*}
&\pa_t \|g\|^2_{L^2}\leq -\underbrace{\int_{\R} gOp(p_+)g dx}_{I_+(g)}+ \underbrace{\int_{\R} g Op(\varphi-\frac{1}{|\xi|}p)gdx}_{I_{good}(g)}+\underbrace{\int_{\R}g Op(p_b)gdx}_{I_b(g)}+\underbrace{\int_{\R}g \mathcal{J}^{-1}[Op(p_{main}),\, \mathcal{J}]g dx}_{I_{com}}\\
&+\underbrace{\int_{\R} g \mathcal{J}^{-1}(a\pa_xF)dx}_{I_{transport}(g)} +\underbrace{\int_{\R}g \mathcal{J}^{-1}G dx}_{I_{l.o.t.}(g)}.
\end{align*}

We recall that the symbols $|||f|||$ and $\lar$ will denote some polynomial function evaluated in $||f||_{H^4}$
and $||A||_{H^3}$ respectively ($A=\partial_xf$ ). Thus since $\|\mathcal{J}^{-1}F\|_{L^2} $ is comparable with $\|D^{-1}F \|_{L^2}$ it holds
that, for finite time,

\begin{equation}\label{glot} |||f|||+\lar \le C\left(\|f\|_{L^2}+ \|\mathcal{J}^{-1}F\|_{L^2}\right),
\end{equation}
 where the right hand side of \eqref{glot} $C$ means a smooth function evaluated at $\|f\|_{L^2}+ \|\mathcal{J}^{-1}F\|_{L^2}$.

We can estimate this collection of terms in the following way:

\begin{enumerate}
 \item $|I_{+}(g)| \leq \frac{\lar }{\sqrt{t}}\|g\|^2_{L^2}$. In order to get this inequality we first use lemma~\ref{positive}. After that we use that
  $$\|p_+^\frac12\|_{1,1}  \left|\left|Op\left(p_+^\frac12\right)^{skew} \right|\right|_{L^2\to L^2}+ \left|\left|\mathfrak{C}\left(p_+^\frac12,p_+^\frac12\right)\right|\right|_{L^2\to L^2} \leq \lar t^{-\frac12}.$$
(See section \ref{notation} for the notation).  The estimate for $\|p_+^\frac12\|_{1,1}\leq \lar t^{-\frac{1}{2}}$ can be found in lemma \ref{symbolp_+}. The bound for $\left|\left|\Op\left(p_+^\frac12\right)^{skew} \right|\right|_{L^2\to L^2}\leq \lar $ is a consequence of theorem \ref{adjoints} and lemma \ref{symbolp_+}. The bound for $\left|\left|\mathfrak{C}\left(p_+^\frac12,p_+^\frac12\right)\right|\right|\leq \lar t^{-\frac{1}{2}}$ follows from theorem \ref{CastroHwangComp} and lemma \ref{symbolp_+}.

\item $|I_{good}(g)|+|I_{b}(g)|\leq \lar ||g||_{L^2}^2$
  by the estimates $||Op(p_b)||_{L^2\to L^2}+||Op(p_{\textrm{good}})||_{L^2\to L^2}\leq \lar.$ These estimates are a consequence of theorem \ref{Hwan} and lemma \ref{pestimates}.

\item $|I_{Com}(g)| \le \lar  \|g\|^2_{L^2}$ follows from $|| \mathcal{J}^{-1}[Op(p_{main}),\mathcal{J}]||_{L^2\to L^2}\leq \lar$. This estimate is a consequence of theorem \ref{ComJp} and lemma \ref{pestimates}.

\item  $|I_{transport}(g)| \le |||f||| \|g\|^2_{L^2}$ by lemma~\ref{transport} and the estimate for the norm of $a$, given in lemma \ref{alemma}.

\item   $|I_{l.o.t}(g)| \le \|\mathcal{J}^{-1}G\|_{L^2}\|g\|_{L^2} \le C  \|\mathcal{D}^{-1}G\|_{L^2} \|g\|_{L^2}  \le  C(|||f|||)  \|g\|_{L^2}$
 where $C$ is the function appearing in the  definition of lower order terms, definition~\ref{l.o.t}.
 \end{enumerate}

 Finally notice that in the definition of $\varphi$, appears a constant $B$ which depends on $f_0$. Thus  as long as
 $0<t<T_p$, since  $p_+>0$, the claim follows
 where the function $M$ is built from the function $C$ and a high power of $\|f\|_{L^2}+\|g\|_{L^2}$.

\end{proof}

\begin{proposition} \label{apriorif} Let $f$ be and smooth solution of equation \eqref{zmedio}, with $f_0\in H^5$ and $c$ as in theorem \ref{existencialocal}. Then there is $T=T(\|f_0\|_{H^5})$ such that
\begin{align*}
\sup_{0<t<T}||f||_{H^4}\leq \sup_{0<t<T}\left(||f||_{L^2}+2||\mathcal{D}^{-1}F||_{L^2}\right)\leq P\left(||f_0||_{H^{5}}\right)
\end{align*}
where $P$ is some bounded function.
\end{proposition}

\begin{proof} Let $u(t)=(||f||_{L^2}+ ||D^{-1}\pa_x^5f||_{L^2})^2$.   From theorem \ref{estimaciondeenergia} and since $\pa_t||f||_{L^2}$ is easy to control by a function of $u(t)$, we have that, for $t\in [0,T_p]$,

\begin{equation}\label{ode}
\frac{\pa_t u(t)}{M(||f_0||_{H^5},u(t))}\leq \frac{1}{\sqrt{t}}.
\end{equation}
Since $M$ is positive, the function  $U:\R^+ \to \R$ defined by   $U(x)=\int_0^x \frac{1}{M(||f^0||_{H^5}, y)}dy$ is increasing. Let us   integrate both sides of
\eqref{ode} respect to time.  Since $|U(u(0))| \le U(||f^0||_{H^5})$, it follows that
\begin{align*}
U(u(t))\leq U(||f^0||_{H^5})+2\sqrt{t}.
\end{align*}
Since $U(x)$ is increasing,  we see than for small time depending on $f_0$ the initial data all smooth solutions
satisfy that

\[u(t) \le P(\|f_0\|_{H^5}) .\]
In particular since the time of positiveness $T_p$ depends on $|\pa_xf|$, this
yields a lower bound $T_p$ which depends on $||f_0||_{H^5}$ but not on $f$. Thus we can select   $T$, in such a way that we achieve the conclusion of proposition \eqref{apriorif}.
\end{proof}

\subsubsection{The regularized system and local existence} In order to be able to apply a Picard's theorem we will regularize the system by using two parameters, $\delta$ and $\kappa$. With the parameter $\delta$ we regularize the transport term and with the parameter $\kappa$ the nonlocal  operator.  We will consider  the following equation for $ f^{\kappa,\,\delta}(x,t)$,
\begin{align}\label{fmedior}
\pa_t f^{\kappa,\,\delta}(x)= & -\phi_\delta * \int_{-\infty}^\infty \left(\phi_{\delta}*\pa_x f^{\kappa,\,\delta}(x)- \phi_\delta * \pa_y f^{\kappa,\, \delta}(y)\right)K_{\ep_\kappa}(x,y)dy \nonumber\\ &\underbrace{-\frac{1}{4\pi}\int_{\R}\int_{-1}^{1}\int_{-1}^1\frac{\pa_x\ep_\kappa(x)\lambda-\pa_x\ep_\kappa(x-y)\lambda'}{y^2+\left(\Delta f(x,x-y)+\ep_\kappa(x)\lambda-\ep_\kappa(x-y)\lambda'\right)^2}d\lambda d\lambda'dy}_{G_\kappa[f^{\kappa,\,\delta}]}\nonumber\\ & +\kappa \phi_\delta*\pa_x^2 \phi_\delta* f^{\kappa,\delta}  \\ f^{\kappa,\,\delta}(x,0)= & f^0(x),\nonumber\\
\end{align}
where $\kappa,\, \delta >0$, $\phi$ is a positive and smooth function with mean equal to one and $\phi_\delta=\frac{1}{\delta}\phi\left(\frac{x}{\delta}\right)$ and $K_{\kappa}(x,y)$ is like $K(x,y)$ in lemma \ref{ordenmasalto} but replacing $\ep(x,t)=c(x,t)t$ by $\ep_\kappa(x,t)=c(x,t)(t+\kappa)$ (also $\ep(y,t)=c(y,t)t$ pass to $c(y,t)(t+\kappa)$).

The Picard's theorem that we  will  apply is the following
\begin{thm}[Picard]\label{picard} Let $B$ be a Banach space and $O\subset B$ an open set. Let us consider the equation
\begin{align}\label{p1}
\frac{d \textbf{X}(t)}{dt}= & \textbf{F}[\textbf{X},t]\\
\textbf{X}(0)= & \textbf{X}_0,\label{p2}
\end{align}
where
\begin{align*}
\textbf{F}[\cdot,t] \,:\, O \to B \quad \text{for $|t|<\eta$, for some $\eta>0$}
\end{align*}
is continuous in a neighbourhood of  $\textbf{X}_0\subset O$. Suppose further that,  $\textbf{F}$ is Lipschitz in $O$, i.e.,
\begin{align*}
||\textbf{F}[\textbf{X}^1,t]-\textbf{F}[\textbf{X}^2,t]||_{B}\leq C(O)||\textbf{X}^1-\textbf{X}^2||_{B},\quad \text{for $|t|\leq \eta$},
\end{align*}
and $\textbf{F}[\textbf{X}_0,t]$ is a continuous function of $t$ for $t\leq |\eta|$ with values on $B$, with $||\textbf{F}[X_0,t]||_{B}\leq C$. Then,
there exist $T>0$ and  a unique $\textbf{X}(t)\in C^1([-T,T], O)$ solving \eqref{p1}, \eqref{p2}.
\end{thm}

By applying theorem \ref{picard} the following result holds:
\begin{thm} Let $f^0\in H^4(\R)$, $c$ as in theorem \ref{existencialocal}  and $\delta$, $\kappa >0$. Then there exist $T^{\kappa,\delta}>0$ (depending on $\kappa$ and $\delta$) and
$$f^{\kappa,\, \delta}\in C((-T^{\kappa,\delta},T^{\kappa,\delta}); H^4(\R))$$ such that $ f^{\kappa,\,\delta}(x,t))$ solves the system \eqref{fmedior}. In addition, this solution can be extended if its $H^4$-norm is bounded.
\end{thm}
\begin{proof} In order to apply theorem \ref{picard} we choose $B= H^4$,
$$O_{\,M}=\{f\in H^4\,:\,\, ||f||_{H^4}<M\},$$
 $X_0=f^0$ (we take $M>||f^0||_{H^k}$) and
 $$F=-\phi_\delta * \int_{-\infty}^\infty \left(\phi_{\delta}*\pa_x f^{\kappa,\,\delta}(x)- \phi_\delta * \pa_y f^{\kappa,\, \delta}(y)\right)K_{\ep^{\kappa}}(x,y)dy+G_{\kappa}[f^{\kappa,\delta}]+\kappa \phi_\delta*\pa^2_x\phi_\delta* f^{\kappa,\delta}(x).$$
Because the properties of the mollifiers $\phi_\delta$ and  that the kernel $K_{\ep^{\kappa}}$ is not singular in $O_{M}$ ($\ep_\kappa> \frac{\kappa}{2}$, for $T^{\kappa,\delta}<\frac{\kappa}{2}$ in this open set), the hypothesis of theorem \ref{picard} can be verified. In addition we notice that $F$ is also Lipschitz on $t$ thus the solutions can be extended on time as long as its $H^4$-norm is bounded. This is rather standard and we will omit the details.\end{proof}

{\it Proof of Thorem~\ref{existencialocal}}
Once, we dispose of the solutions $ f^{\delta,\kappa}$ we need to obtain estimates independent of $\delta$ and $\kappa$, for positive time, in order to be able extend these solutions to an interval $[0,T)$, with $T$ independent of $\delta$ and $\kappa$. Then, we are entitled to  take the limit. After taking four derivatives in $F$, we find that
\begin{align*}
\pa^4_x F=&\phi_\delta* \left(a(x)\phi_\delta*\pa^{5}_x f(x)\right)+\phi_\delta* \int_{-\infty}^\infty K_{\ep_\kappa}(x,y)\phi_\delta*\pa^{5}_{x}f(y)dy\\&+\kappa \phi_\delta*\pa^2_x\phi_\delta* f^{\kappa,\delta} + l.o.t.,
\end{align*}
where $$a(x)=-P.V.\int_{-\infty}^\infty K_{\ep_\kappa}(x,y)dy,$$
and $l.o.t$ means terms bounded in $H^4$ independently of $\delta$.

Therefore, the main terms in the derivative $\frac{1}{2}\pa_t ||f||_{H^4}$ are
\begin{align*}
&\int_{-\infty}^\infty\phi_\delta* \left(a(x)\phi_\delta*\pa^{5}_x f(x)\right)\pa^4_x f(x)dx\end{align*}
and
\begin{align*}
&\int_{-\infty}^\infty\phi_\delta*\int_{-\infty}^\infty K_{\ep_k}(x,y)\phi_\delta*\pa^{5}_{x}f(y)\,dy\,\pa^{4}_x f(x)dx.
\end{align*}
The first term can be bounded in the following way
\begin{align*}
&\left|\int_{-\infty}^\infty\phi_\delta* \left(a(x)\phi_\delta*\pa^{5}_x f(x)\right)\pa^4_x f(x)dx\right|\\
&=\left|\int_{-\infty}^\infty a(x)\phi_\delta*\pa^{5}_x f(x)\phi_\delta*\pa^4_x f(x)dx\right|
\\&\leq C||\pa_x a||_{L^\infty}||\pa^4_x f||_{L^2}^2.
\end{align*}
And in order to bound the second one we just notice that $K_{\ep_\kappa}$ is not singular because $\ep_\kappa=c(x,t)(t+\kappa)$ and then we can integrate by parts in order to gain a derivative in $x$. Thus, the uniform estimate in $\delta$ are easy to get (the term coming from the Laplacian operator is treated in the usual way). The main difficulty to prove theorem  \ref{existencialocal} is then performing  estimates uniform in $\kappa$ for the equation
\begin{align}\label{fmediok}
\pa_t f^{\kappa}(x)=-\frac{1}{4\pi}\int_{\R}\int_{-1}^{1}\int_{-1}^1\frac{\pa_x f^{\kappa}(x)-\pa_x f^{\kappa}(x-y)+\pa_x\ep_\kappa(x)\lambda-\pa_x\ep_\kappa(x-y)\lambda'}{y^2+\left(\Delta f(x,x-y)+\ep_\kappa(x)\lambda-\ep_\kappa(x-y)\lambda'\right)^2}d\lambda d\lambda'dy+\kappa \pa^2_x f^{\kappa}.
\end{align}
 We notice that because of the effect of the term $\kappa \pa^2_x f^{\kappa}$ the solution to \eqref{fmediok} are actually smooth, and then, we have enough regularity to apply our energy estimates
to obtain estimates uniform in $\kappa$ as in the proof of Proposition~\ref{apriorif}. The only difference is that for the regularized system there is the new term  coming from the Laplacian. Again, this term is harmless as it is a differential and positive operator. Then we have a control of the $H^4-$norm of the solution uniform in $\kappa$. This information is enough to pass to the limit and to find a classical solution for \eqref{zmedio}.

Finally we show the continuity on time of the $H^4$-norm of the solution then $C^1-$continuity in $H^3$ follows directly from the equation. For $t>0$ the proof follows standard techniques. The continuity at $t=0$ is more delicate. This fact follows from the following argument. We can write the difference $\pa^4_x f -\pa^4_x f_0$ as
\begin{align*}
\pa^4_x f -\pa^4_x f_0= \mathcal{D}^{-1}\mathcal{D}\left(\pa^4_x f -\pa^4_x f_0\right)=\mathcal{D}^{-1}(\pa^4_xf-\pa_x^4f_0)+t\mathcal{D}^{-1}(\Lambda\pa_x^4f-\Lambda\pa_x^4f_0).
\end{align*}
The second term is controlled by the energy estimate.

In addition,
\begin{align*}
\mathcal{D}^{-1}\pa^4_x f_0=\pa^4_xf_0- t \mathcal {D}^{-1}\Lambda\pa^4_x f_0.
\end{align*}
Thus, the only problematic term is
\begin{align*}
\frac{\widehat{\pa_x^4 f}}{1+t|\xi|}-\widehat{\pa^4_xf_0}=\int_{0}^t \pa_s \frac{\widehat{\pa^4_x f}(s)}{1+s|\xi|}ds=\int_{0}^t \left( -\frac{|\xi|\widehat{\pa^4_xf}(s)}{(1+s|\xi|)^2}+\frac{\pa_t\widehat{\pa^4_xf}(s)}{1+s|\xi|}\right)ds.
\end{align*}
but $\pa_t\pa_x^4f$ is of the order of $\pa^5_xf$ by the equation.  Then taking the $L^2$ norm, again the energy estimate implies that
\begin{align*}
||\pa_x^4f-\pa_x^4f_0||_{L^2}\leq C t,
\end{align*}
for small $t$.
	
\section{Semiclassical analysis with limited smoothness on the symbols}\label{Semiclassical}
In the following section we develop what we call our semiclassical estimates. As a matter of fact, our symbols are a bit more general than those
of  the type $p(x,t \xi)$  but our results certainly apply to those. We have divided the section into a first part where we state result for general symbols
and a second one where we deal with the ones appearing in the current paper.

\subsection{General symbols}
\subsubsection{Results}
We start by recalling the basic boundedness of pseudodifferential operators with optimal smoothness as proved in \cite{CyM,Cordes,Hwang}. We
state it exactly as \cite[Theorem 1.3]{Hwang} as we will  elaborate on  ideas from this work.
\begin{thm}[I. L. Hwang]\label{Hwan} Let $p \in \Hw$ and $f\in L^2$. Then
\[ \|\Op(p)f\|_{L^2} \le C ||p||_{1,1}\|f\|_{L^2}. \]
\end{thm}

The semiclassical type estimates we need are related to the results for symbols with a limited degree of smoothness studied in \cite{Lannes} and \cite{Texier} via paradifferential calculus. However the estimates in these two papers are not enough for our purposes.

Our first result is on the correlation of symbols (see section \ref{notation}).
\begin{thm} \label{CastroHwangComp} Let $p_1,p_2 \in \Hw \cap S_{2,0}$   and $f \in L^2$.

Then,
\begin{equation}\label{Cestimate}
\|\mathfrak{C}(p_1,p_2) f\|_{L^2} \le  ||| \mathfrak{C}(p_1,p_2)|||\, \| f\|_{L^2},\end{equation}
where
\[ ||| \mathfrak{C}(p_1,p_2)||| \le C   \left(\|p_1\|_{1,1}  \|\pa_\xi p_2\|_{1,0}+ \left(\|p_2\|_{1,1}+\|p_2\|_{2,0} \right) \|\pa_\xi p_1\|_{1,0} \right).\]

\end{thm}

\begin{thm}\label{adjoints} Let $p \in \Hw$  be even in the $\xi$ variable. Let $0<\ep<1$ such that $$\sup_{\xi}\left(\| \pa_x\partial_{\xi} p(\cdot,\xi) \|_{H^{-\ep}}+\|\pa_x \partial_{\xi} p(\cdot,\xi) \|_{H^{1+\ep}}\right)<\infty.$$

 Let $f\in L^2$.

 Then
\[ \|\Op(p)^{skew} f \|_{L^2} \le |||\Op(p)^{skew}|||\, \|f\|_{L^2}, \]
where $$|||\Op(p)^{skew}|||\leq C \sup_{\xi}\left(\| \pa_x\partial_{\xi} p(\cdot,\xi) \|_{\dot{H}^{-\ep}}+\| \pa_x\partial_{\xi} p(\cdot,\xi) \|_{H^{1+\ep}}\right).$$
\end{thm}

\begin{rem} Notice that since in both theorems,  in the estimate of the norms there is multiplying  factors with $\partial_\xi$ in the case of semiclassical symbols $p(x,t\xi)$ our theorems yield a gain a factor of $t$.  The whole semiclassical calculus e.g \cite{Zworski} or \cite{Salo} for more general symbols can be replicated for non smooth symbols. A prime example  is  the coercivity of elliptic semiclassical symbols for $t$ small, which is a corollary of our results.
\end{rem}

Positive symbols have additional properties. The next G\aa rding inequality gives control of them at the price
of bounding the derivatives of $p_+^\frac12$.
\begin{lemma}[G\aa rding inequality]\label{positive}
Let $p_+$ be an even in the $\xi$ variable  positive symbol such that $p_+^\frac{1}{2}\in \mathcal{S}_{1,1}\cap \mathcal{S}_{2,0}$ and $$\sup_{\xi} \left(||\pa_x \pa_\xi p_+^\frac{1}{2}||_{H^{-\ep}}+||\pa_x\pa_\xi p_+^\frac{1}{2}||_{H^{1+\ep}}\right)<\infty,$$ for some $\ep>0$, and $f\in L^2$.  Then
 \[ -\int_{\mathbb{R}} f \Op(p_+)f dx \le  C(||| \mathfrak{C}(p_+^\frac12,p_+^\frac12)|||+ \|p_+^\frac12\|_{1,1} |||\Op(p_+^\frac12)^{skew}|||) \|f\|_{L^2}^2.\]

\end{lemma}

\subsubsection{Proofs}

Our proof are inspired in the ideas of Hwang to prove theorem~\ref{Hwan}.  As usual, in the proofs we obtain the estimates applying the
various operators to functions in the Schwarz class, where we can use  the explicitly representation of the operators as integrals against the symbols, and achieve $L^2$ results by density. Moreover this fact makes it enough  to obtain the correct bounds considering  smooth and fast decaying approximations of the symbols. We will provide some of the details in
the proof of Theorem~\ref{adjoints}, where these arguments are slightly more involved, and skip them in the rest of the theorems.
 Several integration by parts in combination with the basic
properties of the exponential and Plancherel identity are used recurrently. Hence  we have isolated them in some preliminary lemmas.

The first lemma is an extension of  \cite[Lemma 3.1]{Hwang}.
\begin{lemma} \label{basic1} Let $f \in L^2$. We define, for $(y,\eta) \in \mathbb{R}^2$,
\[h_f(y,\eta)=\int_{\R}e^{2\pi i \eta z}\frac{f(z)}{1+2\pi i (y-z)} dz.\]
Then for $k \in 0 \cup \mathbb{N} $,
\[ \|\partial^k_y h\|_{L^2(\mathbb{R}^2)}\le C \|f\|_{L^2}.\]
 Let $p(y,\eta) \in S_{k,0}$  and set
$\Gamma_f(y,\eta)=p(y,\eta) h_f(y,\eta)$. Then,
\[ \|\partial_y^k \Gamma_f\|_{L^2(\mathbb{R}^2)}\le \|p\|_{k,0}  \|f\|_{L^2}.\]
\end{lemma}

\begin{proof}
\cite[Lemma 3.1]{Hwang}, which follows by Plancherel and a change of variable, says that if $g,f \in L^2$
\[h_{f,g}(y,\eta)= \int_{\R}e^{2\pi i \eta z} f(z)g(y-z) dz  \]
satisfies that
\[ \| h\|_{L^2(\mathbb{R}^2)}\le C\|g\|_{L^2} \|f\|_{L^2}.\]
Notice next that for  $k \in 0 \cup \mathbb{N} $ the function  $g=\pa_x^k(\frac{1}{1+2\pi ix}) \in L^2$. Hence
we can differentiate $h_f(y,\eta)$  under the integral sign and the claim follows from \cite[Lemma 3.1]{Hwang}.

The second estimation follows  from the first, the assumptions and  the product rule.
\end{proof}

\begin{lemma}\label{basic2}Let $r,g \in L^2$ and $F \in L^2(\mathbb{R}^2)$ and define, for $(x,\xi) \in \mathbb{R}^2$,
\begin{equation}\label{G1}G(x,\xi)=\int_{\R}\int_{\R} e^{2\pi i (\eta-\xi)y}  F(y,\eta) r(\eta-\xi) g (x-y) dyd\eta.\end{equation}

Then,
\[\|G(x,\xi)\|_{L^2(\mathbb{R}^2)} \le \|r\|_{L^2}\|g\|_{L^2} \|F(y,\eta)\|_{L^2(\mathbb{R}^2)}. \]
\end{lemma}
\begin{proof}
We first take Fourier transform in $x$  and do a change  of variables  to obtain

\[
\hat{G}(\alpha,\xi)=\int_{\R} e^{-2\pi i x\alpha} G(x,\xi) dx= \int_{\R}\int_{\R} e^{2\pi i (\eta-\xi-\alpha )y}  F(y,\eta) r(\eta-\xi) \hat{g} (\alpha) dyd\eta.  \]

Next,   Cauchy-Schwarz inequality respect to  $\eta$ yields  the pointwise estimate,

\begin{equation}\label{pointwiseestimate}|\hat{G}(\alpha ,\xi)|^2 \le \|r\|_{L^2}^2  |\hat{g}(\alpha)|^2 \int_{\R}  \left|\int_{\R}  e^{2\pi i (\eta-\xi-\alpha)y}  F(y,\eta)  dy\right|^2d\eta.
\end{equation}

We will need that   Plancherel identity, with variables $y,\xi$, yields the equality

\begin{equation}\label{plancherel} \int_{\R}  \left|\int_{\R}  e^{2\pi i (\eta-\xi-\alpha)y}  F(y,\eta)  dy\right|^2 d\xi= \int_{\R} |F(y,\eta)|^2 dy. \end{equation}

Thus, we first apply  \eqref{pointwiseestimate} and Plancherel again to bound the $L^2$ norm of $G$,
\[\|G\|_{L^2(\mathbb{R}^2)}^2 \le  \|r\|_{L^2}^2  \int_{\R} |\hat{g}(\alpha)|^2  \int_{\R}  \left|\int_{\R}  e^{2\pi i (\eta-\xi-\alpha)y}  F(y,\eta)  dy\right|^2 d\xi d\alpha,\]
and we conclude by integrating first in $\xi$ and then applying \eqref{plancherel}.  With a final use of Plancherel the lemma is proved.
\end{proof}

\begin{lemma}\label{basic3}
Let $g \in L^2, \Gamma, \partial_y \Gamma \in L^2(\mathbb{R}^2)$ and define, for $(x,\xi) \in \mathbb{R}^2$,
\begin{equation}\label{G2}G(x,\xi)=\int_{\R^2} e^{2\pi i y(\eta-\xi)} \Gamma(\eta,y) g(x-y)d\eta dy.\end{equation}
Then,
\[
\|G\|_{L^2(\mathbb{R}^2)} \le \|g\|_{H^1} \|(1-\pa_y) \Gamma \|_{L^2(\mathbb{R}^2)}.\]

\end{lemma}
\begin{proof}

Let $r(x)=\frac{1}{1+2\pi ix}$. We will  use that
\begin{equation}\label{partes1} e^{2\pi i (\eta-\xi)y}=\frac{1}{1+2\pi i(\eta-\xi)}(1+\pa_y)e^{2\pi i (\eta-\xi)}.
\end{equation}
 We insert \eqref{partes1} into \eqref{G2} and integrate by parts respect to  $y$ to obtain that,

\[\begin{aligned}G(x,\xi)&=\int_{\mathbb{R}^2} e^{2\pi i y(\eta-\xi)}r(\eta-\xi) (1-\pa_y) \big(\Gamma(\eta,y) g(x-y))d\eta dy\\&=\int_{\mathbb{R}^2} e^{2\pi i y(\eta-\xi)} (1-\pa_y)(\Gamma(\eta,y)) r(\eta-\xi) g(x-y)d\eta dy\\&+
\int e^{2\pi i y(\eta-\xi)} r(\eta-\xi) \Gamma(\eta,y) g_2(x-y)d\eta dy,\end{aligned}\]
where $g_2(x)=1+\pa_x g $. Thus both terms are as required in \eqref{G1} and the claim follows from a direct
application of lemma~\eqref{basic2}

\end{proof}

\begin{lemma}\label{trivialbound}
Let $Q(x,\xi),\pa_x Q(x,\xi) \in L^2(\mathbb{R}^2)$. We define, for $a.e.x \in \mathbb{R}$,
\[ A_Q(x)=\int_{\R} e^{2\pi i\xi x} Q(x,\xi) d\xi.\]
Then,
\[ \|A_Q \|_{L^2} \le \| (1-\partial_x) Q \|_{L^2(\mathbb{R}^2)} .\]
\end{lemma}
\begin{proof} Let $v \in C_0^\infty(\mathbb{R})$ be a test function. Then we estimate $\|A_Q \|_{L^2}$  by duality.  Thus for $v \in L^2$,
Fourier inversion formula and the definition of $A_Q$ imply that
\begin{align*} &\int_{\R} A_Q(x) v(x)dx=\int_{\R}\int_{\R} e^{2\pi i\lambda x } A_Q(x) \hat{v}(\lambda)d \lambda dx\\
&=\int_{\R}\int_{\R}\int_{\R} e^{2\pi i(\lambda+\xi) x } Q(x,\xi) \hat{v}(\lambda)d \lambda dx d\xi.\end{align*}

 Now we use \eqref{partes1} and  integrate by parts in $x$ to  get
\[ \int_{\R} A_Q(x) v(x)dx=-\int_{\R}\int_{\R} e^{2\pi i\xi x}(1-\partial_x)Q(x,\xi) h_{\hat{v}(-\lambda)}(\xi,-x)  dxd\xi, \]

by direct application of the definition of $h_{\hat{v}}$ as defined in lemma~\ref{basic1}. A direct application of  Cauchy-Schwarz inequality in $\mathbb{R}^2$ and lemma~\ref{basic1} finishes the proof.
\end{proof}

In our proof we will use lemma~\ref{trivialbound}  for functions defined by integrals, e.g   \[Q(x,\xi)= \int_{\R^2}  e^{2\pi i\xi(x-y)+2\pi i\eta y}\tilde{Q}(x,y,\xi,\eta)  dy d\eta, \]
or \[Q(x,\xi)=\int_{\R^2}  e^{2\pi i\xi(x-y)}\tilde{Q}(x,y,\xi,\eta)  dy d\eta.\]

\paragraph{Proof of Theorem~\ref{CastroHwangComp}}

\begin{proof}

We start by   giving an explicit expression of $\Op(p_1)\circ \Op(p_2)f$,
\begin{align*}
\Op(p_1)\circ \Op(p_2)f=\int_{\R^3}e^{2\pi i(x\xi-\xi y +y\eta)}p_1(x,\xi)p_2(y,\eta)\hat{f}(\eta)d\eta dy d\xi.
\end{align*}
We bring in $p_1p_2$ by adding and subtracting suitable terms,
\begin{align*}
p_1(x,\xi)p_2(y,\eta)=&(p_1(x,\xi)-p_1(x,\eta))p_2(y,\eta)+\\
&p_1(x,\eta)(p_2(y,\eta)-p_2(x,\eta))
+p_1(x,\eta)p_2(x,\eta).\end{align*}
Therefore, we can write
\begin{align*}
\mathfrak{C}(p_1,p_2)f(x)=&\int_{\R^3}e^{2\pi i(x\xi-\xi y +y\eta)}(p_1(x,\xi)-p_1(x,\eta))p_2(y,\eta)\hat{f}(\eta)d\eta dy d\xi\\
+&\int_{\R^3}e^{2\pi i(x\xi-\xi y +y\eta)}p_1(x,\eta)(p_2(y,\eta)-p_2(x,\eta))\hat{f}(\eta)d\eta dy d\xi.
\end{align*}

Notice that the second term is zero (e.g use that as distributions, $\int_{\R} e^{2\pi i(x-y)\xi} d\xi =\delta(x-y)$).

Thus,
\begin{align*}
\mathfrak{C}(p_1,p_2)f(x)=&\int_{\R^3}e^{2\pi i(x\xi-\xi y +y\eta)}(p_1(x,\xi)-p_1(x,\eta))p_2(y,\eta)\hat{f}(\eta)d\eta dy d\xi,
\end{align*}
and we aim to bound it in $L^2$.  We express it directly as an operator  on $f$ itself:
\begin{align*}
\mathfrak{C}(p_1,p_2)f(x)=\int_{\R^4}e^{2\pi i(x\xi-\xi y +y\eta-\eta z)}(p_1(x,\xi)-p_1(x,\eta))p_2(y,\eta)f(z)dz d\eta dy d\xi.
\end{align*}
Now the basic formula (like \eqref{partes1}),
\[\frac{1}{1+2\pi i (y-z)}(1+\pa_\eta) e^{2\pi i (y-z)\eta}=e^{2\pi i (y-z)\eta},\] and an integration by parts in the $\eta$ variable, yields
\begin{align*}
&\mathfrak{C}(p_1,p_2)f(x)\\
&=\int_{\R^4}e^{2\pi i(x\xi-\xi y +y\eta-\eta z)}(1-\pa_\eta)\left\{(p_1(x,\xi)-p_1(x,\eta))p_2(y,\eta)\right\}\frac{f(z)}{1+2\pi i (y-z)}dz d\eta dy d\xi\\
&=\int_{\R^3}e^{2\pi i(x\xi-\xi y +y\eta)}(1-\pa_\eta)\left\{(p_1(x,\xi)-p_1(x,\eta))p_2(y,\eta)\right\} h_f(y,\eta) d\eta dy d\xi,
\end{align*}
where in the last equality we have absorbed the integral respect to  $z$ in the  definition of $h_f$ (Lemma~\ref{basic1}). Now we expand the $\eta$ derivative to express $\mathfrak{C}(p_1,p_2)f$ as a sum of three terms:
\begin{align*}
&=\int_{\R^3}e^{2\pi i(x\xi-\xi y +y\eta)}\left\{(p_1(x,\xi)-p_1(x,\eta))p_2(y,\eta)\right\}
h_f(y,\eta) d\eta dy d\xi\\
&-\int_{\R^3}e^{2\pi i(x\xi-\xi y +y\eta)}p_1(x,\xi)\pa_\eta p_2(y,\eta)h_f(y,\eta) d\eta dy d\xi\\
&+{\int_{\R^3}e^{2\pi i(x\xi-\xi y +y\eta)}\pa_\eta \left( p_1(x,\eta)p_2(y,\eta)\right)h_f(y,\eta) d\eta dy d\xi}\\&\equiv \mathfrak{C}(p_1,p_2)f_1+\mathfrak{C}(p_1,p_2)f_2+\mathfrak{C}(p_1,p_2)f_3.
\end{align*}

Notice that in fact, if we use again that $\int e^{2\pi i\xi(x-y)}d\xi =\delta(x-y)$, we obtain that

\[\mathfrak{C}(p_1,p_2)f_3=\int_{\R}e^{2\pi i\eta x} \pa_\eta\left(p_1(x,\eta))p_2(x,\eta)\right)h_f(x,\eta) d\eta. \]

We treat each of the above terms individually.
\begin{enumerate}
\item \underline{\emph{Estimation for $\mathfrak{C}(p_1,p_2)f_1$:}}

In order to estimate $\mathfrak{C}(p_1,p_2)f_1$ we integrate again by parts to obtain
\begin{align*}
&\mathfrak{C}(p_1,p_2)f_1=\int_{\R^3}e^{2\pi i(x\xi-\xi y +y\eta)}(1-\pa_\xi)\left\{(p_1(x,\xi)-p_1(x,\eta))p_2(y,\eta)\right\}\\
&\qquad\qquad\times \frac{h_f(y,\eta)}{(1+2\pi i (x-y))}dz d\eta dy d\xi\\
&=\int_{\R^3}e^{2\pi i(x\xi-\xi y +y\eta)}(p_1(x,\xi)-p_1(x,\eta))p_2(y,\eta) \frac{h_f(\eta,y))}{(1+2\pi i (x-y))}dz d\eta dy d\xi\\
&-\int_{\R^3}e^{2\pi i(x\xi-\xi y +y\eta)}\pa_\xi p_1(x,\xi)p_2(y,\eta)  \frac{h_f(\eta,y)}{(1+2\pi i (x-y))}dz d\eta dy d\xi\\
&\equiv \mathfrak{C}(p_1,p_2)f_{11}+\mathfrak{C}(p_1,p_2)f_{12}.
\end{align*}

\begin{enumerate}
\item \underline{\emph{Estimation for $\mathfrak{C}(p_1,p_2)f_{11}$:}}

We start by  we integrate by parts with respect to $y$ to bring a factor $\frac{1}{\eta-\xi}$ and thus a difference quotient for $p_1$;
\begin{align*}
&\mathfrak{C}(p_1,p_2)f_{11}(x)=\int_{\R^3}\frac{-1}{2\pi i (\xi-\eta)}\pa_y e^{2\pi i (x\xi-\xi y+y\eta)}
\frac{\left(p_1(x,\xi)-p_1(x,\eta)\right)\Gamma_f(y,\eta) }{1+2\pi i (x-y)} dy d\eta d\xi\\
&=\int_{\R^3}e^{2\pi i (x\xi-\xi y+y\eta)}Q(\xi,\eta,x)\pa_y\left(\frac{\Gamma_f(y,\eta)}{1+2\pi i (x-y)}\right)d\eta dy d\xi,
\end{align*}
where
\begin{equation*}
\Gamma_f(y,\eta)\equiv p_2(\eta,y)h_f(y,\eta), \qquad
Q(\xi,\eta,x)\equiv  \frac{\left(p_1(x,\xi)-p_1(x,\eta)\right)}{2\pi i (\xi-\eta)}.
\end{equation*}

The mean value theorem respect to $\xi$, tells us that
\begin{equation}\label{Qestimate}
\begin{aligned}
&\|Q|\|_{L^\infty(\mathbb{R}^3)}\le \|\partial_{\xi} p_1\|_{L^\infty(\mathbb{R}^2)} \le   \|\partial_{\xi} p_1\|_{1,0}\\
& ||\partial_xQ||_{L^\infty(\mathbb{R}^3)}\le \|\partial^2_{\xi\,x} p_1\|_{L^\infty(\mathbb{R}^2)} \le \|\partial_{\xi} p_1\|_{1,0}.
\end{aligned}
\end{equation}

Now a direct application of lemma~\ref{trivialbound} yields that

\[ \|\mathfrak{C}(p_1,p_2)f_{11}\|_{L^2} \le \|G(x,\xi)\|_{L^2(\mathbb{R}^2)} \]
where,
\begin{align*}
G(\xi;x)=\int_{\R^2}e^{2\pi i(\eta-\xi)y} (1-\pa_x)\left(Q(\xi,\eta,x)\pa_y\left(\frac{\Gamma_f(y,\eta)}{1+2\pi i (x-y)}\right)\right)d\eta dy .\end{align*}

Here we can not directly apply lemma~\ref{basic2} as $Q$ depends on $\eta$ but we follow a similar strategy.
We integrate by parts in $y$ to obtain that

\begin{align*}
G(\xi;x)=\int_{\R^2}e^{2\pi i(\eta-\xi)y} \frac{1}{1+2\pi i(\xi-\eta)}(1-\pa_x)\left(Q(\xi,\eta,x)(1-\pa_y)\pa_y\left(\frac{\Gamma_f(y,\eta)}{1+2\pi i (x-y)}\right)\right)d\eta dy.\end{align*}

Let us write $G(\xi;x)$ in the following way
\begin{align*}
G(\xi;x)=\int_{\R}\frac{1}{1+2\pi i (\eta-\xi)}Q^\sharp(\xi,\eta,x) d\eta,
\end{align*}
with
\begin{align*}
Q^\sharp(\xi,\eta,x)=\int_{\R}e^{2\pi i(\eta-\xi)y} (1-\pa_x)\left(Q(\xi,\eta,x)(1-\pa_y)\pa_y\left(\frac{\Gamma_f(y,\eta)}{1+2\pi i (x-y)}\right)\right)dy,
\end{align*}
By Cauchy-Schwarz
\begin{align*}
|G(\xi;x)|^2\leq C \int_{\R}|Q^\sharp(\xi,\eta,x)|^2 d\eta,
\end{align*}
and therefore,
\begin{align*}
||G||^2_{L^2(\mathbb{R}^2)}\leq C\int_{\R^3} |Q^\sharp (\xi,\eta,x)|^2d\eta d\xi dx.
\end{align*}
Our next task is to deal with  $Q^\sharp (\xi,\eta,x)$. We first expand the derivatives in $x$. Notice that
\begin{align*}
&(1-\pa_x)\left(Q(\xi,\eta,x)(1-\pa_y) \pa_y\left(\frac{\Gamma_f(\eta,y)}{1+2\pi i (x-y)}\right)\right)\\&=Q(\xi,\eta,x)(1-\pa_y)\pa_y\left(\Gamma_{f}(y,\eta)(1-\pa_x)\left(\frac{1}{1+2\pi i (x-y)}\right)\right)\\
&-\pa_x Q(\xi,\eta,x)(1-\pa_y)\pa_y\left(\frac{\Gamma_f(y,\eta)}{1+2\pi i (x-y)}\right).
\end{align*}
Then
\begin{align*}
&\int_{\R^3}|Q^\sharp(\xi,\eta,x)|^2d\eta d\xi dx\\
& = \int_{\R^3}\left|\int_{\R}e^{2\pi i (\eta-\xi)y}(1-\pa_x)\left(Q(\xi,\eta,x)(1-\pa_y)\pa_y\left(\frac{\Gamma_f(\eta,y)}{1+2\pi i (x-y)}\right)\right)dy\right|^2d\xi d\eta dx\\
&\leq ||Q||_{L^\infty(\mathbb{R}^3)}\\ &\int_{\R^3}\left|\int_{\mathbb{R}}e^{2\pi i (\eta-\xi)y}(1-\pa_y)\pa_y\left(\Gamma_f(y,\eta)(1-\pa_x)\left(\frac{1}{1+2\pi i (x-y)}\right)\right)dy\right|^2d\eta d\xi dx\\
&+||\pa_xQ||_{L^\infty(\mathbb{R}^3)}\int_{\R^3}\left|\int_{\R}e^{2\pi i (\eta-\xi)y}(1-\pa_y)\pa_y\left(\frac{\Gamma_f(y,\eta)}{1+2\pi i (x-y)}\right)dy\right|^2d\eta d\xi dx\\
&\equiv ||Q||_{L^\infty(\mathbb{R}^3)}I_1+||\pa_xQ||_{L^\infty(\mathbb{R}^3)}I_2.
\end{align*}
Now we expand the derivatives in $y$. We  obtain that both $I_1,I_2$ are a sum of terms of the type

\begin{equation*}
I_i=\int_{\R^3}\left|\int_{\R}e^{2\pi i (\eta-\xi)y}\pa^j_y {\Gamma_f(y,\eta)}g_i(x-y) dy\right|^2d\eta d\xi dx,
\end{equation*}
with $g_i \in L^2$ and $j=0,1,2$.
We proceed as in the proof of lemma~\ref{basic2}.  We first do  Plancherel in the $x$ variable and  then Fubini to integrate first
respect to $\xi$ and conclude by Plancherel again with real variable $y$ and Fourier variable $\xi$.
\[ \begin{aligned} I_i&= \int_{\R^3}|\hat{g}_i)(\alpha)|^2\left|\int_{\R}e^{2\pi i (\eta-\xi)y}\pa^j_y {\Gamma_f(y,\eta)} dy\right|^2 d\xi d\alpha d\eta\\&= \|g_i\|_{L^2}
\int  |\int e^{2 \pi i \xi y } \pa^j_y {\Gamma_f(y,\eta)}dy|^2 d\xi d\eta\\ &= \|g_i\|_{L^2} \int  |\pa^j_y {\Gamma_f(y,\eta)}|^2 dyd\eta \le
C \|p_2\|^2_{2,0} \|f\|^2_{L^2}, \end{aligned} \]
where the last inequality follows from a direct use of  of lemma~\ref{basic1} and the uniform bound for $\|g_i\|_{L^2}$.
Combined with \eqref{Qestimate} yields the desired bound,

\begin{equation} \label{bound1} \|G\|_{L^2(\mathbb{R}^2)} \le C \|p_2\|_{2,0} \|\partial_\xi p_1\|_{1,0} \|f\|_{L^2}.
\end{equation}

\item \underline{\emph{Estimation for $\mathfrak{C}(p_1,p_2)f_{12}$:}}

The estimate goes in a similar way to the previous one.  However we already have a derivative of the symbol,
so the first integration by part is not necessary. By applying   lemma~\ref{trivialbound} it holds that,

\[ \|\mathfrak{C}(p_1,p_2)f_{12}\|_{L^2} \le \|G(x,\xi)\|_{L^2 (\mathbb{R}^2)}. \]

where

\[G(x,\xi)= \int_{\R^3}e^{2\pi i y(\eta-\xi))}p_2(y,\eta)h_f(\eta,y) (1-\pa_x) \frac{\pa_\xi p_1(x,\xi)}{1+2\pi i (x-y)}  d\eta dy.  \]
Thus, we have to control terms of the form

\[ G_i(\xi,x)=  q_i(x,\xi)\int_{\R^2}e^{2\pi i(\eta-\xi)y}  \Gamma_f(y,\eta) g_i(x-y)d \eta dy, \]
where either $q_i(x,\xi)= \pa_{\xi} p_1,$ or $q_i=\pa_{x} (\pa_\xi p_1)$, and thus  $\|q_i\|_{L^\infty(\mathbb{R}^2)} \le \|\pa_\xi p_1\|_{1,0}$,
 $\|g_i\|_{L^2}$ is uniformly bounded and $\Gamma_f=p_2 h_f$.

Hence a direct application of lemma~\ref{basic3} yields that

\[ \|G_i(x,\xi)\|_{L^2(\mathbb{R}^2)} \le \|\pa_\xi p_1\|_{1,0} \| (1-\pa_y) \Gamma_f(\eta,y)\|_{L^2(\mathbb{R}^2)}. \]

Now $\Gamma_f$ is exactly as in the   lemma~\ref{basic1}.Thus

\begin{equation}\label{bound2} \|\mathfrak{C}(p_1,p_2)f_{12}\|_{L^2} \le C  \|\partial_\xi p_1\|_{1,0} \|p_2\|_{2,0} \|f\|_{L^2}. \end{equation}

 This finishes the estimate for $\mathfrak{C}(p_1,p_2)f_{12}$ and hence that of  $\mathfrak{C}(p_1,p_2)f_{1}$.

\end{enumerate}

\item \underline{\emph{Estimation of $\mathfrak{C}(p_1,p_2)f_{2}$:}}

In order to bound $\mathfrak{C}(p_1,p_2)f_{2}$ in $L^2$ we start by integrating by parts in $\xi$,

\begin{align*}
&\mathfrak{C}(p_1,p_2)f_{2}=\int_{\R^3}e^{2\pi i (x\xi-\xi y+y\eta)}p_1(x,\xi)\pa_\eta p_2(y,\eta)h_f(y,\eta) dy d\eta d\xi\\
&=\int_{\R^3}e^{2\pi i (x\xi-\xi y+y\eta)}(1-\pa_\xi)p_1(x,\xi)\pa_\eta p_2(y,\eta)\frac{h_f(y,\eta)}{1+2\pi i(x-y)} dy d\eta d\xi.
\end{align*}
 By lemma~\ref{trivialbound} we are  led to estimate in $L^2(\mathbb{R}^2)$ the function

\[G(x,\xi)= \int_{\R^2}e^{2\pi i y(\eta-\xi))}\pa_\eta p_2(y,\eta) h_f(y,\eta)(1-\pa_x)\frac{(1-\pa_\xi)p_1(x,\xi)}{1+2\pi i(x-y)}\big) dy d\eta. \]

Expanding the derivatives in $x$ and $\xi$ we discover that $G$ is a
 a sum of terms of the type

\[G_i(x,\xi)= q_i(x,\xi)  \int_{\R^2}e^{2\pi i y(\eta-\xi))} \Gamma_f(\eta,y) g_i(x-y) dyd\eta. \]
Here $g_i \in L^2$ uniformly,  $\Gamma_f=\pa_\eta p_2h_f$ and $q_i=\pa_{x,\xi}^{\alpha,\beta} p_1$ with $\alpha,\beta=0,1$. Thus
we have the uniform bound $\|q_i\|_{L^\infty} \le \|p_1\|_{1,1}$ as well.  Hence, a
direct application of lemma~\ref{basic3} yields the bound

\[ \|G_i(x,\xi)\|_{L^2(\mathbb{R}^2} \le  \|p_1\|_{1,1} \|(1-\pa_y) \Gamma_f \|_{L^2(\mathbb{R}^2)} . \]

Therefore lemma~\ref{basic1} (with $k=0,1$) applied to $\pa_\eta p_2$  yields

\begin{equation} \label{bound3} \|\mathfrak{C}(p_1,p_2)f_{2}\|_{L^2}  \le \ C ||p_1\|_{1,1} \| \pa_\xi p_2 \|_{1,0} \|f\|_{L^2}. \end{equation}

\item \underline{\emph{Estimation for $\mathfrak{C}(p_1,p_2)f_3$:}}

We denote  $M(x,\eta)=\pa_\eta (p_1(x,\eta) p_2(x,\eta))$ and $\tilde{M}(x,\eta)=(1-\partial_x)M(x,\eta)$.  Notice that,
by expanding the various derivatives, it holds that
\begin{equation}
\| \tilde{M}\|_{L^\infty(\mathbb{R}^2)} \le  \|\pa_\xi p_1 \|_{1,0} \|p_2\|_{1,1}+\|\pa_\xi p_2 \|_{1,0} \|p_1\|_{1,1}.
\end{equation}
Then

\[\mathfrak{C}(p_1,p_2)f_3(x)=\int_{\R} e^{2\pi i \eta x} M(x,\eta) h_f(x,\eta) d\eta.\]

Lemmas~\ref{trivialbound} and  \ref{basic1} gives

\begin{equation}\label{bound4}\begin{aligned} \|\mathfrak{C}(p_1,p_2)f_3\|_{L^2} & \le \| \tilde{M} h_f \|_{L^2(\mathbb{R}^2)} \le \|\tilde{M}\|_{L^\infty} \|f\|_{L^2}\\ & \le C  (\|\pa_\xi p_1 \|_{1,0} \|p_2\|_{1,1}+\|\pa_\xi p_2 \|_{1,0} \|p_1\|_{1,1})\|f\|_{L^2}.
 \end{aligned}
 \end{equation}

Finally,  by combining the bounds  \eqref{bound1},\eqref{bound2},\eqref{bound3}, \eqref{bound4}
we have achieved the conclusion of Theorem~\ref{CastroHwangComp} with norm,

\begin{equation}\label{constante}
\begin{aligned}
 ||| \mathfrak{C}(p_1,p_2)|||= & C( 2\|\partial_\xi p_1\|_{1,0} \|p_2\|_{2,0}+  \|p_1\|_{1,1} \| \pa_\xi p_2 \|_{1,0}\\ & +  \|\pa_\xi p_1 \|_{1,0} \|p_2\|_{1,1}+\|\pa_\xi p_2 \|_{1,0} \|p_1\|_{1,1}).
 \end{aligned}
 \end{equation}
  Thus, $||| \mathfrak{C}(p_1,p_2)||| \le C \left(\|p_1\|_{1,1}  \|\pa_\xi p_2\|_{1,0}+ \left(\|p_2\|_{1,1}+\|p_2\|_{2,0} \right) \|\pa_\xi p_1\|_{1,0} \right)$ as claimed.

\end{enumerate}

\end{proof}

\paragraph{ Proof of theorem~\ref{adjoints}}

\begin{proof}
Since $p(x,\xi)=p(x,-\xi)$ for $f,g \in \mathcal{S}$, the Schwarz class, it holds that

\[ \int_\R \Op(p)^{skew}f (x)g(x)  dx=\int_{\R^2} e^{2\pi i(x-y)\xi} \left(p(x,\xi)-p(y,\xi)\right) f(y) dyd\xi dx. \]
We  consider  the following smooth and fastly decaying approximation of the symbol,
$$p^{\delta,\kappa}(x,\xi)=e^{-\delta\xi^2} e^{-\kappa x^2} \varphi_\kappa * p(x,\xi) \in H^{k}$$
for every $k \in 0 \cup \mathbb{N}$. Here $\varphi_k$ is an standard approximation of the identity in the $x$ variable.
Since f,g are in the Schwarz class, by Dominated Convergence Theorem we have that,

\begin{equation}\label{aproximation} \lim_{\kappa \searrow 0, \delta \searrow 0}  \int_\R \Op(p^{\delta,\kappa})^{skew}f(x)g(x) dx=\int_\R \Op(p)^{skew}f(x) g(x) dx. \end{equation}

Therefore, we can  integrate by parts in $\xi$ to obtain,

\[\Op(p^{\delta,\kappa})^{skew}f(x)=\int_{\R^2} e^{2\pi i(x-y)\xi} \frac{\partial_\xi (p^{\delta,\kappa}(x,\xi)-p^{\delta,\kappa}(y,\xi) )}{2\pi i(x-y)} f(y) dyd\xi
=
\int_{\R^2}  e^{2\pi i(x-y)\xi} Q(x,y,\xi)  f(y)dyd\xi,\]

where $$Q(x,y,\xi)=\frac{\partial_\xi (p^{\delta,\kappa}(x,\xi)-p^{\delta,\kappa}(y,\xi) )}{2\pi i(x-y)}.$$
Thus, by Lemma~\ref{trivialbound},

\begin{equation}\label{Skewtrivial} \|\Op(p^{\delta,\kappa})^{skew}f \|_{L^2} \le \|G\|_{L^2(\mathbb{R}^2)}, \end{equation}

where
$$G(x,\xi)=\int_{\R}e^{-2\pi i y\xi} (1-\pa_x)Q(x,y,\xi)f(y)dy.$$

Now for $2\pi iq^{\delta,\kappa}(x,\xi)=\partial_\xi p^{\delta,\kappa}(x,\xi)$, the basic properties of the Fourier transform yield that,

\[ Q(x,y,\xi)=\frac{q^{\delta,\kappa}(x,\xi)-q^{\delta,\kappa}(y,\xi)}{ (x-y)}=\int \widehat{\partial_x q^{\delta,\kappa}}(\eta,\xi) \frac{e^{i2\pi\eta x}-e^{i2\pi\eta y}}{2\pi i \eta(x-y)}d\eta .\]

Thus, if we declare $\psi(\eta,x-y)= (1-\partial_x)\frac{e^{i2\pi\eta (x-y)}-1}{2\pi i \eta(x-y)} $, it holds that

\[G(x,\xi)=\int_{\R^2} e^{2\pi i(-y\xi+\eta y)} \psi(\eta,x-y) \widehat{\partial_x q^{\delta,\kappa}}(\eta,\xi)f(y)dyd\eta.\]

Next, we compute the Fourier transform of  $G(x,\xi)$ respect to $x$,  denoted  by $\widehat{G}(\alpha,\xi)$, and change variables in $x-y$. We obtain
the formula,

\[ \begin{aligned}
\widehat{G}(\alpha,\xi)&=\int_{\R^2} e^{i y (\eta-\xi-\alpha)} \widehat{\psi}(\eta,\alpha)  \widehat{\partial_x q^{\delta,\kappa}}(\eta,\xi)f(y) dyd\eta\\
& =\int_{\mathbb{R}} \frac{b(\eta)}{b(\eta)} \widehat{\psi}(\eta,\alpha) \widehat{\partial_x q^{\delta,\kappa}}(\eta,\xi) \widehat{f}(-\eta+\xi+\alpha) d \eta, \end{aligned} \]

where $b(\eta)$ is an auxiliary function, which will be specified later, introduced to bargain differentiability into integrability,  Now Cauchy Schwarz  yields the pointwise estimate,

\[ |\widehat{G}(\alpha,\xi)| \le \left(\int_{\mathbb{R}} |b(\eta) \widehat{\partial_x q^{\delta,\kappa}}(\eta,\xi)|^2 d\eta\right)^{\frac12} \left(\int_{\mathbb{R}} \frac{1}{b(\eta)} |\widehat{\psi}(\eta,\alpha)|^2
| \widehat{f}(\eta-\xi-\alpha)|^2 d \eta\right)^\frac12. \]

Thus, for $C(p)=\sup_{\xi}  \int_{\mathbb{R}} |b(\eta) \widehat{\partial_x q^{\delta,\kappa}}(\eta,\xi)|^2 d\eta$, it holds that
 \[\begin{aligned}
 \|G\|_{L^2(\mathbb{R}^2)}^2  &\le C(p)  \int_{\mathbb{R}^3} \frac{1}{|b(\eta)|^2} |\widehat{\psi}(\eta,\alpha)|^2
| \widehat{f}(\eta-\xi-\alpha)|^2 d \eta d\xi d\alpha \\&=C(p)\|f\|_{L^2}^2  \int_{\mathbb{R}^2} \frac{1}{|b(\eta)|^2} |\widehat{\psi}(\eta,\alpha)|^2
 d \eta d\alpha =C(p)\|f\|_{L^2}^2  \int_{\mathbb{R}^2} \frac{1}{|b(\eta)|^2} |{\psi}(x,\eta)|^2
 d \eta dx. \end{aligned}\]

 Now, since  $\int_{\mathbb{R}} \left|\frac{e^{ix}-1}{x}\right|^2+\left|\pa_x \left(\frac{e^{ix}-1}{x}\right)\right|^2 dx  \le C$, it holds that
 \[ \int_{\mathbb{R}} |\psi(x,\eta)|^2dx \le C \left( |\eta|+|\eta|^{-1}\right).\]
Therefore, by Fubini,
 \[\int_{\mathbb{R}^2} \frac{1}{|b(\eta)|^2} |{\psi_2}(x,\eta)|^2
 d \eta dx \le C \int_{\mathbb{R}} \frac{|\eta|+|\eta|^{-1} }{|b(\eta)|^2} d\eta. \]
 This last expression,  is integrable for every $0<\ep<1$ if we take $b(\eta)=\eta^{-\ep} \chi_{[0,1]}(|\eta|)+(1-\chi_{[0,1]}. (|\eta|)) \eta^{1+\ep}.$

Hence  inserting the bound of $\|G\|_{L^2}$ in \eqref{Skewtrivial} we obtain that

\[ \|\Op(p^{\delta,\kappa})^{skew}f \|_{L^2} \le C \|f\|_{L^2}, \]

with $C=C(p^{\delta,\kappa} )=\sup_{\xi}  \int_{\mathbb{R}} |b(\eta) \widehat{\partial_x q^{\delta,\kappa} (\eta,\xi)}|^2 d\eta$. Given our choice of $b(\eta)$, it holds
that,

\[ C(p) \le \sup_\xi \left(\|\pa_x q^{\delta,\kappa}\|_{H^{1+\ep}}+\|\pa_x q^{\delta,\kappa}\|_{\dot{H}^{-\ep}}\right)=\sup_{\xi} \left(\|\partial^2_{x\xi} p^{\delta,\kappa} \|_{H^{1+\ep}}+\|\partial^2_{x\xi} p^{\delta,\kappa} \|_{\dot{H}^{-\ep}}\right).\]
However,  setting $p^{\kappa}= e^{-\kappa x^2} \varphi_k * p$, it holds that
\[\sup_{\xi} \|\pa_x q^{\delta,\kappa}\|_{\dot{H}^{-\ep}}=\sup_{\xi} \left(2\delta|\xi|e^{-\delta \xi^2}\|\pa_x p^\kappa\|_{\dot{H}^{-\ep}}\right)+\sup_{\xi}\left( e^{-\delta \xi^2}\|\pa_x \pa_\xi p^{\kappa}\|_{\dot{H}^{-\ep}}\right). \]

Notice that  $p^{\kappa} \in L^2$ implies that  $\pa_x p^{\kappa}  \in \dot{H}^{-\ep} $. Hence we can take first the limit $\delta \searrow 0$ to get
rid of the term $\|\pa_x p^\kappa\|$.  Then, continuity of the Sobolev norms respect to mollifiers allows us  to let $\kappa$  go to  $0$, to obtain the  bound
$\sup_{\xi}\|\pa_x \pa_\xi  p\|_{\dot{H}^{-\ep}}$

Arguing exactly in the same way with  the $ \dot{H}^{1+\ep}-$ term, in combination with \eqref{aproximation} yields the desired,

\begin{align*}
\|Op(p)^{skew} \|_{L^2 \to L^2}\leq \sup_{\xi}\left(\|\pa_x \pa_\xi p \|_{\dot{H}^{-\ep}}+\|\pa_x \pa_\xi p\|_{H^{1+\ep}}\right).
\end{align*}
The proof is finished.
\end{proof}
\paragraph{Proof of lemma~\ref{positive}}
\begin{proof}
\[\begin{aligned}
-\int_{\mathbb{R}} f\Op(p_+)fdx&= -\int_{\mathbb{R}} f\Op(p_+^\frac12 p_+^\frac12)fdx\\&=-\int_{\mathbb{R}} f\Op(p_+^\frac12)\circ \Op(p_+^\frac12)fdx
-\int_{\mathbb{R}} fC(p_+^\frac12,p_+^\frac12)fdx.
\end{aligned}\]
Cauchy-Schwarz and Theorem~\ref{CastroHwangComp} imply that
\begin{equation}
| \int_{\mathbb{R}} f\mathfrak{C}(p_+^\frac12,p_+^\frac12)fdx| \le |||\mathfrak{C}(p_+^\frac{1}{2},p_+^\frac{1}{2}) |||\|f\|_{L^2}^2.
\end{equation}
For the first, notice that $$\int \Op(p_+^\frac12)f \Op(p_+^\frac12)f dx=\int |\Op(p_+^\frac12)f|^2dx>0.$$ Thus,
\[ \begin{aligned}-\int f \Op(p_+^\frac12)\circ \Op(p_+^\frac12)fdx&=
-\int \Op(p_+^\frac12)^Tf \Op(p_+^\frac12) f dx \\ &\le -\int \Op(p_+^\frac12)^Tf\Op(p_+^\frac12) f dx+\int (\Op(p_+^\frac12)f\Op(p_+^\frac12) fdx
\\ & = \int  [\Op(p_+^\frac12)^Tf-\Op(p_+^\frac12) f]\Op(p_+^\frac12) fdx \\ &\le\| \Op(p_+^\frac12) (f)\|_{L^2}
\|\Op(p_+^\frac12)^{skew} f\|_{L^2}. \end{aligned}
\]
The claim follows from  theorem~\ref{adjoints} and theorem~\ref{Hwan}.

\end{proof}

\subsection{Lemmas for the apriori estimate}

\subsubsection{Transport term}

Recall that $D$ is the operator associated with the symbol $d(\xi)=1+2\pi i t|\xi|$ and $\mathcal{J}$ is associated
with $j(\xi)=e^{\int_0^{t|\xi|}\varphi(\tau)d\tau}\hat{f}(\xi)$ where $\varphi$ was defined in \eqref{varphi}.

In order to deal with the transport term we need the following lemma which states that $\mathcal{J}$  and $D$
have a similar behaviour.

\begin{lemma}\label{ftoc} There exists $\phi$ such that
\[j(t|\xi|)=1+t |\xi|\phi(t|\xi|)\]
\end{lemma}
satisfying
\begin{align*}
||\phi(t|\xi|)||_{L^\infty}\leq C,\quad
|||\xi|\pa_{\xi}\left(\phi(t|\xi|)\right)||_{L^\infty}\leq C,
\end{align*}
where $C$ does not depend on $t$.
\begin{proof}
Notice that by the fundamental theorem of Calculus
\[ j(t|\xi|)=1+t |\xi| \int_0^1 \varphi(st|\xi|)e^{\int_{0}^{st|\xi|}\varphi(\tau)d\tau}ds. \]

Thus,
\begin{align*}
\phi(t|\xi|)=\int_{0}^1\varphi(st|\xi|)e^{\int_{0}^{st|\xi|}\varphi(\tau)d\tau}ds.
\end{align*}
With these representations the claimed properties follow readily.

\end{proof}

\begin{lemma}\label{transport}Let $\pa_{x}a \in H^{1+\ep}(\mathbb{R})$ for $\epsilon>0$ and
$\mathcal{J}^{-1}(f) \in L^2$. Then,
\begin{equation} |\int \mathcal{J}^{-1}f \mathcal{J}^{-1}(af_x) dx| \le C||\pa_x a ||_{H^{1+\ep}} \|\mathcal{J}^{-1}(f)\|_{L^2}^2.\end{equation}
\end{lemma}

\begin{proof}

In order to bring in a suitable commutator we
first  notice that
\begin{equation}\label{cuadrado} \int_{\mathbb{R}}  \mathcal{J}^{-1}(f) a\pa_x\mathcal{J}^{-1}(f) dx= \frac12 \int_{\mathbb{R}}   \int a\pa_x|\mathcal{J}^{-1}(f)|^2 dx \end{equation}
and thus, integrating by parts
\begin{equation}\label{easy} |\int_{\mathbb{R}}   \mathcal{J}^{-1}f  a \pa_x\mathcal{J}^{-1} f dx| \le \|a_x\|_{L^\infty} \|\mathcal{J}^{-1}(f)\|^2_{L^2} \end{equation}
and since $ \|a_x\|_{L^\infty} \le \|a_x\|_{H^{1+\epsilon}}$, we conclude that $\int  \mathcal{J}^{-1}f a\pa_x\mathcal{J}^{-1}f dx$ is a harmless term. Thus we can subtract it to the  transport term and  we are led to bound the commutator,
\[  [\mathcal{J}^{-1},a][f_x]. \]

Let $g=\mathcal{J}^{-1}f$ so that $f=\mathcal{J}g$. Let $\psi=\phi(t|\xi|)|\xi|$. Then

\[ a f_x= a \mathcal{J}g_x= ag_x+t \Psi(ag_x)+ t [a,\Psi] (g_x)=\mathcal{J} (ag_x)+t [a,\Psi] (g_x), \]

\[ [\mathcal{J}^{-1},a]f_x= ag_x-\mathcal{J}^{-1}(af_x)= t \mathcal{J}^{-1}[a,\Psi](g_x). \]

We iterate this   trick once more. Let  us denote $G=t\mathcal{J}^{-1}g_x$, which has  $L^2$-norm bounded by  $||g||_{L^2}$. Then
\[[\mathcal{J}^{-1},a]f_x= [a,\Psi] t g_x= [a,\Psi] G+ t\Psi ([a,\Psi])(G)+t[[a,\Psi],\Psi] (G)= \mathcal{J}([a,\Psi] G)+t [[a,\Psi],\Psi] (G) \]
and therefore

\[ t \mathcal{J}^{-1}[a,\Psi](g_x)= [a,\Psi] G+t \mathcal{J}^{-1}[[a,\Psi],\Psi] (G). \]

Now notice that
\[ [[a,\Psi],\Psi]= a(\Psi)^2-2\Psi(a\Psi)+(\Psi)^2 a \]
and then
\begin{align*}& \widehat{[a,\Psi]}G(\xi)= \int \hat{a}(\eta)\hat{G}(\xi-\eta)
\left(|\xi|\phi(t|\xi|)-|\xi-\eta|\phi(t|\xi-\eta|)\right) d\eta\\
&\widehat{[[a,\Psi],\Psi]}G(\xi)=\int \hat{a}(\eta)\hat{G}(\xi-\eta)
\left(|\xi|\phi(t|\xi|)-|\xi-\eta|\phi(t|\xi-\eta|)\right)^2 d\eta.\end{align*}

Thus by the mean value theorem and  lemma \ref{ftoc}
\[ \left||\xi|\phi(t|\xi|)-|\xi-\eta|\phi(t|\xi-\eta|)\right| \le C |\eta|,  \]
Thus
\[ \|[a, \Psi ]G\|_{L^2} + \|[[a, \Psi ],\Psi]G\|_{L^2} \le \int_{\mathbb{R}}| \int_{\mathbb{R}} |\hat{G}(\xi-\eta)| (|\widehat{\pa_{x} a}(\eta)|+|\widehat{\pa_{xx} a}(\eta)|) (1+|\eta|)^\ep(1+|\eta|)^{-\ep} d\eta |^2d\xi \]
and we conclude by H\"older inequality in the $\eta$ variable (Recall that$\|G\|_{L^2} \le \|g \|_{L^2}$).

 \end{proof}

\subsubsection{Commutator  between $\Op(p)$ and $\mathcal{J}$}

Similar computations to the above allow us to interchange $\mathcal{J}$ and $\Op(p)$. In order to simplify  the proof
we will first relate $D$ with $\Op(p)$. Then we  use our commutator estimation theorem~\ref{CastroHwangComp} to transfer the result to $\mathcal{J}$ to finish the estimate.

\begin{lemma}\label{ComDp}
Let $t \pa_x p \in \Hw$ and $g \in L^2$. Then,

\[ \|D^{-1}\Op(p) D g-\Op(p)g \|_{L^2} \le \|t \pa_x p\|_{1,1}\|g\|_{L^2}. \]

\end{lemma}
\begin{proof}
By the definition of $D$,

\[ \begin{aligned}  \Op(p) (Dg)&=\Op(p)(g)+t \Op(p)(\pa_x g)\\ &=D(\Op(p)(g))+t[\pa_x,\Op(p)](g)=D(\Op(p)(g))+t\Op(\pa_x p)(g). \end{aligned}\]
Hence
\[ D^{-1}\Op(p) D {g}-\Op(p)g = D^{-1}(\Op(t \pa_x p) (g)).  \]

Thus, taking $L^2$ norms and using theorem~\ref{Hwan} for the symbol $t\pa_x p$, and
that  $d^{-1}$ is bounded in $L^\infty$ the claim is straightforward.

\end{proof}

\begin{thm}\label{ComJp} Suppose that $t\pa_x p,\, tp \in \Hw$ and $\partial_\xi p \in S_{1,0}$. Let  $g \in L^2$, then

\[ \| \mathcal{J}^{-1}\Op(p) \mathcal{J}(g) -\Op(p) g \|_{L^2} \le C(p) \|g\|_{L^2},\]
where $C(p)= C(\|t \pa_x p \|_{1,1}+ \|t p\|_{1,1} +\| \pa_\xi p\|_{1,0}). $
\end{thm}

\begin{proof}
Define  $\tilde{g}=D^{-1} \mathcal{J} g$  and observe that $\|\tilde{g}\|_{L^2} \le \|g\|_{L^2}$.
We write $\Op(p) g=\Op(p) \mathcal{J}^{-1}D \tilde{g}, \Op(p) \mathcal{J}g= \Op D \tilde{g}$ and sum and subtract $\mathcal{J}^{-1}D \Op(p) \tilde{g}$. Then

\[ \begin{aligned}\mathcal{J}^{-1} \Op(p) \mathcal{J}g-\Op(p)(g)&= \mathcal{J}^{-1}D\big(D^{-1}\Op(p) D \tilde{g}-\Op(p)\tilde{g}\big)\\ &+
[\Op(p), \mathcal{J}^{-1}D] \tilde {g}.
\end{aligned} \]

In order to deal with the first term readily notice that $\mathcal{J}^{-1}D$ has a bounded Fourier multiplier and
 lemma~\ref{ComDp} implies that

 \begin{equation}
 \|D^{-1}\Op(p) D \tilde{g}-\Op(p)\tilde{g}\|_{L^2} \le \|t \pa_x p \|_{1,1} \|g\|_{L^2}.
 \end{equation}

  For the second,   notice that since $\mathcal{J}^{-1}D$ has a symbol $m=j^{-1}d$ independent of $x$ then $Op(p)\circ \mathcal{J}^{-1}D=Op\left(p\cdot m\right)$. Thus

  $$[ \mathcal{J}^{-1}D, Op(p)] =\mathfrak{C}(m,p)
.$$

Therefore we can estimate $||\mathfrak{C}(m,p)||_{L^2\to L^2}$ by theorem \ref{CastroHwangComp},

\begin{align*}
|||\mathfrak{C}(m,p)|||\leq C\left(\|m\|_{0,1}  \|\pa_\xi p\|_{1,0}+ \left(\|p\|_{1,1}+\|p\|_{2,0} \right) \|\pa_\xi m\|_{0,0} \right).
\end{align*}

Lemma \ref{ftoc} implies that  $m\in L^\infty$ and $\|\partial_\xi m\|_{L^\infty} \le C t$. Therefore, we achieve the conclusion of the lemma by the assumptions on $p$.

\end{proof}

\

\section{Mixing solutions in the stable regime}\label{stable}
As discussed in the introduction our work was motivated by \cite{laslo} where it is shown that in the case of horizontal interface there exists subsolutions in the unstable regime but
it seems imposible to find them in the stable regime and perhaps they do not exist.  Surprisingly, if the flat interface is not horizontal then one can construct mixing solutions with a straight initial interface in both the fully stable and the fully unstable regime.  The proof runs along similar steps than  the one  in \cite{laslo}. Even if  we will need the machinery expose in section \ref{Hprinciple} to carry out this construction this section only expect to be a remark.

Let's consider the change of variables $\xb(s,\lambda)=s\tb+\nb \lambda$, with $\tb=\frac{(\mu_1,\mu_2)}{\sqrt{\mu_1^2+\mu_2^2}}$, $\mu_1\geq 0$ and $\mu_2\in \R$. We declare $\ep=ct$, with $c>0$ and $\Omega_{mix}=\{ \xb\in \R^2\,:\, \xb=\xb(s,\lambda),\quad s\in\R, \quad -\ep(t)<\lambda<\ep(t)\}$. We define $\rho$, $\ub$ and $\mb$ through $\rho^\sharp=-\sign(\sigma)\frac{\lambda}{\ep}$, $\ub^\sharp=-\frac{\mu_2}{\sqrt{\mu_1^2+\mu_2^2}} \rho^\sharp \tb$ and $\mb^\sharp=\rho^\sharp \ub^\sharp-\gamma^\sharp\left(1-\left(\rho^\sharp\right)^2\right)\nb -\frac{1}{2}\left(1-\left(\rho^\sharp\right)^2\right)(0,1)$, with $\gamma^\sharp\in \R$. Here $\sigma>0$ yields an initial data in the stable regime and $\sigma<0$ an initial data in the unstable regime. Then
\begin{align*}
\nabla f (x(s,\lambda))= & \tb \pa_s f^\sharp+\nb \pa_\lambda f^\sharp\\
\nabla\cdot  \fb (x(s,\lambda)) = & \tb\cdot \pa_s \fb^\sharp +\nb \cdot \pa_\lambda \fb^\sharp.
\end{align*}
Using this formulas is easy to check that $\nabla \cdot \ub =0$, $\nabla^\perp \cdot\ub =-\pa_{x_1}\rho$ and $\ub\cdot\nabla \rho=0$. In addition, the equation $\pa_t \rho +\nabla \cdot \mb=0$ transforms to $\gamma^\sharp =\frac{1}{2}\left(\frac{\mu_1}{\sqrt{\mu_1^2+\mu_2^2}}+\sign(\sigma) c\right)$

If $\sigma<0$ then we obtain from \eqref{convex12} the constrain $0<c<1+\frac{\mu_1}{\sqrt{\mu_1^2+\mu_2^2}}$. If $\sigma>0$ we obtain $0<c<1-\frac{\mu_1}{\sqrt{\mu_1^2+\mu_2^2}}$, what give rise to a mixing solution in the stable regime but if the interface is flat and horizontal.

At this point, it is convenient to notice that the case non horizontal and flat interface is only stable in the sense of the Muskat curve-evolution equation. In the hydrodynamical context this configuration is unstable even if the lighter fluid is above, because it leads to an instantaneous velocity shear layer (i.e. discontinuity in the velocity). The only hydrodynamically stable configuration seems to be the flat horizontal interface with the lighter fluid above.

\section*{Acknowledgements}

AC, DC and DF were partially supported by ICMAT Severo Ochoa projects SEV-2011-0087 and SEV-2015-556 and by the Spanish Ministry of Science and Innovation, through the "Severo Ochoa Program for Centre of Excellence in R\&D" (CEX2019-00904-S)".
AC and DC were partially supported by the grant MTM2014-59488-P and MTM2017-89976-P (Spain).  AC and DF were partially supported  by the ERC grant 307179-GFTIPFD. AC was partially supported by the Ram\'on y Cajal program RyC-2013-14317 and  the Europa Excelencia program ERC2018-092824 (Spain). DC was partially supported by the ERC grant 788250-NONFLU. DF was partially supported by the grants MTM2014-57769-P-1, MTM2017-85934-C3-2-P(Spain) and ERC grant 834728-QUAMAP.  We are very thankful to Fabricio Mac\'ia for suggesting
the  semiclassical interpretation and to Mikko Salo for illuminating discussions on the classical theory.

\appendix
\section{Appendix}\label{Linearization and lower order terms}

In this appendix we will prove lemma \ref{ordenmasalto}, lemma \ref{descomposicion} and the required estimates
for the velocity $u$ and for the coefficient of the transport term $a$. Throughout the whole section, there are integrals
which are interpreted in the principal value sense, both at $0$ and $\infty$. Since, this is standard and harmless in our context, we will not make it  explicit.

\subsection{Lower order terms. Proof of lemma~\ref{ordenmasalto}}\label{ordenmasaltosection}

We can write
\begin{align*}
\mathcal{M} u =-\frac{1}{4\pi}\int_{-1}^1\int_{-1}^1 \int_{\R} k_{\theta}(x,y)\pa_x \theta dy d\lambda'd\lambda.
\end{align*}
 The main part of the proof of the lemma  will be showing that
\begin{align}
&\left|\left|\mathcal{D}^{-1}\left(\pa^5_x \mathcal{M} u+\frac{1}{4\pi}\int_{-1}^1\int_{-1}^1 \int_{\R} k_{\theta}(\cdot,y)\pa^6_x \theta dy d\lambda'd\lambda\right)\right|\right|_{L^2}\nonumber\\&\leq C\left(||f||_{H^4},||\mathcal{D}^{-1}\pa^5_x f||_{L^2}\right)\label{loquesea}.
\end{align}

 In order to accomplish this,  we will need to compute the derivatives of the function
\begin{align*}k_\theta(x,y)=\frac{y}{y^2+\theta^2},\end{align*}
where
\begin{align*}
\theta=\Delta f(x,x-y)+\ep(x)\lambda-\ep(x-y)\lambda'=\Delta f+\Delta \ep \lambda' +\ep(x)(\lambda-\lambda').
\end{align*}
In addition we introduce
\begin{align*}
h= f(x)+\lambda\ep(x)\quad \text{and} \quad h'=f(x)+\lambda'\ep(x),\gamma=\ep(x)(\lambda-\lambda').
\end{align*}
Thus,
\begin{align*}
\theta=\Delta h' +\ep(x)(\lambda-\lambda')=\Delta h'+\gamma,
\end{align*}
where we remark that $\gamma$ depends on $x$ and on $t$ although we will not make this dependence explicit. We recall that $c(x,t)$ is as in the statement of theorem \ref{existencialocal}. Since $\ep(x,t)=c(x,t)t$, with $0<\overline{c}<c(x,t)$ and $||c(x,t)||_{C^5}\leq C$,

\begin{equation}
 \|\gamma \|_{C^5} \le C t|\lambda-\lambda'|.
\end{equation}

A big part of our proof will be based on comparing $\theta$ with its linearized version

\begin{align*}
\theta_{lin}=\pa_x h'(x) y+\gamma.
\end{align*}

\begin{rem}
Notice that in order to obtain selfadjoint pseudodifferential operators we could deal as well with
\[ \theta_{lin}^W=\pa_x h'(\frac{x+y}{2}) y+\gamma, \]
Here the $W$ in the exponent stands for the Weyl quantization. We do not pursue this issue here.
\end{rem}

To make the notation even more compact we will write $\Delta f=\Delta f(x,x-y)$ and $\pa^k_x f(x)=\pa^k_x f$. Then we have
\begin{align*}
&\pa_x k_\theta(x,y)=-2\frac{y\theta\pa_x\theta}{(y^2+\theta^2)^2}\equiv -2k_\theta^{11}(x,y),
\end{align*}

\begin{align*}
&\pa^2_x k_\theta(x,y)=-2\frac{y((\pa_x\theta)^2+(\theta)\pa_x^2\theta)}{(y^2+(\theta)^2)^2}
+8\frac{y(\left(\theta)\pa_x\theta\right)^2}{(y^2+(\theta)^2)^3}\equiv c_{21} k_{\theta}^{21}(x,y)+c_{22} k_\theta^{22}(x,y),
\end{align*}

\begin{align}\label{k3}
&\pa_x^3 k_\theta(x,y)=-2y \frac{3\pa_x\theta\pa^2_x\theta +(\theta)\pa^3_x\theta}{(y^2+(\theta)^2)^2}+24 y \frac{(\theta)\pa_x\theta((\pa_x\theta)^2+(\theta)\pa^2_x\theta)}{(y^2+(\theta)^2)^3}\nonumber-48y\frac{(\theta)^3(\pa_x\theta)^3 }{(y^2+(\theta)^2)^4}\nonumber \\
&\equiv c_{31}k_\theta^{31}(x,y)+c_{32}k_\theta^{32}(x,y)+c_{33}k_\theta^{33}(x,y),
\end{align}

\begin{align*}
&\pa^4_x k_\theta(x,y)=384y\frac{(\theta)^4(\pa_x\theta)^4}{(y^2+(\theta)^2)^5}-288y\frac{(\theta)^2(\pa_x\theta)^2((\pa_x \theta)^2+(\theta)
\pa_{x}^2\theta)}{(y^2+(\theta)^2)^4}\\&+y\frac{24(\pa_x\theta)^4+144(\theta)(\pa_x\theta)^2\pa^2_x\theta +24(\theta)^2(\pa^2_x\theta)^2}{(y^2+(\theta)^2)^3}+y\frac{32(\theta)^2\pa_x\theta \pa^3_x \theta}{(y^2+(\theta)^2)^3}-y\frac{6(\pa_x^2\theta)^2+8\pa_x\theta\pa^3_x\theta-2(\theta)\pa^4_x\theta)}{(y^2+(\theta)^2)^2}\\
&\equiv c_{41}k_\theta^{41}(x,y)+c_{42}k_\theta^{42}(x,y)+c_{43}k_\theta^{43}(x,y)+c_{44}k_\theta^{44}(x,y),
\end{align*}

\begin{align*}
&\pa^5_x k_\theta(x,y)=-3840y\frac{(\theta)^5(\pa_x\theta)^5}{(y^2+(\theta)^2)^6}+3840y\frac{(\theta)^3(\pa_x\theta)^5+(\theta)^4(\pa_x \theta )^3\pa^2_x\theta}{(y^2+(\theta)^2)^5}\\
&+y\frac{240 (\pa_x\theta)^3\pa^2_x\theta+360(\theta)\pa_x \theta (\pa_x^2\theta)^2+240(\theta)(\pa_x\theta)^2\pa_x^3\theta }{(y^2+(\theta)^2)^4}
+y\frac{80(\theta)^2\pa^2_x\theta\pa^3_x\theta+40(\theta)^2\pa_x\theta \pa^4_x\theta}{(y^2+(\theta)^2)^3}\\
&-y\frac{10\pa_x\theta \pa^4_x \theta +2(\theta )\pa^5_x \theta}{(y^2+(\theta)^2)^2}\equiv c_{51}k_\theta^{51}(x,y)+c_{52}k_\theta^{52}(x,y)+c_{53}k_\theta^{53}(x,y)+c_{54}k_\theta^{54}(x,y)+c_{55}k_\theta^{55}(x,y),
\end{align*}
where we notice that the numbers $c_{ij}$ $i,j=1,2,3,4,5$ are harmless coefficients.
Then, by applying Minkowski inequality, we need to bound
\begin{align*}
\sum_{j=1}^5\left|\left|\mathcal{D}^{-1}\int_{\R}\pa^j_xk_{\theta}(\cdot,y)\pa^{5-j}_x\pa_x\theta dy\right|\right|_{L^2}\\
=\sum_{j=1}^5\sum_{i=1}^j\left|\left|\mathcal{D}^{-1}\int_{\R}k^{ji}_{\theta}(\cdot,y)\pa^{5-j}_x\pa_x\theta dy\right|\right|_{L^2}
\end{align*}
independently of $t$, $\lambda$ and $\lambda'$. The highest order terms in this sum are given by
\begin{align}\label{safe}
\mathcal{D}^{-1}\int_{\R}k^{11}_{\theta}(x,y)\pa^5_x\theta dy \quad \text{and}\quad \mathcal{D}^{-1}\int_{\R}k^{55}_{\theta}(x,y)\pa_x\theta dy.
\end{align}
Since  there are 5 derivatives of the function $\theta$ in both terms we have to use the operator $D^{-1}$. Since $D^{-1}$ is bounded in $L^2$
it holds that
\begin{align}\label{easyb}
&\sum_{j=2}^5\sum_{i=1}^{\min(j,\,4)}\left|\left|\mathcal{D}^{-1}\int_{\R}k^{ji}_{\theta}(\cdot,y)\pa^{5-j}_x\pa_x\theta dy\right|\right|_{L^2}\\ &\leq \sum_{j=2}^5\sum_{i=1}^{\min(j,\,4)}\left|\left|\int_{\R}k^{ji}_{\theta}(\cdot,y)\pa^{5-j}_x\pa_x\theta dy\right|\right|_{L^2},\nonumber
\end{align}
which make the computation easy. We first deal with the sum \eqref{easyb} and finally with \eqref{safe}, which are somewhat more delicate.

We will use the following convection. We will write:
\begin{enumerate}
\item $\precsim$ meaning "is bounded  in absolute value by ".
\item $f(x)\sim g(x)$ if $||f-g||_{L^2}\leq \lar$.

\item We will denote by $k^{ji}(x)$ the integral $$\int_{\R}k^{ji}_\theta(x,y)\pa^{6-j}_x\theta dy.$$
\item $C_{k+\alpha}$ will be a constant depending on $||f||_{C^{k+\alpha}}$,  with $k$ an integer and $0\leq \alpha<\frac{1}{2}$. $C_{k+\alpha}(x)$ will be a function whose $L^\infty-$norm is bounded by a constant depending on $||f||_{C^{k+\alpha}}$.
\item Given an integral $\int_{\R}f(x,y)dy$ we will  estimate separately, $\int_{|y|>1}f(x,y)dy$ its $in-$part and  $\int_{|y|<1}f(x,y)dy$ its $out-$part.
Several terms $k^{ji}(x)$, with $i$ and $j$ integers, will arise in the computations above. In these terms there always will be an integration of the form $$k^i_j(x)=\int_{\R}...dy.$$
We will call $k^{ji}_{\,in}(x)$ and $k^{ji}{j\, out}(x)$ to its $in-$part and to its $out-$part respectively.
\item For any $f$, we will write $\Delta_{t|\lambda-\lambda'|} f \equiv \Delta f (x, x-|t(\lambda-\lambda')|y).$
\item We always assume that $\ep<1$.
\item In every integral we take a principal value.
\end{enumerate}

\subsubsection{ Preliminary lemmas}
The proof is rather long and will be armed by  the lemmas below.  They could be ordered as follows.
\begin{itemize}
\item [i)]Estimates of pointwise of the kernel. Lemmas~\ref{Ca}, \ref{ca} and \ref{yelambda} estimate operators with non singular terms.
\item[ii)]Comparison between the kernels depending $\theta$ and those depending on  $\theta_{lin}.$ Lemmas \ref{guapis1} and \ref{121in}.
\item[iii)]Lemmas on kernels depending on $\theta,\theta_{lin}$.
\end{itemize}
The proof of iii) either use i) to show that the kernels are not singular or rely on properties of the Hilbert transform.

 The first two lemmas are pointwise properties of the functions involved. The proofs follows from the mean value theorem.
\begin{lemma}\label{Ca}
There exists a constant $1<C_A<\infty$ depending only on the $L^\infty$-norm of the $\pa_xh'(x)$ such that
\begin{align*}
\frac{1}{(y\pm\sigma_h(x) A_h(x)c(x))^2+c(x)^2\sigma(x)^2}\leq \mathcal{C}_A(y)\equiv\left\{\begin{array}{cc}C_A \quad |y|\leq C_{A} \\ \frac{C_A}{y^2} \quad |y|>C_A\end{array}\qquad \forall x\in \R\right.
\end{align*}
Here $A_h(x)=\pa_xh'(x)$ and $\sigma_h(x)=\frac{1}{1+A_h(x)^2}$.
\end{lemma}
\begin{proof}
Along this proof $C_A$ denotes a constant bigger that 1 and depending only in $||A_h||_{L^\infty}$. Firstly we notice that
$$\frac{1}{(y\pm\sigma_h(x) A_h(x)c(x))^2+c^2\sigma^2}\leq \frac{1}{(y\pm\sigma_h(x) A_h(x)c(x))^2+\left(\inf_{x\in\R}c(x)\sigma_h(x)\right)^2},$$
where $$\inf_{x\in\R} \sigma_h(x)=\frac{1}{1+||A_h||_{L^\infty}^2}\equiv \sigma_{inf}.$$  Fixed $x$, the function $\frac{1}{(y\pm\sigma(x)A(x)c(x))^2+(\sigma_{inf})^2}$ is a translation of the function $\frac{1}{y^2+(\sigma_{\inf})^2}$.  This is bounded by $\frac{1}{(\sigma_{inf})^2}$ and decay like $\frac{C_A}{y^2}$.  But $-C_{A}\leq \sigma_h(x)A_h(x)c(x)\leq C_{A}$. Then the conclusion of the lemma follows easily.
\end{proof}
The following lemma will allow us to show that numerators of various kernels are in fact bounded for $|y|$ sufficiently small.
\begin{lemma}\label{ca} There exists a constant $c_A$ which depends on $||f||_{C^1} +||\ep||_{C^1}$ and on $\overline{c}$  such that for $|y|<c_A$ the following inequality holds:
\begin{align*}
|c(x,t)|-\left|\frac{\Delta_{t|\lambda-\lambda'|} h  }{t|\lambda-\lambda'|y}y\right| & \geq \frac{\overline{c}}{2}\\
|c(x,t)|-|\pa_x h'(x) y|  & \geq \frac{\overline{c}}{2}.
\end{align*}
\end{lemma}
\begin{proof} By the mean value theorem,
\begin{align*}
\left|\frac{\Delta f(x,x-t|\lambda-\lambda'| y)+\Delta \ep(x,x-t|\lambda| y)}{t|\lambda-\lambda'|y}\right|\leq ||f||_{C^1} +||\ep||_{C^1}
\end{align*}
thus the claim follows.
\end{proof}
For the reiterative use we state that $L^1$ kernels gives us good bounds.

\begin{lemma}\label{yelambda} Consider a kernel $J_\lambda(x,y)$ satisfying $|J_\lambda(x,y)|\leq j(y)\in L^1(\R)$, for all $x\in \R$ and $\lambda \in [-1,1]$. Then, the  the integral
\begin{align*}
I_\lambda(x)=\int_{\R}J_\lambda(x,y)f(x-\lambda y)dy
\end{align*}
is bounded in $L^2$ as  $||I_\lambda||_{L^2}\leq C(||j||_{L^1})||f||_{L^2}$.
\end{lemma}

\begin{proof}
Again the proof is straightforward by using Minkowski inequality.
\end{proof}

The next  two lemmas allow us to compare $\theta$ with $\theta_{lin}=\pa_x h' y +\gamma $ in various expressions. We start with a pointwise bound.

\begin{lemma}\label{guapis1}The following bound holds for every $y \in \mathbb{R}$.
\begin{align*}
&\frac{1}{\left(y^2+\theta^2\right)^a}-\frac{1}{\left(y^2+\theta_{lin}^2\right)^a}\\
&\precsim  C_2\sum_{l=-1}^{2(a-1)}\frac{|y|^{2a-l}|\gamma|^{l+1}}{\left(y^2+\theta)^2\right)^a\left(y^2+\theta_{lin}^2\right)^a}
\end{align*}
for $a\geq 2$.
\end{lemma}
\begin{proof}
We just write that
\begin{align*}
&\frac{1}{\left(y^2+\left(\theta\right)^2\right)^a}-\frac{1}{\left(y^2+\theta_{lin}^2\right)^a}\\&=
\frac{\left(y^2+\theta_{lin}^2\right)^a-\left(y^2+\left(\theta\right)^2\right)^a}{\left(y^2+\left(\theta\right)^2\right)^a\left(y^2+\theta_{lin}^2\right)^a}
\end{align*}
and, since, $c^a-b^a= (c-b)\sum_{l=1}^{a} c^{a-l}b^{l-1}$ for $c,b\in \R$, we have that
\begin{align*}
&\left(y^2+\theta_{lin}^2\right)^a-\left(y^2+\left(\theta\right)^2\right)^a\\
&=\left(\theta_{lin}^2-\left(\theta\right)^2\right)\sum_{l=1}^a \left(y^2+\theta_{lin}^2\right)^{a-l}\left(y^2+\theta^2\right)^{l-1}.
\end{align*}
Next we introduce the expansions
\begin{align*}
&\left(\theta_{lin}^2-\left(\theta\right)^2\right)=\left(\pa_x h' y-\Delta h'\right)\left(\pa_x h' y +\Delta h' +2\gamma\right)\\
&\precsim C_2 |y|^2\left(|y|+|\gamma|\right),
\end{align*}

Here, we have used that since $h \in H^4$ , we have uniform $L^\infty$ bound of $\pa_x h',\pa^2_x h'$
and thus its uniform Lipschitz continuity. Next,  since

\begin{align*}
&\left(y^2+\theta_{lin}^2\right)^{a-l}=\sum_{i=0}^{a-l}c(i,a-l)y^{2(a-l-i)}(\theta_{lin})^{2i}\\
&=\sum_{i=0,n=0}^{a-l,2i}c(i,a-l)c(n,2i)y^{2(a-l)-n}(\pa_x h')^{2i-n}\gamma^n\\
&\precsim C_1 \sum_{i=0,n=0}^{a-l,2i}|y|^{2(a-l)-n}|\gamma|^n
\end{align*}
and
\begin{align*}
&\left(y^2+\left(\theta\right)^2\right)^{l-1}\\
&\precsim C_1 \sum_{j=0,m=0}^{l-1,2i}|y|^{2(l-1)-m}|\gamma|^m,
\end{align*}
it follows that,
\begin{align*}
&\left(y^2+\theta_{lin}^2\right)^a-\left(y^2+\left(\theta\right)^2\right)^a\\
&\precsim C_2 |y|^2(|y|+|\gamma|)\sum_{i=0,n=0}^{a-l,2i}\sum_{j=0,m=0}^{l-1,2i}|y|^{2(a-1)-(n+m)}|\gamma|^{n+m}\\
&\precsim C_2 |y|^2(|y|+|\gamma|)|y|^{2(a-1)}\sum_{l=0}^{2(a-1)}|y|^{-l}|\gamma|^l=C_2 |y|^{2a}(|y|+|\gamma|)\sum_{l=0}^{2(a-1)}|y|^{-l}|\gamma|^l\\
& =C_2 y^{2a}\left(\sum_{l=0}^{2(a-1)}|y|^{-l+1}|\gamma|^l +\sum_{l=0}^{2(a-1)}|y|^{-l}|\gamma|^{l+l}\right)
\\
&\precsim C_2 |y|^{2a}\sum_{l=-1}^{2(a-1)}|y|^{-l}|\gamma|^{l+1}.
\end{align*}
From this last inequality is easy to achieve the conclusion of the lemma.
\end{proof}
Next we show that we can also compare operators depending on $\theta$ by those depending on its linearization $\theta_{lin}$.

\begin{lemma}\label{121in}  Let  $a=2,3,4$ or $5$ and define
\begin{align*}
&k[g](x)=\int_{|y|<1}\left(\sum_{i=0}^{2a-1}|\gamma|^i|y|^{2a-1-i}\right)\\ &\quad\times \left(\frac{1}{\left(y^2+\theta^2\right)^a}-\frac{1}{\left(y^2+\theta_{lin}^2\right)^a}\right)g(x-y)dy.
\end{align*}
Then,
\[ \|k[g]\|_{L^2} \le  \lar ||g||_{L^2}, \|k[g]\|_{L^\infty} \le   \lar ||g||_{L^\infty}, \|k[1]\|_{L^\infty}  \le \lar, \]

where $g\in L^2$ in the first estimate and $g\in L^\infty$ in the second one.

\end{lemma}

\begin{proof}
By lemma \ref{guapis1} we have that
\begin{align*}
k(x)\precsim &C_2\int_{|y|<1}\sum_{i=0}^{2a-1}|\gamma|^i|y|^{2a-1-i}\\
&\quad \times \sum_{l=-1}^{2(a-1)}\frac{|y|^{2a-l}|\gamma|^{l+1}}{\left(y^2+\left(\theta\right)^2\right)^a\left(y^2+\theta_{lin}^2\right)^a}|g(x-y)|dy\\
&=C_2\sum_{l=-1,\, i=0}^{2(a-1),\, 2a-1}\int_{|y|<1}\frac{|y|^{4a-1-(i+l)}|\gamma|^{i+l+1}}{\left(y^2+\left(\theta\right)^2\right)^a\left(y^2+\theta_{lin}^2\right)^a}|g(x-y)|dy.
\end{align*}
By the upper bound on $\gamma$ we need to estimate $$k_{l,\,i}(x)=\int_{|y|<1}\frac{|y|^{4a-1-(i+l)}|t|\lambda-\lambda'||^{i+l+1}}{\left(y^2+\left(\theta\right)^2\right)^a\left(y^2+\theta_{lin}^2\right)^a}|g(x-y)|dy.$$

After the change of variable $y'=\frac{y}{t|\lambda-\lambda'|}$ we have that
\begin{align*}
&k_{l,\,i}(x)=t|\lambda-\lambda'|\int_{|y|<\frac{1}{t|\lambda-\lambda'|}}\frac{|y|^{4a-1-l-i} |g(x-t|\lambda-\lambda'| y)|}{\left(y^2+\left(\frac{\Delta_{t|\lambda-\lambda'|} h'}{t|\lambda-\lambda'|y}y+c(x)\sign(t(\lambda-\lambda'))\right)^2\right)^a\left(y^2+\left(\pa_xf y+c(x)\sign(t(\lambda-\lambda'))\right)^2\right)^a} dy\\
&=\int_{|y|<c_A}... \,dy + \int_{c_A <|y|<\frac{1}{t|\lambda-\lambda'|}}...\, dy.
\end{align*}
The integrand in $k_{l,\,i}(x)$ is bounded in $|y|<c_A$ by lemma \ref{ca} for every $-1\leq l\leq 2(a-1)$ and $0\leq i \leq 2a-1$. In $|y|>c_A$, the integrand is bounded by $C |y|^{-1-i-l}$ for every $-1\leq l\leq 2(a-1)$ and $0\leq i \leq 2a-1$. Then  Minkowski inequality yields
\begin{align*}
||k_{l,\,i}||_{L^2}\leq C_1 ||g||_{L^2}\left(1+|\gamma|\int_{c_A\leq |y|\leq \frac{1}{t|\lambda-\lambda'|}}|y|^{-1-i-l}dy\right)\leq C_1||g||_{L^2}
\end{align*}
for every $-1\leq l\leq 2(a-1)$ and $0\leq i \leq 2a-1$.
The $L^\infty$ bound simply follows by extracting $\|g\|_\infty$ as a constant.
\end{proof}

The next lemma  gives $L^2$ and $L^\infty$ bounds for the various operators. It turns out that after a change
of variable, lemma~\ref{ca} shows that in fact the kernels are not singular near cero, whereas away from cero direct $L^\infty$ bounds
are available for the kernel. This yields direct proofs for $L^\infty$ bounds and a further use of Minkowski inequality yields $L^2$ bound.
We state separately the action of the operator on 1 for later use.

\begin{lemma}\label{In}
Let $g\in L^2$.

\begin{itemize}
\item [a)] Let  $a\ge 2 $ and  $k[g](x)$ be given by
\begin{align*}
&k[g](x)=\sum_{i=0}^{2a}\int_{|y|<1}\frac{|y|^{2a-i}|\gamma|^i}{\left(y^2+\theta^2\right)^a}|g(x-y)|dy.
\end{align*}
Then
\[ \|k[g]\|_{L^2} \le  \lar ||g||_{L^2}, \quad \|k[1]\|_{L^\infty}  \le \lar. \]

\item [b)] Let $a=2,...,10$ and $k[g]$ be given by
\begin{align*}
&k[g](x)=\sum_{i=1}^{2a-1}\int_{|y|<1}\frac{|y|^{2a-1-i} |\gamma |^i}{\left(y^2+\theta^2\right)^a}g(x-y)dy.
\end{align*}
Then
\[ \|k[g]\|_{L^2} \le  \lar ||g||_{L^2}, \quad  \|k[1]\|_{L^\infty}  \le \lar. \]

\end{itemize}
\end{lemma}

\begin{proof}
We prove first the $L^2$ bound for $a)$.
 We notice that, by the upper bound on $\gamma$,it  is enough to estimate
\begin{align*}
k_i(x)=\int_{|y|<1}\frac{|y|^{2a-i}|t|\lambda-\lambda'||^i}{\left(y^2+\theta\right)^a}|g(x-y)|dy.
\end{align*}
After the change of variables $y'=\frac{y}{t|\lambda-\lambda'|}$ we have that
\begin{align*}
&k_i(x)=\int_{|y|<\frac{1}{t|\lambda-\lambda'|}}t|\lambda-\lambda'|\frac{|y|^{2a-i}}{\left(y^2+\left(\frac{\Delta_{t|\lambda-\lambda'} h' }{t|\lambda-\lambda'| y}y+c(x,t)\sign(\lambda-\lambda')\right)^2\right)^a}|g(x-t|\lambda-\lambda'|y)|dy\\
&=\int_{|y|<c_A}...\,dy +\int_{c_A<|y|<\frac{1}{t|\lambda-\lambda'|}}...\,dy.
\end{align*}
For every $i=0,...,2a$, in the region $|y|<c_A$ we can apply lemma \ref{ca}, obtaining that  the kernel is uniformly bounded. In the region $c_A<|y|<\frac{1}{t|\lambda-\lambda'|}$ we can estimate
\begin{align*}
t|\lambda-\lambda'|\frac{|y|^{2a-i}}{\left(y^2+\left(\frac{\Delta_{t|\lambda-\lambda'|} h}{t|\lambda-\lambda'| y}y+c(x,t)\sign(\lambda-\lambda')\right)^2\right)^a}\leq C_A t|\lambda-\lambda'| |y|^{-i}\leq C_A t|\lambda-\lambda'|,
\end{align*}
for every $i=0,...,2a$. Then we can apply Minkowski inequality to prove the lemma. Indeed,
\begin{align*}
&C_At|\lambda-\lambda'|\left|\left|\int_{c_A<|y|<\frac{1}{t|\lambda-\lambda'|}}|g(x-t|\lambda-\lambda'|y)|dy\right|\right|_{L^2}\leq
C_At|\lambda-\lambda'|\int_{c_A<|y|<\frac{1}{t|\lambda-\lambda'|}}|\left|\left|g(\cdot-t|\lambda-\lambda'|y)\right|\right|_{L^2}dy\\
&\leq C_At|\lambda-\lambda'|||g||_{L^2}\int_{c_A<|y|<\frac{1}{t|\lambda-\lambda'|}}dy \leq C_A ||g||_{L^2}.
\end{align*}

The $L^\infty$ bound follows in the same way. The case $b)$ is dealt with by the same change of variables.

\end{proof}

In the next lemmas the kernel scales as $y^{-1}$  and thus the estimates are more delicate. For the outer
integral we show that  the kernel can be decomposed as the sum of  $c(x)\frac1y$ and a function which decays as $|y|^{-2}$.  Thus, the Hilbert transform controls the first
and the second is not singular.

\begin{lemma}\label{121out}
Let $g\in L^2$,   $a\geq 2$ and
\begin{align*}
&k(x)=\int_{|y|>1}\frac{y^{2a-1}}{\left(y^2+\theta_{lin}^2\right)^a}g(x-y)dy.
\end{align*}
Then,
\[ \|k[g]\|_{L^2} \le  \lar ||g||_{L^2}, \quad  \|k[1]\|_{L^\infty}  \le \lar.  \]
\end{lemma}

\begin{proof}
Direct computation shows hat
\[\frac{y^{2a-1}}{\left(y^2+\theta_{lin}^2\right)^a}-\frac{1}{(1+(\pa_x h')^2)^a}\frac{1}{y}=j(x,y), \]
where
$j(x,y)=\frac{(1+(\pa_x h')^2)^a y^{2a}-(y^2+(\pa_x h' y+\gamma)^2)^a}{\left(y^2+\theta_{lin}^2\right)^a(1+(\pa_x h')^2)^2y}.$
Since  $|j(x,y)| \le C|y|^{-2}$ the claim for $L^2$ follows from the $L^2$ boundedness of the truncated Hilbert transform and lemma~{\ref{yelambda}}. In
order to estimate $k[1]$ notice that $\int_{|y|>1} \frac{1}{y} dy =0$.

\end{proof}

The proof of the next lemma is more subtle as it uses an  explicit computation of the  Hilbert transform of our kernel.
\begin{lemma} \label{hilbert} Let $g\in L^2$ and $a=2,3,4$ or $5$. Then, the integral
\begin{align*}
I[g](x)=\int_{\R}\frac{y^{2a-1}}{\left(y^2+\theta_{lin}^2\right)^a}g(x-y)dy,
\end{align*}
satisfies

\[ \|I[g]\|_{L^2} \le  \lar ||g||_{L^2},\quad \|I[1]\|_{L^\infty}  \le \lar .\]

\end{lemma}
\begin{proof}
In analogy with other estimates,  we denote $\sigma_h=(1+\pa_xh'(x)^2)^{-1}$ and $A_h=\pa_x h'(x)$. Therefore
\begin{align}\label{sigmaidentity}
&y^2+(\pa_x h'y+\gamma)^2=\sigma_h^{-1}((y+A_h\sigma_h\gamma)^2+\sigma_h^2\gamma^2).
\end{align}
and
\begin{align}\label{I}
I(x)= & C_1(x)\int_{\R}\frac{y^{2a-1}}{((y+A_h\sigma_h\gamma)^2+\sigma_h^2\gamma^2)^a}g(x-y)dy\nonumber\\
&=C_1(x)\int_{\R}\frac{(y+A_h\sigma_h\gamma)^{2a-1}}{((y+A_h\sigma_h \gamma)^2+\sigma_h^2\gamma^2)^a}g(x-y)dy\nonumber\\
&+C_1(x)\int_{\R}\frac{y^{2a-1}-(y+A_h\sigma_h\gamma)^{2a-1}}{((y+A_h\sigma_h \gamma)^2+\sigma_h^2\gamma^2)^a}g(x-y)dy\nonumber\\
&\equiv I_1(x)+I_2(x).
\end{align}
Now we use the identity (for fixed $\sigma_h$ and $\gamma$)
\begin{align*}
H\left[\frac{x^{2a-1}}{(x^2+\sigma_h^2\gamma^2)^{2a-1}}\right](x)=-\frac{\sigma_h |\gamma|}{(x^2+\sigma_h^2\gamma^2)^a}\sum_{l=0}^{a-1}\alpha_{al}\left(\sigma_h|\gamma|\right)^{2(a-1-l)}x^{2l}
\end{align*}
where the $\alpha_{al}$'s are harmless coefficients. Then
\begin{align*}
I_1(x)=\int_{\R}\frac{\sigma_h |\gamma|}{((y+A_h \sigma_h\gamma)^2+\sigma_h^2\gamma^2)^a}\sum_{l=0}^{a-1}\alpha_{al}\left(\sigma_h|\gamma|\right)^{2(a-1-l)}(y+A_h \sigma_h \gamma)^{2l}Hg(x-y)dy,\\
\end{align*}
 and after the usual change of variables
 \begin{align*}
& I_1(x)\precsim C_1\sum_{i=0}^{a-1}\int_{\R}\frac{(y+c A_h\sigma_h t\sign((\lambda-\lambda'))^{2i}}{((y+c\sigma_h A_h\sign(\lambda-\lambda'))^2+c^2\sigma_h^2)^a}H\ g(x-t|\lambda-\lambda'| y)dy\\
 &\precsim C_1\sum_{i=0}^{a-1}\int_{\R}(1+|y|^{2i})\mathcal{C}_A(y)^a Hg(x-t|\lambda-\lambda'|y)dy,
 \end{align*}
 where we have applied lemma \ref{Ca}. Then $I_1(x)$ is bounded in $L^2$ thanks to lemma \ref{yelambda}. To bound $I_2(x)$ we notice that we can write this term in the following way:
\begin{align*}
I_2(x)=\int_{\R}\frac{y^{2a-1}-(y+c\sigma_h A_h \sign(\lambda-\lambda'))^{2a-1}}{((y-c\sigma_h A_h\sign(\lambda-\lambda'))^2+c^2\sigma_h^2)^a} g(x-t|\lambda-\lambda'| y)dy.
\end{align*}
Since the numerator is a polynomial in $y$ of order $2(a-1)$ we can apply again lemmas \ref{Ca} and \ref{yelambda} to obtain a suitable estimate in $L^2$.

In order to estimate $I[1]$ simply replace $g(x-y)$ by $1$ and notice that the analogous term to $I_1(x)$ in \eqref{I} is equal to zero in this case.

\end{proof}

\subsubsection{ Estimation of the terms in the sum \eqref{easyb}.}\label{Afirstestimates}

We need to estimate the terms $k_\theta^{ji}\partial_x^{5-j}\theta$ which can be further decomposed into
a sum of products of various derivatives of $\theta$ divided by $(y^2+\theta^2)^a$ for $a=2,3,4$. It will be important
that since we assume that $h$ and hence $\theta$ belongs to $H^4$ we can assume that $\partial_x^2\theta$ is uniformly
Lipschitz. Thus if in the product of derivatives $\partial_x^{j_1} \theta \partial_x^{j_2} \theta \cdots \partial_x^{j_n}
 \partial_x \theta$ there is only one $j_i \ge 3$ we can interpreted as an operator acting on $\partial^{j_i}_x \theta$ and we could
 put our hands on the kernel, linearizing it  using the Lipschitz continuity.  We illustrate this in the first section with $j=2$ and leave
 the rest to the reader. Unfortunately, there are some terms with $j=3$ with
 e.g $(\partial^3_x \theta)^2$ appears. Those have to be dealt with independently. We do it by means of another pseudodifferential operator.

\textbf{1.1. Terms in \eqref{easyb} with $j=2$, $i=1,2$.}

\textbf{1.1.1.} Let us estimate $k^{21}(x)$. We split it into two terms $k^{21}(x)=k^{21}_{1}(x)+k^{21}_2(x)$, with
\begin{align*}
k^{21}_1(x)=-\int_{\R}k^{21}_\theta (x,y)\pa^4_x h'(x-y) dy \quad \text{and} && k^{21}_2(x)=\pa^4_x h\int_{\R} k^{21}_\theta(x,y)dy.
\end{align*}

To bound $k^{21}_1$ we proceed as follows,
\begin{align*}
k^{21}_1(x)=-\int_{\R}y\frac{(\pa_x\theta)^2+\theta\pa^2_x\theta}{(y^2+\theta^2)^2}\pa^4_x h'(x-y) dy.
\end{align*}

Since $y\frac{(\pa_x\theta)^2+\theta\pa^2_x\theta}{(y^2+\theta^2)^2}\precsim C_2|y|^{-3}$, $k^{21}_{1\,out}$ is bounded by lemma \ref{yelambda}. In order to bound the inner part we further split it  into two terms, one with only terms depending on $h'$ and the other where $\gamma$ and its derivatives appear. Namely,  $k^{21}_{1\,in}=k^{21}_{11\,in}+k^{21}_{11\,in}$, with
\begin{align}\label{k2111in}
k^{21}_{11\,in}&=\int_{|y|<1} y \frac{(\pa_x\gamma)^2+2\Delta \pa_x h'\pa_x\gamma}{(y^2+\theta^2)^2}\pa^4_x h'(x-y)dy\nonumber\\
&+\int_{|y|<1} y \frac{\gamma (\Delta\pa^2_x h'+ \pa^2_x\gamma)+\Delta h'  \pa_x^2\gamma}{(y^2+\theta^2)^2}\pa^4_x h'(x-y)dy
\end{align}
and
\begin{align}\label{k2112in}
k^{21}_{12\,in}=\int_{|y|<1} y \frac{(\pa_x\Delta h')^2+\Delta h'\pa^2_x\Delta h'}{(y^2+\theta^2)^2}\pa^4_x h'(x-y)dy.
\end{align}

To estimate $k^{21}_{11\, in}$ we use the regularity of $h'$ and mean value theorem to bound  its integrand by
\begin{align*}
C_3\frac{t^2|\lambda-\lambda'|^2|y|+|y|^2t|\lambda-\lambda'|}{(y^2+\theta^2)^2}|\pa^4h(x-y)|
\end{align*}
and after that we apply lemma  \ref{In}.

To bound $k^{21}_{12}(x)$ we write
\begin{align}\label{k2112split}
k^{21}_{12}(x)=&-\int_{\R}\left(\frac{y\left((\pa_x\Delta h')^2+\Delta \pa^2_x\Delta h'\right)}{\dl^2}-\frac{y^3\left((\pa^2_xh')^2+\pa_x h'\pa^3_x h\right)}{\dlx^2}\right)\pa_x^4h'(x-y)dy\nonumber\\
&-\left((\pa^2_xh')^2+\pa_xh'\pa^3_xh'\right)\int_{\R}\frac{y^3}{\dlx^2}\pa^4_xh'(x-y)dy\nonumber\\
&\equiv k^{21}_{121}(x)+k^{21}_{122}(x).
\end{align}

To bound $k^{21}_{121\, out}(x)$ we notice that $\frac{(\pa_x\Delta h')^2+\Delta \pa^2_x\Delta h'}{\dl^2}\precsim C_2 |y|^{-3}$ and  we can directly apply \ref{121out} with $a=2$ to
the second term other part. To bound $k^{21}_{121\,in}(x)$ we split it  into  two new terms,
\begin{align*}
&k^{21}_{121\, in}(x)=-\int_{|y|<1}\left(\frac{y((\pa_x\Delta h') ^2+\Delta h' \pa_x^2\Delta h')-y^3((\pa^2_xh')^2+\pa_xh'\pa^3_xh')}{\dl^2}\right)\pa^4_x h'(x-y)dy\\
&-((\pa^2_xh')^2+\pa_xh'\pa^3_xh')\int_{\R}y^3\left(\frac{1}{\dl^2}-\frac{1}{\dlx^2}\right)\pa^4_x h'(x-y)dy\\
&\sim -\int_{|y|<1}\left(\frac{y((\pa_x\Delta h')^2+\Delta h' \pa_x^2\Delta h')-y^3((\pa^2_xh')^2+\pa_xh'\pa^3_xh')}{\dl^2}\right)\pa^4_x h'(x-y)dy,
\end{align*}
where we have applied lemma \ref{121in}, with $a=2$. Therefore, applying lemma \ref{In} , with $a=2$ we check that
\begin{align*}
&k^{21}_{121\,in}(x)\precsim \int_{|y|<1}\frac{C_3 |y|^4+C_{3+\alpha}|y|^{3+\alpha}}{\dl^2}|\pa^4_x h'(x-y)|dy\\ & \sim \int_{|y|<1}\frac{C_{3+\alpha}|y|^{3+\alpha}}{\dl^2}|\pa^4_x h'(x-y)|dy\precsim \int_{|y|<1}|y|^{-1+\alpha}|\pa^4_x h'(x-y)|dy.
\end{align*}
Thus, we can apply lemma \ref{yelambda} to finish the estimate of $k^{21}_{121\, in}(x)$.

To bound $k^{21}_{122}(x)$ we apply lemma \ref{hilbert} with $a=2$. This finishes the bound for $k^{21}_1(x)$.

To bound $k^{21}_2(x)$ in $L^2$ is enough to bound $\ok^{21}_2(x)\equiv \int_{\R}k^{21}_\lambda(x,y)dy$ in $L^\infty$. The outer part is estimated as we did for $k^{21}_1$. For the inner part we split $\ok^{21}_{2\, in}(x)\equiv \ok^{21}_{21\, in}(x)+\ok^{21}_{22\,in}$, with $\ok^{21}_{21\,in}$ as $k^{22}_{21\,in}(x)$ in \eqref{k2111in} and $\ok^{21}_{22\, in}(x)$ as $k^{21}_{12\, in}(x)$ in \eqref{k2112in} but replacing $\pa^4_xh'(x-y)$ by 1. $\ok^{21}_{21\,in}(x)$ is bounded by applying lemma \ref{In}. To bound $\ok^{21}_{22}(x)$ we split this term into  $\ok^{21}_{221}(x)+\ok^{21}_{221}$ (analogous to $k^{21}_{121}(x)$ and $k^{21}_{122}(x)$ in \eqref{k2112split}). $\ok^{21}_{221\, out}(x)$ is bounded by using lemmas \ref{yelambda} and \ref{121out}. For $\ok^{21}_{221\,in}(x)$ we do the analogous splitting than for $k^{21}_{121\,in}$ and we apply the same argument together with lemma \ref{In}. To bound $\ok^{21}_{222}(x)$ we use lemma \ref{hilbert}. This finishes the estimate of $k^{21}(x)$ in $L^\infty$ and  completes the proof of the estimate of $k^{21}(x)$ in $L^2$.

To bound $k^{21}_{12}(x)$ we write
\begin{align}\label{k2112splitb}
k^{21}_{12}(x)=&-\int_{\R}\left(\frac{y\left((\pa_x\Delta h')^2+\Delta \pa^2_x\Delta h'\right)}{\dl^2}-\frac{y^3\left((\pa^2_xh')^2+\pa_x h'\pa^3_x h\right)}{\dlx^2}\right)\pa_x^4h'(x-y)dy\nonumber\\
&-\left((\pa^2_xh')^2+\pa_xh'\pa^3_xh'\right)\int_{\R}\frac{y^3}{\dlx^2}\pa^4_xh'(x-y)dy\nonumber\\
&\equiv k^{21}_{121}(x)+k^{21}_{122}(x).
\end{align}

To bound $k^{21}_{121\, out}(x)$ we notice that $\frac{(\pa_x\Delta h')^2+\Delta \pa^2_x\Delta h'}{\dl^2}\precsim C_2 |y|^{-3}$ and  we can directly apply \ref{121out} with $a=2$ to
the second term. To bound $k^{21}_{121\,in}(x)$ we split it  into  two new terms,
\begin{align*}
&k^{21}_{121\, in}(x)=-\int_{|y|<1}\left(\frac{y((\pa_x\Delta h') ^2+\Delta h' \pa_x^2\Delta h')-y^3((\pa^2_xh')^2+\pa_xh'\pa^3_xh')}{\dl^2}\right)\pa^4_x h'(x-y)dy\\
&-((\pa^2_xh')^2+\pa_xh'\pa^3_xh')\int_{\R}y^3\left(\frac{1}{\dl^2}-\frac{1}{\dlx^2}\right)\pa^4_x h'(x-y)dy\\
&\sim -\int_{|y|<1}\left(\frac{y((\pa_x\Delta h')^2+\Delta h' \pa_x^2\Delta h')-y^3((\pa^2_xh')^2+\pa_xh'\pa^3_xh')}{\dl^2}\right)\pa^4_x h'(x-y)dy,
\end{align*}
where we have applied lemma \ref{121in}, with $a=2$. Therefore, applying lemma \ref{In} , with $a=2$ we check that
\begin{align*}
&k^{21}_{121\,in}(x)\precsim \int_{|y|<1}\frac{C_3 |y|^4+C_{3+\alpha}|y|^{3+\alpha}}{\dl^2}|\pa^4_x h'(x-y)|dy\\ & \sim \int_{|y|<1}\frac{C_{3+\alpha}|y|^{3+\alpha}}{\dl^2}|\pa^4_x h'(x-y)|dy\precsim \int_{|y|<1}|y|^{-1+\alpha}|\pa^4_x h'(x-y)|dy.
\end{align*}
Thus, we can apply lemma \ref{yelambda} to finish the estimate of $k^{21}_{121\, in}(x)$.

To bound $k^{21}_{122}(x)$ we apply lemma \ref{hilbert} with $a=2$. This finishes the bound for $k^{21}_1(x)$.

To bound $k^{21}_2(x)$ in $L^2$, it  is enough to bound $\ok^{21}_2(x)\equiv \int_{\R}k^{21}_\lambda(x,y)dy$ in $L^\infty$. The outer part is estimated as we did for $k^{21}_1$. For the inner part we split $\ok^{21}_{2\, in}(x)\equiv \ok^{21}_{21\, in}(x)+\ok^{21}_{22\,in}$, with $\ok^{21}_{21\,in}$ as $k^{22}_{21\,in}(x)$ in \eqref{k2111in} and $\ok^{21}_{22\, in}(x)$ as $k^{21}_{12\, in}(x)$ in \eqref{k2112in} but replacing $\pa^4_xh'(x-y)$ by 1. $\ok^{21}_{21\,in}(x)$ is bounded by applying lemma \ref{In}. To bound $\ok^{21}_{22}(x)$ we split this term into  $\ok^{21}_{221}(x)+\ok^{21}_{221}$ (analogous to $k^{21}_{121}(x)$ and $k^{21}_{122}(x)$ in \eqref{k2112splitb}). $\ok^{21}_{221\, out}(x)$ is bounded by using lemmas \ref{yelambda} and \ref{121out}. For $\ok^{21}_{221\,in}(x)$ we do an analogous splitting to that  for $k^{21}_{121\,in}$ and we apply the same argument together with lemma \ref{In}. To bound $\ok^{21}_{222}(x)$ we use lemma \ref{hilbert}. This finishes the estimate of $k^{21}(x)$ in $L^\infty$ and  completes the proof of the estimate of $k^{21}(x)$ in $L^2$.

\textbf{1.1.2.} Let us estimate $k^{22}(x)$. We split into two terms \begin{align}\label{k22} k^{22}(x)=k^{22}_1(x)+k^{22}_2(x),\end{align}
with
\begin{align*}
k^{22}_1(x)=-\int_{\R}k^{22}_\theta (x,y)\pa^4_x h'(x-y)dy,\quad{\text{and}} && k^{22}_2=\pa^4_x h \int_{\R}k^{22}_\theta(x,y)dy.
\end{align*}
We split $k^{22}_1(x)$ in four terms, $k^{22}_1(x)=k^{22}_{11}(x)+k^{22}_{12}(x)+k^{22}_{13}+k^{22}_{14}$, with
\begin{align}
\label{k2211}&k^{22}_{11}(x)=-\int_{\R}y\frac{(\Delta \pa_x h')^2(\gamma^2+2\gamma \Delta h')
 }{(y^2+\theta^2)^3}\pa^4_x h'(x-y)dy\\
\label{k2212}&k^{22}_{12}(x)=-\int_{\R}y\frac{(\Delta \pa_x h')^2(\Delta h')^2}{(y^2+\theta^2)^3}\pa_x^4h'(x-y)dy.\\
\label{k2213}&k^{22}_{13}(x)=-\int_{\R}y\frac{(2\pa_x \gamma \Delta \pa_x h'+\pa_x \gamma^2)(\Delta h')^2 }{(y^2+\theta^2)^3}\pa^4_x h'(x-y)dy\\
\label{k2214}&k^{22}_{14}(x)=-\int_{\R}y\frac{((\pa_x \gamma)^2+2\pa_x \gamma \Delta \pa_x h')(\gamma^2+2\gamma \Delta h')
 }{(y^2+\theta^2)^3}\pa^4_x h'(x-y)dy.
\end{align}
The function $k^{22}_{11}(x)$ can be bounded in $L^2$ as follows. The integrand of $k^{22}_{11}(x) \precsim C_1|y|^{-5}|\pa^4_x h'(x-y)|$. Then $k^{22}_{11\,out}(x)$ is estimated by lemma \ref{yelambda}. Also the integrand of
 $k^{22}_{11}(x)\precsim C_2\frac{y^3(\gamma^2+2|\gamma|y)}{(y^3+(\theta)^2)^3}|\pa^4_x h'(x-y)|$ and we can apply lemma \ref{In} with $a=3$ to bound $k^{22}_{11\, in}(x)$. Similarly, we can deal with $k^{22}_{13}$ and $k^{22}_{14}$ as we can obtain the correct estimates in powers of $|y|$ and $|\gamma|$ to apply Lemma
 \ref{In}.

To bound $k^{22}_{12}(x)$ we write
\begin{align}\label{k2212split}
&k^{22}_{12}(x)=-\int_{\R}\left(\frac{y(\Delta \pa_x h')^ 2 (\Delta h')^2}{\left(y^2+\theta^2\right)^3}-\frac{(\pa^2_x h')^2(\pa_x h')^2 y^5}{\left(y^2+\theta_{lin}^2\right)^3}\right)\pa^4_x h'(x-y) dy\nonumber\\
&-(\pa_x h')^2(\pa^2_x h')^2 \int_{\R}\frac{y^5}{\left(y^2+\theta_{lin}^2\right)^3}\pa_x^4h'(x-y)\nonumber\\
&\equiv k^{22}_{121}(x)+k^{22}_{122}(x).
\end{align}

To bound $k^{22}_{121,\, out}(x)$ we notice that $\frac{y(\Delta \pa_x h')^2(\Delta h')^2}{\left(y^2+(\theta)^2\right)^3}\precsim C_1 |y|^{-5}$ and  we can apply \ref{121out} to the other part with $a=3$. To bound $k^{22}_{121,in}(x)$ we split in two terms,
\begin{align*}
k^{22}_{121,\,in}&=\int_{|y|<1}\left(\frac{y(\Delta h')^2 (\Delta \pa_x h')^2 -y^5(\pa_x h)^2(\pa^2_x h)^2}{\left(y^2+(\theta)^2\right)^3}\right)\pa^4_x h'(x-y)dy\\
&+(\pa_x h')^2(\pa^2_x h')^2\int_{|y|<1}y^5\left(\frac{1}{\left(y^2+\theta^2\right)^3}-\frac{1}{\left(y^2+\theta_{lin}^2\right)^3}\right)\pa^4_x h'(x-y)dy,
\end{align*}
in such a way that we can apply lemma \ref{In} with $a=3$, since $|y(\Delta h')^2 \Delta (\pa_x h')^2 -y^5(\pa_x h')^2(\pa^2_x h')^2|\precsim C_3 y^6$, and lemma \ref{121in} with $a=3$.

To bound $k^{22}_{122}(x)$ we  apply lemma \ref{hilbert} with $a=3$.
This finishes the bound for $k^{22}_1(x)$.The term $k^{22}_2(x)$is estimated in a similar way to $k^{21}_2$.
We have completed the proof of the estimate of $k^{22}(x)$ in $L^2$.

\textbf{1.2. Terms in \eqref{easyb} with $j=3$ and $i=1$.}

Let us bound $k^{31}(x)$. Unfortunately the proof of the estimation of $k^{31}(x)$ does not follow the same steps than the rest of the functions $k^{i}_{j}(x)$. Indeed we need to do something different and use the boundedness of pseudodifferential operators used in the body of the text.

We split into two terms $k^{31}(x)=k^{31u}(x)+k^{31d}_2(x)$
with
\begin{align*}
&k^{31u}(x)=\int_{\R}y\frac{\pa_x \theta \pa^2_x \theta}{\dl^2}\pa^3_x \Delta h' dy\\ & k^{31d}(x)=\int_{\R}y\frac{(\theta)\pa^3_x \theta}{\dl^2} \pa^3_x \Delta h' dy.
\end{align*}
The proof for $k^{31u}(x)$ will be similar to the above.

 For $k^{31d}(x)$ we need new estimates since as $h \in H^4$ we can not bound $\pa^3_x  \theta$ by $|y|$ uniformly.

  Let us bound $k^{31d}(x)$ first. We split this function in four  parts,
\begin{align}\label{k31d}
& k^{31d}_1(x)=  \int_{\R}y\frac{\gamma \pa^3_x  h'}{\dl^2} \pa^3_x \Delta h' dy,\quad \text{and} \\
& k^{31d}_{2}=\int_{\R}y\frac{ \Delta h'  \pa^3_x \Delta h'}{\dl^2}\pa^3_x \Delta h' dy \quad \text{and} \\
& k^{31d}_{3}=\int_{\R}y\frac{ \Delta h'  \pa^3_x \gamma}{\dl^2}\pa^3_x \Delta h' dy \\
&k^{31d}_4(x)=  \int_{\R}y\frac{\gamma \pa^3_x  \gamma }{\dl^2} \pa^3_x \Delta h' dy.\quad \text{and}
\end{align}

The last two terms are easily bounded by Lemma~\ref{In} for the in part and by Lemma~\ref{yelambda} for the outer part.

To bound $k^{31d}_{2}(x)$ we use that $\Delta (\pa_x^3h )^2=2\pa^3_x h \Delta \pa^3_x h - \Delta \left((\pa^3_x h)^2\right)$, and then, we need to show that the integral
\begin{align}\label{N}
N(x)=\int_{\R}\frac{y\Delta h'}{\dl^2}\Delta gdy
\end{align}
is in $L^2$ with either $g=\pa^3_x h'$ or $(\pa^3_x h')^2$. Notice that, in both cases, we can allow in our estimates that $||g||_{H^1}$ appears. We will split $N(x)$ in two terms $N_1(x)$ and $N_2(x)$ with
\begin{align}
&N_1(x)=\int_{\R}y\left(\frac{\Delta h'}{\dl^2}-\frac{\pa_x h' y}{\dlx^2}\right)\Delta g dy,\label{N1}\\
&N_2(x)=\pa_x h'\int_{\R}\frac{y^2}{\dlx^2}\Delta g dy \label{N2}.
\end{align}
That is $N_1$ compares with the linearized version and $N_2$ treats with the linearized kernel.
We can not deal directly with  $N_1$ with our previous lemmas by replacing $\Delta h'$ by $y$ as the denominator is too singular. However
 we can add an subtract a term

\[  \int \frac{\frac{1}{2}\pa^2_x h' y^2}{\dlx^2} \Delta g  dy. \]

Then several terms appear but since $h \in H^4$ we have that

\begin{equation}\Delta h'-\pa_x h'y-\partial_x^2 h'(x) y^2 \le |y|^3.
\end{equation}

After splitting in various terms, all of them can be deal with  using our lemmas ~\ref{121in}, \ref{In}, \ref{hilbert} and a small modification.

In order to deal with $N_2(x)$ we need to introduce another new idea.  We first use \eqref{sigmaidentity} and treat the two terms in
$\Delta g$ separately.
proceed as follows
\begin{align*}
N_2(x)=&\pa_x h'\gamma\sigma_h^2 \left(g(x)\int_{\R}\frac{y^2}{((y+\sigma_h A_h \gamma)^2+\sigma^2\gamma^2)^2}dy\right.\nonumber \\ &\left.\qquad\qquad-\int_{\R}\frac{y^2}{((y+\sigma_h A_h \gamma)^2+\sigma_h^2\gamma^2)^2}g(x-y)dy\right)
\end{align*}
Now we can compute that
\begin{align*}
\int_{\R}\frac{y^2}{((y+\sigma_h A_h \gamma)^2+\sigma_h^2\gamma^2)^2}dy=\frac{\pi}{2\sigma_h^2|\gamma|}
\end{align*}

For the convolution term,   the following Fourier transform computation

\begin{equation}\begin{aligned} &\mathcal{F}\left[\frac{y^2}{\left((y+\sigma A \lambda)^2+\sigma^2\lambda^2\right)^2}\right](\xi)\\
&=\frac{\pi}{2\sigma^2|\lambda|}e^{2\pi i A\sigma \xi \lambda}e^{-2\pi\sigma |\lambda||\xi|}\left(1+2\pi\sigma^2(A^2-1)+4\pi iA\sigma^2\xi\lambda\right)\end{aligned}\end{equation}
yields that
\begin{align*}
&\int_{\R}\frac{y^2}{((y+\sigma_h A_h \gamma)^2+\sigma_h^2\gamma^2)^2}g(x-y)dy\\
&= \frac{\pi}{2\sigma_h^2|\gamma|}\int_{\R}e^{2\pi i\xi x}\hat{g}(\xi)e^{2\pi i \xi A_h\sigma_h \gamma}e^{-2\pi \sigma_h |\gamma||\xi|}\\& \qquad\qquad \qquad \times(1+2\pi\sigma_h^2|\gamma||\xi|(A^2-1)+4\pi i A\sigma_h^2 \xi \gamma)d\xi,
\end{align*}
so that

\[ N_2(x)=\pa_x h'\frac{\pi}{2}\Op(p)(\Lambda g) \]

where the symbol is given by,
\[\begin{aligned}  p(x,\xi)= &\frac{1}{|\xi||\gamma|}\left(1-e^{2\pi i A_h\sigma_h \xi\gamma}e^{-2\pi \sigma_h |\xi||\gamma|}\right)\\
&-\sigma_h^2\pi^2 e^{2\pi i A_h\sigma_h \xi\gamma}e^{-2\pi \sigma_h |\xi||\gamma|}((A_h^2-1)+2i \sign(\xi)\sign(\gamma))\end{aligned} \]

Therefore by applying lemma \ref{Hwan} we obtain that the $L^2-$norm of $N_2(x)$ is bounded by $\lar (||\Lambda g||_{L^2}+||\pa_xg||_{L^2})$.

This concludes the proof of the estimate of the $L^2-$norm of $k^{31d}_2(x)$.

To deal with $k_1^{31d}$, we use again that $(\Delta \pa_x^3 h')^2=2\pa^3_x h' \Delta \pa^3_x h' - \Delta \left((\pa^3_x h')^2\right)$. Thus,  it suffices  to bound the integral
\begin{align}\label{M}
M(x)=\int_{\R}\frac{y\gamma}{(y^2+(\theta)^2)^2} \Delta g(x-y) dy
\end{align}
in terms of the $H^1$-norm of $g$.  The proof is analogous since at the only delicate point it holds that

$$\int_\R\frac{y\gamma}{\dlx^2}dy = -\frac{A\pi}{2\sigma^2 |\gamma|},$$

\textbf{1.3.The rest of the terms in \eqref{easyb}.}

The estimation of the rest of the terms in \eqref{easyb} follow the same steps than the estimation for either $k^{21}(x)$ or $k^{22}(x)$.

\subsubsection{Estimation of the terms in \eqref{safe}.}
We will  show how to estimate is $k^{11}(x)$ in  \eqref{safe}, since $k^{55}(x)$ is analogous. Here we recall that we are concerned with $||D^{-1} k^{11}||_{L^2}$. In order to bound this norm we will proceed as follows
\begin{align*}
k^{11}(x)=\int_{\R}k^{11}_\theta(x,y)\pa^5_x \theta dy=\int_{\R}k^{11}_\theta (x,y)D\Theta dy,
\end{align*}
where $\Theta \equiv =D^{-1}\pa^5_x\theta$ (we clarify that the operator $D=(1+t\pa_x)$ acts on $x$ rather than $y$). Then we would like to estimate $||D^{-1} k^{11}||_{L^2}\leq \lar ||\Theta||_{L^2}.$
In order to do it we notice that
\begin{align*}
\pa_x \int_{\R} k^{11}_\theta(x,y) \Theta dy =\int_{\R}\pa_x k^{11}_\theta(x,y) \Theta dy +\int_{\R}k^{11}_\theta(x,y)\pa_x  \Theta dy,
\end{align*}
so that
\begin{align*}
&D^{-1}k^{11}(x)= \int_{\R} k^{11}_\theta(x,y) \Theta dy-tD^{-1}\int_{\R}\pa_x k^{11}_\theta(x,y) \Theta dy\\ &\equiv S_1(x)+tD^{-1}S_2(x).
\end{align*}
Happily,  the proof of the estimation for $S_1$ and $S_2$ follow the same steps that the estimation of $k^{22}(x)$ in $\eqref{k22}$. Thus we have proven \eqref{loquesea}. That is

$$\pa^5_x\mathcal{M} u =-\frac{1}{4\pi}\int_{-1}^1\int_{-1}^1 \int_{\R} k_{\theta}(\cdot,y)\pa^6_x \theta dy d\lambda'd\lambda+ l.o.t.$$

In order to finish the proof of lemma \ref{ordenmasalto} notice that \[ k_\theta \partial_x^6\theta= k_\theta \pa^6_x \Delta f(x,x-y)+ k_\theta  \left(\pa_x^6\ep(x)\lambda-\pa_x^6\ep(x-y)\lambda'\right) \]
and, by assumptions on $c(x,t)$, $\pa_x^6c(\cdot,t)\in L^2$ uniformly in $t$. We then have that

$$\pa^5_x\mathcal{M} u =-\frac{1}{4\pi}\int_{-1}^1\int_{-1}^1 \int_{\R} k_{\theta}(\cdot,y)\pa^6_x \Delta f(x,x-y) dy d\lambda'd\lambda+ l.o.t.$$

Lemma  \ref{ordenmasalto}  is proved.

\subsection{Proof of lemma \ref{descomposicion}}\label{apendice4}
In this section we will prove lemma \ref{descomposicion}. We will use the same convection as in the previous section.

Since the transport term in lemma \ref{descomposicion} arises in an obvious way,
t he main issue is  to linearize $\theta$ to $\theta_{lin}$ in the integral
$$\int_{-1}^1\int_{-1}^1\int_{\R}\frac{y}{y^2+\theta^2}\pa^6_x f(x-y)dyd\lambda \lambda'.$$

This is the content of the following $L^2$ estimate.

\begin{lemma}Let $f\in H^6$, $c$ as in theorem \ref{existencialocal} and
\begin{align*}
& F(x)=D^{-1}\frac{1}{4}\int_{-1}^1\int_{-1}^1\int_{\R} \left(\frac{y}{y^2+\theta^2}-\frac{y}{y^2+\theta_{lin}^2}\right)\pa^6_x f(x-y)dyd\lambda \lambda'.
\end{align*}
Then,
\[ \|F\|_{L^2}\leq  \lar ||D^{-1}\pa^5_x f||_{L^2}. \]
\end{lemma}

After applying Minkowski inequality, to obtain the estimate, it is enough to show that
\begin{align*}
D^{-1}\int_{\R} \left(\frac{y}{y^2+\theta^2}-\frac{y}{y^2+\theta_{lin}^2}\right)\pa^6_x f(x-y)dy
\end{align*}
is in $L^2$ with $L^2-$norm bounded by $\lar \left(||D^{-1}\pa^5_x f||_{L^2}+1\right)$ uniformly in $t$ (for small $t$), $\lambda$ and $\lambda'$.

Let us call $g(x)=D^{-1}\pa^5_x f(x)$. Then we proceed as when investigating the commutators $[D^{-1},\Op(p)]$ but this time working directly with the
kernel
\[k(x,y)=\left(\frac{y}{y^2+\theta^2}-\frac{y}{y^2+\theta_{lin}^2}\right).\]
Then we have to estimate the function
\begin{align*}
P(x)=D^{-1}\int_{\R} k(x,y)D \pa_x g(x-y)dy.
\end{align*}

By direct application of the definition of $D$, it holds that
$\int k(x,y) D\pa_xg(x-y)dy=D\int k(x,y)\pa_x g(x-y)dy-t\int  \pa_xk(x,y)\pa_x g(x-y)dy$ we have
\begin{align*}
P(x)= & \int_{\R} k(x,y) \pa_x g(x-y)dy-tD^{-1} \int_{\R} \pa_xk(x,y) \pa_x g(x-y)dy\\
\equiv & M(x)+tD^{-1}L(x),
\end{align*}
We need to estimate both $M(x)$ and $L(x)$ in $L^2$.

To bound $M(x)$ we first integrate by parts, recalling that $\pa_x g(x-y)=-\pa_y g(x-y)$,
\begin{align}\label{Mterm}
&M(x)=\int_{\R} \pa_y\left(\frac{y}{y^2+\theta^2}-\frac{y}{y^2+\theta_{lin}^2}\right)g(x-y)dy\nonumber\\
&=\int_{\R} \left(\frac{1}{y^2+\theta^2}-\frac{1}{y^2+\theta_{lin}^2}\right)g(x-y)dy\nonumber\\
&+2\int_{\R}y\left(\frac{y+\theta_{lin}\pa_x h'}{\dlx^2}-\frac{y+\theta\pa_x h'(x-y)}{\dl^2}\right)g(x-y) dy\nonumber\\
&\equiv M_1(x)+2M_2(x).
\end{align}

As a matter of fact both $M_1, M_2$ can be  estimated by means our previous lemmas. The outer integrals  can be estimated by brute force as the kernels
decay fast enough. The inner ones,  by using the first and second Taylor polynomial of $h'$, can be split  into terms with the appropriate powers
of $y$ in order to apply   Lemmas~\ref{121in}, \ref {hilbert} as before.

It remains to bound $L(x)$. This function is given by
\begin{align*}
&L(x)=-2\int_{\R}y\left(\frac{\theta\pa_x\theta}{\dl^2}-\frac{\theta_{lin}(\pa^2_x h'y+\pa_x\gamma)}{\dlx^2}\right)\pa_x g(x-y)dy\\
&=-2\int_{\R}y\left(\frac{(\Delta h'\Delta \pa_x h+\gamma \Delta \pa_x h')}{\dl^2}-\frac{\pa^2_x h'\pa_x h y^2+\gamma\pa^2_x h' y }{\dlx^2}\right)\pa_x g(x-y)dy\\
&-2\int_{\R}y\left(\frac{\pa_x\gamma\Delta h'}{\dl^2}-\frac{\pa_x \gamma \pa_xh' y }{\dlx^2}\right)\pa_x g(x-y)dy\\
&-2\int_{\R}y\left(\frac{\gamma\pa_x\gamma}{\dl^2}-\frac{\gamma\pa_x\gamma }{\dlx^2}\right)\pa_x g(x-y)dy
\\ &\equiv
S(x)+ \tilde{S}(x)+\overline{S}(x).
\end{align*}

Firstly, we  will carefully  bound $S(x)$ since the numerators of the terms $\tilde{S}$ and $\overline{S}$ have the same behaviour.

We repeat the trick of observing that  $\pa_x g(x-y)=-\pa_y g(x-y)$ to integrate by parts and obtain that
\begin{align*}
&S(x)\\&=-2\int_{\R}\left(\frac{(\Delta h'+\gamma)\pa_x\Delta h'}{\dl^2}-\frac{\theta_{lin}\pa^2_x h'y}{\dlx^2}\right) g(x-y)dy\\
&-2\int_{\R}y\left(\frac{\pa_x h(x-y)\pa_x\Delta h'+(\Delta h'+\gamma)\pa^2_x h'(x-y)}{\dl^2}\right.\\ &\qquad\left.-\frac{\pa_xh' \pa^2_x h'y+(\theta_{lin})\pa^2_x h'}{\dlx^2}\right) g(x-y)dy\\
&+8\int_{\R}y\left(\frac{\theta\pa_x\Delta h'(y^2+\theta^2)(y+(\Delta h'+\gamma))\pa^2_x h'(x-y)}{\dl^2} \right.\\
&\qquad \left. -\frac{\theta_{lin}\pa_x^2h' y (y^2+\theta_{lin}^2)(y+\theta_{lin})\pa^2_xh'}{\dlx^2}\right)g(x-y)dy\\
&\equiv -2S_1(x)-2S_2(x)+8S_3(x).
\end{align*}

To bound $S_1(x)$ we split it in the following way
\begin{align}\label{S1}
&S_1(x)=\int_{\R}\frac{\theta\pa_x\Delta h'-\theta_{lin}\pa^2_x h' y}{\dl^2}g(x-y)dy\nonumber\\
&+\int_{\R}(\theta_{lin})\pa^2_x h' y \left(\frac{1}{\dl^2}-\frac{1}{\dlx^2}\right)g(x-y)dy\nonumber\\
&\equiv S_{11}(x)+S_{12}(x).
\end{align}
To bound $S_{11}(x)$ we split into two terms
\begin{align*}
&S_{11}(x)\\&=\int_{\R}\frac{\Delta h' \pa_x\Delta h' -\pa_x h'\pa^2_xh' y^2}{\dl^2}g(x-y)dy+\int_{\R} \frac{\gamma(\pa_x\Delta h' -\pa^2_x h'y)}{\dl^2}g(x-y)dy\\
\sim &\int_{\R}\frac{\Delta h' \pa_x\Delta h' -\pa_x h'\pa^2_xh' y^2}{\dl^2}g(x-y)dy.
\end{align*}
The last line follows from the bound $(\pa_x\Delta h'-\pa^2_x h' y)\precsim C_2(|y|+1)$ and  lemma~\ref{In}.

The next term is more complicated as in principle the numerator scales as $y^2$ which is not enough to apply our lemmas.
As before we Taylor $\Delta_h$ up to second order and differentiate
to obtain $\Delta h' \pa_x \Delta h'=pa_x h' \pa^2_x h' y^2+G(\pa^j_x h')|y|^3+ C|y|^\alpha$. Being
explicit,

\[G=\pa_x h' \pa_x^3 h'+(\pa_x^2 h')^2+\pa^2_x h' \pa^3_x h')  \]

which is uniformly bounded in $x$ since $h' \in H^4$.
Since terms of the type

\[ \int_{R}  \frac{G(h')}{(y^2 +\theta_{lin}^2)^2} g(x-y) \]

are estimated like our terms $k^{ji}(x)$ we can subtract them freely. Hence

\[ |S_{11in} (x)| \le  \int  \frac{|y|^{3+\alpha}} {(y^2+\theta^2)} dy + G(h')(x) \int y^3 \left(\frac{1}{(y^2 +\theta_{lin}^2)^2} -\frac{1}{(y^2 +\theta^2)^2}\right) .\]

$S_{12}(x)$  is easy by now. It
can be estimated as $M_2(x)$ in \eqref{Mterm}. This finishes the estimation of $S_1(x)$.

To bound $S_2(x)$ we split it into two terms
\begin{align}\label{S2}
&S_2(x)\nonumber\\&=\int_{\R}y\left(\frac{\pa_xh'(x-y)\pa_x\Delta h' +\theta\pa^2_xh'(x-y)-\pa_x h' \pa^2_x h' y -(\pa_xh' y +\gamma)\pa^2_x h'}{\dl^2}\right)g(x-y)dy\nonumber\\
+& \int_{\R}y\left(\pa_x h' \pa^2_x h'y+(\pa_xh' y +\gamma)\pa^2_x h'\right)\left(\frac{1}{\dl^2}-\frac{1}{\dlx}\right)g(x-y)dy\nonumber\\
\equiv & S_{21}(x)+S_{22}(x).
\end{align}

The term $S_{22}$ has the correct behaviour in powers of $y$ and $\gamma$ to deal with
them $S_{21}(x)$ we add freely a term

\[ \frac{1}{2} \pa_x^3 h'(x) \int \frac{y^3} {(y^2+\theta_{lin}^2)} dy \]
and proceed exactly as with $S_{11}$.

Finally, it remains to bound $S_3(x).$  Since the computation are longer but no new idea is needed we skip the details.

Then we have achieved the conclusion of lemma \ref{descomposicion}.

\subsection{Estimates for the velocity}\label{eftv}

\begin{lemma}\label{veloLinftyb} Let $\ub$ be like in expression \eqref{velo} with $f$ and $\ep=ct$ in theorem \ref{existencialocal}. Then $\ub\in L^\infty(\R^2)$ and
$$||\ub(\cdot,t)||_{L^\infty(\R^2)}\leq P(||f||_{H^4})$$
for some smooth function $P$.
\end{lemma}

\begin{proof}

In this proof $C$ stands for a constant that may depend on $||f||_{H^4}$ and on the regularity of $c(x,t)$. The velocity in \eqref{velo} reads
\begin{align*}
\ub(\xb)=\frac{1}{\pi}P.V.\int_{\R}\frac{1}{2}\int_{-1}^1 \frac{x_1-y}{|(x_1-y, x_2-f(y)-\ep(y)\lambda')|^2} (1,\pa_x f(y)+\pa_x\ep(y)\lambda')d\lambda' dy.
\end{align*}
And evaluating in at $(x,f(x)+\lambda), (x,\lambda) \in \mathbb{R}^2$, we have that
\begin{align*}
\ub(x, \lambda+f(x))=\frac{1}{\pi}P.V.\int_{\R}\frac{1}{2}\int_{-1}^1 \frac{x-y}{|(x-y, \Delta f(x,y)+(\lambda-\ep(y)\lambda'))|^2}(1,\pa_x, f(y)+\pa_x\ep(y)\lambda')d\lambda'dx'
\end{align*}
with $\Delta f(x,y)=f(x)-f(y)$. Next we check that the integral
\begin{align}\label{valor}
I(x)=&P.V.\int_{\R}\frac{x-y}{(x-y)^2+(\Delta f(x,y)+\lambda-\ep(y)\lambda')^2}dy\nonumber\\&
=P.V.\int_{\R}\frac{y}{(y^2+(\Delta f(x,x-y)+\lambda-\ep(x-y)\lambda')^2}dy
\end{align}
belongs to $L^\infty(dx)$ uniformly in $\lambda\in \R$ and $\lambda'\in [-1,1]$. In order to do it we split \eqref{valor} into two parts
\begin{align*}
&I_1(x)=P.V.\int_{\R}\frac{y}{y^2+(\Delta f(x,x-y)+\lambda-\ep(x-y)\lambda')^2}-\frac{y}{y^2+((\pa_xf(x)+\pa_x\ep(x)\lambda')y+\lambda-\ep(x)\lambda')^2}dy\\
&I_2(x)=P.V.\int_{\R}\frac{y}{y^2+((\pa_x f(x)+\pa_x\ep(x)\lambda')y+\lambda-\ep(x)\lambda')^2}dy.
\end{align*}
We will denote
\begin{align*}
A_\lambda'=\pa_xf(x)+\pa_x\ep(x)\lambda',&& \sigma_{\lambda'}=\frac{1}{1+A_{\lambda'}^2}, && \gamma=\lambda+\ep(x)\lambda'.
\end{align*}

Thus
\begin{align*}
&I_{2}(x)=\sigma_{\lambda'}P.V.\int_{\R}\frac{y}{(y+\sigma_{\lambda'}
A_\lambda'\gamma)^2+\gamma^2\sigma_{\lambda'}^2}dy\\
&=\sigma_{\lambda'}P.V.\int_{\R}\frac{y+\sigma_{\lambda'}A_{\lambda'}\gamma}{(y+\sigma_{\lambda'}A_{\lambda'}\gamma)^2+\gamma^2\sigma_{\lambda'}^2}dy
-\sigma_{\lambda'}\int_{\R}\frac{\sigma_{\lambda'}A_{\lambda'}\gamma}{(y+\sigma_{\lambda'}A_{\lambda'}\gamma)^2+\gamma^2\sigma_{\lambda'}^2}dy.
\end{align*}
The first integral on the right hand side of the previous equation is equal to zero. The second one is a bounded integral for every value of $\gamma\in \R$ and $\lambda'\in [-1,1]$.

In order to bound $I_1(x)$ we split it into two terms
\begin{align*}
&I_{11}(x)=\int_{|y|<1}\frac{y}{y^2+(\Delta f(x,x-y)+\lambda-\ep(x-y)\lambda')^2}-\frac{y}{y^2+(A_{\lambda'}y+\gamma)^2}dy\\
&I_{12}(x)=P.V.\int_{|y|>1}\frac{y}{y^2+(\Delta f(x,x-y)+\lambda-\ep(x-y)\lambda')^2}-\frac{y}{y^2+(A_{\lambda'}y+\gamma)^2}dy.
\end{align*}
To bound $I_{12}(x)$ we consider $I_{121}(x)$ and $I_{122}(x)$ with
\begin{align*}
&I_{121}(x)=P.V.\int_{|y|>1}\frac{y}{y^2+(\Delta f(x,x-y)+\lambda-\ep(x-y)\lambda')^2}dy\\&=P.V.\int_{|y|>1}\left(\frac{y}{y^2+(\Delta f(x,x-y)+\lambda-\ep(x-y)\lambda')^2}-\frac{y}{y^2+\lambda^2}\right)dy
\end{align*}
and
\begin{align*}
&I_{122}(x)=P.V.\int_{|y|>1}\frac{y}{y^2+(A_{\lambda'}y+\gamma)^2}dy\\ &=P.V.\int_{|y|>1}\left(\frac{y}{y^2+(A_{\lambda'}y+\gamma)^2}-\frac{y}{(1+A_{\lambda'}^2)y^2+\gamma^2}\right)dy.
\end{align*}
Then
\begin{align*}
|I_{121}(x)|\leq C\int_{|y|>1}\frac{(1+|\lambda|)}{|y|(y^2+\lambda^2)}dy\leq C,
\end{align*}
and
\begin{align*}
|I_{122}(x)|\leq C \int_{|y|>1}\frac{|\gamma|}{(y^2+\gamma^2)}dy\leq C.
\end{align*}
To bound $I_{11}(x)$ we notice that
\begin{align*}
&(A_{\lambda'}y+\gamma)^2-(\Delta f(x,x-y)+\lambda-\ep(x-y)\lambda')^2\\&
=(A_{\lambda'}y +\gamma-\Delta f(x,x-y)-\lambda+\ep(x-y)\lambda')(A_{\lambda'}y+\Delta f(x,x-y)+\gamma+\lambda-\ep(x-y)\lambda').
\end{align*}
In addition,
\begin{align*}
A_{\lambda'}y +\gamma-\Delta f(x,x-y)-\lambda+\ep(x-y)\lambda'=\pa_xf(x)y-\Delta f(x,x-y)+(-\ep(x)+\ep(x-y)+\pa_x\ep(x)y)\lambda'
\end{align*}
Then
\begin{align*}
&\left|(A_{\lambda'}y+\gamma)^2-(\Delta f(x,x-y)+\lambda-\ep(x-y)\lambda')^2\right|\leq C y^2,
\end{align*}
and
\begin{align*}
&\left |A_{\lambda'}y+\Delta f(x,x-y)+\gamma+\lambda-\ep(x-y)\lambda'\right|\leq C(|y|+|2\lambda-(\ep(x)+\ep(x-y))\lambda')\\
&\leq C(|y|+2|(\lambda-\ep(x)\lambda'|+|\ep(x)-\ep(x-y)||\lambda'|\leq C(|y|+|\gamma|).
\end{align*}
\begin{align}\label{estadeaqui}
|I_{11}|\leq C\int_{|y|<1}\frac{|y|^3(|y|+|\gamma|)}{(y^2+(A_{\lambda'}y+\gamma)^2)(y^2+(\Delta f (x,x-y)+\lambda-\ep(x-y)\lambda')^2)}dy.
\end{align}
Since we can bound
\begin{align*}
\frac{|y|^4}{(y^2+(A_{\lambda'}y+\gamma)^2)(y^2+(\Delta f (x,x-y)+\lambda-\ep(x-y)\lambda')^2)}\leq C,
\end{align*}
the first integral in \eqref{estadeaqui} is easy to bound. In addition for $|\gamma|>4 (||f||_{L^\infty}+||\ep||_{L^\infty})$ we have that
$$(\Delta f+\lambda-\ep(x-y)\lambda')^2=(\Delta f+(\ep(x)-\ep(x-y))\lambda'+\gamma)^2\geq \frac{\gamma^2}{4},$$
so that, in this range
\begin{align*}
\frac{|y|^3|\gamma|}{(y^2+(A_{\lambda'}y+\gamma)^2)(y^2+(\Delta f (x,x-y)+\lambda-\ep(x-y)\lambda')^2)}\leq \frac{|y||\gamma|}{y^2+\frac{1}{4}\gamma^2},
\end{align*}
and we can estimate the second integral. In the range $|\gamma|\leq 4||A_{\lambda'}||_{L^\infty}$ we can apply lemma~\ref{In}. This concludes the proof of the bound of $I(x)$.

The  bound
\begin{align*}
J(x)=P.V.\int_{\R}\frac{y}{(y^2+(\Delta f(x,x-y)+\lambda+\ep(x-y)\lambda')^2}(\pa_x f(x-y)+\pa_x\ep(x-y)\lambda')dy.
\end{align*}

follows similar steps.

Then we have achieved the conclusion of lemma \ref{veloLinftyb}.
\end{proof}

\begin{lemma}\label{falta} Let $f$  and $\ep=c t$ be as in theorem \ref{existencialocal}. Then
the velocity $u_c^\sharp(x,\lambda)$  satisfies
\begin{align*}
|u_c^\sharp(x,\lambda)-u_c^\sharp(x,\lambda')|\leq Ct,
\end{align*}
where $C$ depends on $||f||_{H^4}$ but it does not depend on either $x$, $\lambda$ or $\lambda'$.
\end{lemma}

\begin{proof}

Recall that

\[ u_c^\sharp(x,\lambda)=\int_{-1}^1 k_\theta(x,y) \partial_x \theta dy d \lambda'.\]

Then it is enough to prove that the function $$h(x,\lambda)=\int_{\R}\frac{y}{y^2+\theta^2}\partial_x \theta dy$$ satisfies
$||\pa_\lambda h||_{L^\infty(dx)}\leq  |||f|||t $ for every $\lambda$. In addition, by Sobolev's embedding we reduce the problem to prove that  $||\pa_\lambda h||_{H^1(dx)}\leq |||f||| t $. We notice that
\begin{align*}
\pa_\lambda h(x,\lambda)= \ep(x)\int_{\R}\frac{2y\theta}{\dl^2}\pa_x \theta.
\end{align*}
The $L^2-norm$ of this function can be bounded in the same way we bounded $k^{31d}_2(x)$ in \eqref{k31d} and $N(x)$ in \eqref{N}. By taking a derivative with respect to $x$ we have
\begin{align*}
&\pa_x\pa_\lambda h(x,\lambda)= \ep_x \int_{\R}\frac{2y\theta}{\dl^2}\pa_x \theta+ \ep \big( \int_{\R}\frac{2y\theta}{\dl^2}\pa^2_x \theta+
\int_{\R}\frac{2y}{\dl^2}(\pa_x \theta)^2-2\int_{\R}\frac{2y\theta^2 }{\dl^3}\pa_x \theta \big).
\end{align*}

The first two terms can be bounded exactly as we bound $N(x)$ in \eqref{N}. The third can be bounded in the same way that $k^{31d}_2(x)$ in \eqref{k31d} and $N(x)$ in \eqref{N}. The last term can be bounded by using a similar strategy, though a different pseudodifferential operator arises.

\end{proof}

\subsection{Estimates on the coefficient $a(x,t)$.}

The function $a(x,t)$ is given by the expression
\begin{align*}
a(x,t)=P.V.\int_{\R}K(x,y)dy
\end{align*}
where the principal value is taken at $0$ and at the infinity. We need to prove the next lemma
\begin{lemma}\label{alemma} Let $f$ and $\ep=c t$ be as  in theorem \ref{existencialocal}. The following estimate holds:
\begin{align*}
||\pa_x a||_{H^2}\leq |||f|||.
\end{align*}
\end{lemma}
\begin{proof} We will use the same convection than  in appendix \ref{ordenmasaltosection}.  Recall that
we can express

\[ a(x,t)=\int_{-1}^1\int_{-1}^1 \frac{y}{y_2+\theta^2} dyd\lambda d\lambda' \]

and thus,  if we set

\[ \int  k_\theta^{ji}(x,y) dy=\bar{k}^{ij},\]

we need to show that for $j=1,2,3,4$  and every $i$.
\[\bar{k}^{ij} \in L^2 \]
The estimation in fact are often easier than in subsection~\ref{Afirstestimates}.  All the outer
integrals are automatic  since the terms $\partial_{j-i} \Delta \theta$ do not appear in the numerators. The inner
integrals can also be dealt with.

We sketch the case with $j=3$ which is the most singular and leave the rest to the energetic reader.  From the fourth terms
corresponding to \eqref{k31d} we can directly bound the inner integrals of $\bar{k}_1^{31},\bar{k}_3^{31},\bar{k}_4^{31}$ by means
of lemma~\ref{yelambda} whereas $\bar{k}_2^{31}$ is equal to $N$ in \eqref{N}, with  $g= \partial_x \theta^3$ directly. The other two terms
are  easily bounded. Namely we have the bounds,

\[ |\bar{k}^{32}_{in}(x) |=\int_{|y|\le 1} \frac {y  \theta \partial_x \theta (\partial_x \theta)^2+\theta \partial_x^2 \theta}
 {(y^2+\theta^2)^3}dy \le \frac{y^5+y \gamma^4}{(y^2+\theta^2)^3}\]

\[|\bar{k}_{in}^{33}(x)|=\int \frac {y \theta^3+\partial_x \theta^3}{(y^2+\theta^2)^4}dy \le C \int \frac{|y|^7+ y \gamma^6}{{(y^2+\theta^2)^4}} dy. \]

For $a=3,4$, the terms with $y|\gamma|^{2a-2}$ in the numerator are directly bounded by lemma~\ref{In}, whereas the term with $|y|^{2a-1}$ is bounded in the usual two
steps. Firstly, we use Lemma~\ref{121in} to replace $\theta$ by $\theta_{lin}$ and then obtain the estimate with  lemma~\ref{hilbert}.
The rest of the derivatives are bounded in the same way.

\end{proof}

\section{Symbols and Estimates}

\subsection{Fourier transform of $K^{c,c'}_{A}(x)$}\label{apendice5}

\begin{lemma}\label{fkcca}Let $K^{c',c}_{A}$ the function given by the expression
\begin{align*}
K^{c,c'}_A(x)=\frac{1}{4\pi}\int_{-1}^{1}\int_{-1}^1 \frac{y}{y^2+(Ay+c't\lambda' y +ct(\lambda-\lambda'))^2}d\lambda d\lambda'.
\end{align*}
Then
\begin{align*}
\widehat{K^{c,c'}_{A}}(\xi)=\frac{-i\text{sign}(\xi)}{4\cdot 2\pi c t |\xi|}\int_{-1}^{1}
\left(2-e^{2\pi\sigma_{\lambda'}ct|\xi|(-1-\lambda')(A_{\lambda'}i\text{sign}(\xi)+1)}
-e^{2\pi\sigma_{\lambda'}ct|\xi|(1-\lambda')(A_{\lambda'}i\text{sign}(\xi)-1)}\right),
\end{align*}
and
\begin{equation}\label{FK0}
\widehat{K}^{c,0}_A(\xi)=\frac{-i\sign(\xi)}{2\pi c t|\xi|}\left(1+\frac{1}{4\pi c t |\xi|}\left(e^{-4\pi\sigma c t |\xi|}\left(\cos(4\pi\sigma A c t |\xi|)-A\sin(4\pi\sigma A c t |\xi|)\right)-1\right)\right).
\end{equation}
where
\begin{align*}
A_{\lambda'}=A+c't\lambda',&& \sigma_{\lambda'}=\frac{1}{1+A_{\lambda'}^2}.
\end{align*}

\end{lemma}

\begin{proof}
We first notice that, if we call
\begin{align*}
\gamma=ct|\lambda-\lambda'|,
\end{align*}
we can write
\begin{align*}
&y^2+(Ay+c'\lambda'ty+ct(\lambda-\lambda'))^2= (1+(A+c'\lambda't)^2)y^2+c^2t^2(\lambda-\lambda')^2+2(A+c'\lambda't')yct(\lambda-\lambda')\\
&=\frac{1}{\sigma_{\lambda'}}\left(y^2+2\sigma_{\lambda'}A_{\lambda'}\gamma\text{sign}(\lambda-\lambda')y+\sigma_{\lambda'}\gamma^2\right)
=\frac{1}{\sigma_{\lambda'}}\left((y+\sigma_{\lambda'}A_{\lambda'}\gamma\text{sign}(\lambda-\lambda'))^2+\sigma_{\lambda'}\gamma^2
-\sigma_{\lambda'}^2A_{\lambda'}^2\gamma^2\right)\\
&=\frac{1}{\sigma_{\lambda'}}\left((y+\sigma_{\lambda'}A_{\lambda'}\gamma\text{sign}(\lambda-\lambda'))^2+\sigma_{\lambda'}^2\gamma^2\right).
\end{align*}
And then
\begin{align*}
K^{c,c'}_A(x)=\frac{1}{4\pi}\int_{-1}^1\int_{-1}^1\frac{\sigma_{\lambda'}y}{(y+\mu)^2+\nu^2}d\lambda d\lambda'
\end{align*}
with $\mu=\sigma_{\lambda'}A_{\lambda'}\gamma \text{sign}(\lambda-\lambda')$ and $\nu=\sigma_{\lambda'}\gamma.$

With these notations Fourier transforms are easier as they resemble the relation between Poisson and Abel kernels. We can compute that
\begin{align*}
&\mathcal{F}\left[\frac{y}{(y+\mu)^2+\nu^2}\right](\xi)=
\mathcal{F}\left[\frac{y+\mu}{(y+\mu)^2+\nu^2}\right](\xi)-\mu\mathcal{F}\left[\frac{1}{(y+\mu)^2+\nu^2}\right](\xi)\\
&=e^{2\pi i\xi\mu}\left(\mathcal{F}\left[\frac{y}{y^2+\nu^2}\right](\xi)-\mu\mathcal{F}\left[\frac{1}{y^2+\nu^2}\right](\xi)\right)
=e^{2\pi i\xi\mu}e^{-2\pi\nu|\xi|}\pi\left(-i\text{sign}(\xi)-\frac{\mu}{\nu}\right).
\end{align*}
Therefore
\begin{align}\label{alprincipio}
\widehat{K}^{c,c'}_{A}(\xi)=\frac{1}{4}\int_{-1}^{1}\int_{-1}^1\sigma_{\lambda'}e^{2\pi i \sigma_{\lambda'}A_{\lambda'}ct(\lambda-\lambda')\xi}e^{-2\pi\sigma_{\lambda'}ct|\lambda-\lambda'||\xi|}
\left(-i\text{sign}(\xi)-A_{\lambda'}\text{sign}(\lambda-\lambda')\right)d\lambda d\lambda'.
\end{align}

For visualization set  $ct\sigma_\lambda' 2 \pi i |\xi|=a$ in the next estimates
\begin{align*}
&\int_{-1}^1  e^{a A_{\lambda'}i \sign(\xi) (\lambda-\lambda')}e^{- a|\lambda-\lambda'| }
\left(-i\sign(\xi)-A_{\lambda'}  \text{sign}(\lambda-\lambda')\right)d\lambda\\
&=\int_{-1}^{\lambda'}e^{ a (1+i A_{\lambda'} \sign(\xi))  (\lambda-\lambda')}
\left(-i\text{sign}(\xi)+A_{\lambda'}\right)d
\lambda\\
&+\int_{\lambda'}^{1}e^{a(-1+iA_{\lambda'} \sign(\xi) )(\lambda-\lambda')}
\left(-i\text{sign}(\xi)-A_{\lambda'}\right)d\lambda,\\
&=\frac{-i\sign(\xi)+A_{\lambda'}}{a(1+iA_{\lambda'}\text{\sign}(\xi))}
\left(1-e^{a(-1-\lambda')(A_{\lambda'}i\text{sign}(\xi)+1)}\right)\\
&+\frac{-i\text{sign}(\xi)-A_{\lambda'}}{a(iA_{\lambda'}\text{\sign}(\xi)-1)}
\left(e^{a (1-\lambda')(A_{\lambda'}i\text{sign}(\xi)-1)}-1\right)\\
&=-i\text{sign}(\xi)\frac{1}{a}\left(2-e^{a(-1-\lambda')(A_{\lambda'}i\text{sign}(\xi)+1)}
-e^{a(1-\lambda')(A_{\lambda'}i\text{sign}(\xi)-1)}\right).
\end{align*}

Thus
\begin{align*}
\widehat{K}^{c,c'}_{A}(\xi)=\frac{-i\text{sign}(\xi)}{4\cdot 2\pi c t |\xi|}\int_{-1}^{1}
\left(2-e^{2\pi\sigma_{\lambda'}ct|\xi|(-1-\lambda')(A_{\lambda'}i\text{sign}(\xi)+1)}
-e^{2\pi\sigma_{\lambda'}|\xi|(1-\lambda')(A_{\lambda'}i\text{sign}(\xi)-1)}\right)d \lambda'.
\end{align*}

This proves the first identity of lemma \ref{fkcca}.

By taking $c'=0$ we have that
\begin{align}
\widehat{K}^{c,0}_{A}(\xi)=\frac{-i\text{sign}(\xi)}{4\cdot 2\pi c t |\xi|}\int_{-1}^{1}
\left(2-e^{2\pi\sigma ct|\xi|(-1-\lambda')(A i\text{sign}(\xi)+1)}
-e^{2\pi\sigma |\xi|(1-\lambda')(A i\text{sign}(\xi)-1)}\right) d\lambda',
\end{align}
where $\sigma=\frac{1}{1+A^2}$. Integrating in $\lambda'$ yields

\begin{align*}
I=\int_{-1}^1 e^{-2\pi \sigma c t (1+\lambda')|\xi|(1+iA\sign(\xi))}d\lambda'=\frac{1}{2\pi\sigma c t|\xi| (1+iA\sign(\xi))}\left(1-e^{-4\pi \sigma c t (1+iA\sign(\xi))|\xi|}\right)
\end{align*}
and
\begin{align*}
II=\int_{-1}^1e^{2\pi \sigma c t (-1+\lambda')|\xi|(1-iA\sign(\xi))}d\lambda'=\frac{1}{2\pi\sigma c t|\xi| (1-iA\sign(\xi))}\left(1-e^{-4\pi \sigma c t (1-iA\sign(\xi))|\xi|}\right).
\end{align*}
However
\begin{align*}
I+II=\frac{2}{2\pi\sigma ct|\xi|} \text{Re}\left(I\right).
\end{align*}
Now if we notice that

\[ \text{Re}\left(\frac{1}{1+iA\sign(\xi)}\right)=\sigma, \quad \text{Re}\left(\frac{e^{-4\pi\sigma c t i A \xi}}{1+iA\sign(\xi)}\right)=\sigma\left(\cos\left(4\pi\sigma A c t |\xi|\right)-A\sin\left(4\pi\sigma c t A |\xi|\right)\right),\]
  equality \eqref{FK0} follows.

\end{proof}

\subsection{Estimations of the various symbols }\label{Estimates}
In this section we will use the notation $t\xi=\tau$.   In the following estimates:
\begin{enumerate}
\item $A$ is function in $H^3$.
\item The function $c$ is as in the statement of theorem \ref{existencialocal}.
\item We can consider that the time $t<<1.$
\end{enumerate}

To alleviate the notation we introduce the following auxiliar function
\begin{equation}\label{hdef}
\begin{aligned}h(x,\tau)&=\frac{1}{c\tau}\left\{1+\frac{1}{4\pi c\tau}\left(e^{-4\pi c\tau\sigma}\left(\cos(4\pi c\tau\sigma A)-A\sin(4\pi c\tau\sigma A)\right)-1\right)\right\}\\&=
\frac{1}{c\tau}(1+h_2(\tau))
\end{aligned}
\end{equation}
with
\[h_2(x,\tau)=\frac{1}{4\pi c\tau}\left(-1+e^{-4\pi \sigma c \tau}\left(\cos(4\pi\sigma A c\tau)-A\sin(4\pi \sigma A c\tau)\right)\right)\]
and $\sigma=\frac{1}{1+A^2}.$

We emphasis that $h$ and $h_2$ depend on $x$ just through $A$ and $c$.

Notice that \eqref{FK0} implies that
\[ 2\pi i \sign(\xi)\hat{K}^{c,0}_{A}(\xi)=h(x,\tau). \]
We will omit that $h$ depends on $A$ as well.  We study the regularity of the function $h$ in detail.

\begin{lemma} \label{identities}The following identities holds:
\begin{equation}
h_2(x,\tau)=-\int_{0}^1e^{-4\pi\sigma c\tau\tau_1}\cos(4\pi\sigma A c \tau\tau_1)d\tau_1,\label{little}
\end{equation}
\begin{align}\label{big}
&h(x,\tau)=4\pi\sigma\int_{0}^1\int_{0}^1 e^{-4\pi\sigma c \tau \tau_1\tau_2}\tau_1\left(\cos(4\pi \sigma A c \tau\tau_1\tau_2)+A\sin(4\pi \sigma A c \tau\tau_1\tau_2)\right)d\tau_2d\tau_1.
\end{align}
\end{lemma}

The following estimate not only gives us how the symbol $p$ grows but it also
implies lemma~\ref{Upbound}, the key in showing that
$p_+$ is positive for small times.

\begin{lemma}\label{cotakc0}Let $\varphi,h$ defined as above. The following estimates hold:
\begin{equation}
\label{cotakc01}
\begin{aligned}
&|h(x,\tau) |\leq \frac{\frac{1}{c}}{1+\tau}+\frac{2|A|+5+8\pi}{1+(\tau)^2},
\\ &|\pa_A h(x,\tau)|+|\pa_c h(x,\tau)| \le \frac{\lar}{1+\tau},
\end{aligned}
\end{equation}
and
\begin{equation} \label{cotakc02}
\left(\varphi(\tau)-h(x,\tau)\right)\geq \frac{1-\frac{1}{c}}{1+\tau}+\frac{B-(2|A|+5+8\pi)}{1+\tau^2}.
\end{equation}

\end{lemma}

\begin{proof}
We start with  $\tau=t|\xi|\geq 1$. Since $c \geq  1$ we have that
\begin{align*}
\frac{1}{c\tau}\left(1+\frac{1}{4\pi c \tau}\left(e^{-4\pi\sigma c\tau}\left(\cos(4\pi \sigma A c \tau)-A\sin(4\pi \sigma A c \tau)\right)-1\right)\right)\leq \frac{1}{c\tau}+\frac{|A|+2}{\tau^2}.
\end{align*}
But since for  $\tau\geq 1$, $\frac{1}{\tau^2}\leq \frac{2}{1+\tau^2}$ and $\frac{1}{\tau}=\frac{1}{1+\tau}+\frac{1}{\tau}-\frac{1}{1+\tau}=\frac{1}{1+\tau}+ \frac{1}{\tau +\tau^2}\leq \frac{1}{1+\tau}+\frac{1}{1+\tau^2}$. Then
\begin{align*}
\frac{1}{c\tau}\left(1+\frac{1}{4\pi c \tau}\left(e^{-4\pi\sigma c\tau}\left(\cos(4\pi \sigma A c \tau)-A\sin(4\pi \sigma A c \tau)\right)-1\right)\right)\leq \frac{\frac{1}{c}}{1+\tau}+\frac{2|A|+5}{1+\tau^2}.
\end{align*}
For $\tau\leq 1$ we use the expression \eqref{alprincipio} to get uniform bounds on $h$ and its derivatives.

 The derivatives of $h$ for $\tau\ge 1$ are controlled by
\[ |\pa_c h| \le C\left( \frac{1}{c^2 \tau}+\frac{1}{c^3\tau^2}(1+|A|)+ \frac{1}{c^2 \tau^2} \tau(1+|A|)+ \tau (|A|+|A|^2) \le \frac{\lar}{\tau}\right),  \]

\[ |\pa_A h| \le  C\left( \frac{1}{c^2 \tau^2} \tau( |A|+|A|^2)\right). \]

 If we combine \eqref{cotakc01} with the definition of $\varphi$,  \eqref{cotakc02} follows as well.
\end{proof}

\begin{lemma}\label{derivativesh} Let $k=1,2,3$
\[ |\pa_\tau h(x,\tau)|+|\pa_{\tau,A} h|+ |\pa_{\tau,c} h| \le \lar \frac{1}{1+\tau^2}, \]
\[ \pa^k_x h(x,\tau) \le \lar \left(\frac{ |\pa_x c|}{1+\tau}+\frac{|\pa^k_x c|+|\pa^k_x A|}{1+\tau^2}\right),\]
\[ \pa^k_x \pa_\tau h(x,\tau) \le \lar \left(\frac{ |\pa_x c|}{1+\tau^2}+\frac{|\pa^k_x c|+|\pa^k_x A|}{1+|\tau|^3}\right),\]
where the constant $C$ depends on $|\pa^{i}_x c|+ |\pa_x^{i} A |$ for $i<k$.
\end{lemma}
\begin{proof}

For $\tau<1$ we use lemma \ref{identities} and the result follows. For $\tau>1$, from the expression
\begin{align}\label{explicacion}
h=\frac{1}{c\tau}-\frac{1}{4\pi c^2\tau^2}+\frac{1}{4\pi c^2\tau^2}e^{-4\pi \sigma \tau}\left(\cos(4\pi c\sigma A\tau)-A\sin(4\pi c\sigma A\tau)\right),
\end{align}
we see that $|\pa_\tau h|\leq \lar (1+\tau^2)^{-1}$. In order to bound the derivatives on $\sigma$ and $A$ of $\pa_\tau h$ we see that the two first terms in \eqref{explicacion} do not cause any difficulty. The third one in \eqref{explicacion} brings down a factor $\tau$ but since $c>1$ and $\sigma\geq \frac{1}{1+||A||_{L^\infty}}$ we still have exponential decay. In order to bound the derivatives with respect to $x$ of $h$ and $\pa_\tau h$ a similar argument applies.

\end{proof}

We recall that
\begin{align*}
p=&2\pi i \xi \hat{K}^{c,0}_{A}, &
p_b=&2\pi i \xi \left(p_{main}-p\right),&
p_{\textrm{good}}=&\frac{1}{|\xi|}p-\varphi(t|\xi|),
 \quad \text{and} \quad  p_+=-(1+|\xi|) p_{\textrm{good}}.\end{align*}
\begin{lemma}\label{pestimates}  Given $t>0$, the symbols $p_b, t\pa_x p_{main}, p_{good} \in \Hw$ with the following
estimates.
\begin{itemize}
\item [i)] $ \|p_b\|_{1,1}+\|p_{good}\|_{1,1} \le \lar$
\item [ii)] $ \|tp_{main}\|_{1,1}+\|t \pa_{x} p_{main}\|_{1,1}+\|\partial_\xi p_{main} \|_{1,0} \le \lar $

\end{itemize}
\end{lemma}

\begin{proof}

We start with the $L^\infty$ estimates (no $x$ derivatives).
The estimation of $p_b$ itself is the most subtle so we address it  first.

\begin{enumerate}
\item Estimation of the $L^\infty$ norm of $p_b$.

The fundamental theorem of calculus tell us that,
\begin{align*}
&p-p_{main}
\\ & =\frac{-i\sign(\xi)}{4\cdot 2\pi c t |\xi|}\int_{-1}^1\int_{0}^1\frac{d}{ds} e^{-2\pi\sigma_{s\lambda}ct|\xi|(1+\lambda)(1+iA_{s\lambda}sign(\xi))} dsd\lambda\\
&+\frac{-i\sign(\xi)}{4\cdot 2\pi c t |\xi|}\int_{-1}^1\int_{0}^1\frac{d}{ds} e^{-2\pi\sigma_{s\lambda}ct|\xi|(1-\lambda)(1-iA_{s\lambda}sign(\xi))} ds d\lambda\\
&\equiv T_1+T_2.
\end{align*}
Now we use the chain rule and that $\pa_s A_{s\lambda}=c't\lambda$ and $\pa_s \sigma_{s\lambda}=-2\sigma_{s\lambda}^2A_{s\lambda}c't\lambda$ to obtain that
\begin{align}\label{2pit1}
&2\pi i \xi T_1\nonumber
\\
&=\frac{1}{4ct}\int_{-1}^1 \int_{0}^1 2\sigma_{s\lambda}^2A_{s\lambda}c't\lambda 2\pi c t |\xi|(1+\lambda)(1+iA_{s\lambda}\sign(\xi))
e^{-2\pi\sigma_{s\lambda} c t|\xi|(1+\lambda)(1+iA_{s\lambda}\sign{\xi}) } ds d\lambda\nonumber\\
&+\frac{1}{4ct}\int_{-1}^1\int_{0}^1 c't\lambda i\sign(\xi) (-2\pi\sigma_{s\lambda} c t|\xi|(1+\lambda))e^{-2\pi\sigma_{s\lambda} c t|\xi|(1+\lambda)(1+iA_{s\lambda})\sign(\xi)}dsd\lambda\\
&\equiv T_{11}+T_{12}.
\end{align}
Now observe that the dangerous  $t$ in the denominator cancels out in both $T_{11}, T_{12}$. Thus, we  are entitled to take modulus and
obtain the elementary bound:
\begin{align*}
|2\pi i \xi T_1|\leq C\int_{-1}^1\int_{0}^1 2\pi\sigma_{s\lambda}c t |\xi|(1+\lambda) e^{-2\pi\sigma_{s\lambda}ct|\xi|(1+\lambda)}dsd\lambda\leq \lar.
\end{align*}
The estimation involving $T_2$ is exactly analogous and thus $\|p_b\|_{L^\infty} \le \lar$.

\item Estimation of the $L^\infty$-norm of  $\pa_\xi(2\pi i \xi( \hat{K}^{c,0}_A(\xi)-\hat{K}^{c,c'}_A(\xi)))$. We recall that
    \begin{align*}
&2\pi i \xi (\hat{K}^{c,0}_A(\xi)-\hat{K}^{c,c'}_A(\xi))\\
&=\frac{1}{4 c t}\int_{-1}^{1}
\left(e^{2\pi\sigma_{\lambda'}ct|\xi|(-1-\lambda')(A_{\lambda'}i\text{sign}(\xi)+1)}
+e^{2\pi\sigma_{\lambda'}|\xi|(1-\lambda')(A_{\lambda'}i\text{sign}(\xi)-1)}\right)d\lambda'\\
&+\frac{1}{4 c t }\int_{-1}^{1}
\left(e^{2\pi\sigma ct|\xi|(-1-\lambda')(A i\text{sign}(\xi)+1)}
+e^{2\pi\sigma (1-\lambda')(A i\text{sign}(\xi)-1)}\right)d\lambda'\\
&\equiv U_1+U_2.
\end{align*}
In this case we can bound $U_1$ and $U_2$ separately and it is enough to estimate $U_1$. In addition, we can split $U_1= U_{11}+U_{12}$ with $U_{11}=\frac{1}{4ct}\int_{-1}^1 e^{2\pi\sigma_{\lambda'}ct|\xi|(-1-\lambda')(A_{\lambda'}i\text{sign}(\xi)+1)}d\lambda'$ and it easy to see that the estimation of $U_{11}$ and $U_{12}$ follows similar steps.  We have that
\begin{align}\label{pau11}
&\pa_\xi U_{11} =\frac{1}{4c t}\int_{-1}^1 (-2\pi\sigma_{\lambda'}c t(1+\lambda')(\sign(\xi)+iA_{\lambda'})e^{2\pi\sigma_{\lambda'}ct|\xi|(-1-\lambda')(A_{\lambda'}i\text{sign}(\xi)+1)}d\lambda'\nonumber\\
& =\frac{1}{4}\int_{-1}^1 (-2\pi\sigma_{\lambda'}(1+\lambda')(\sign(\xi)+iA_{\lambda'})e^{2\pi\sigma_{\lambda'}ct|\xi|(-1-\lambda')(A_{\lambda'}i\text{sign}(\xi)+1)}d\lambda',
\end{align}
and then $|\pa_\xi U_{11}|\leq \lar$.

\item The $L^\infty$ estimate for $p_{good}$ follows directly from lemma~\ref{cotakc0} and the definition of $\varphi$.

\item To estimate $p_{main}$ we notice that $p_{main}=p_{b}+p$. We have already proved that $p_{b}$ is bounded and
recall that $p=\xi h(x,\tau)$. Thus $|p(x,\xi)|= \frac{1}{t} \tau h(x,\tau)$ and thus the estimate \eqref{cotakc01} implies that
$ \| tp_{main} \|_{L^\infty} \le \lar$. Similarly, when estimating $\partial_\xi p_{main}$, by our uniform estimate on $p_b$ we are reduce to estimate
$\partial_\xi p$. Now notice that  by the definition of $p$ and the chain rule,
$ \|\partial_\xi p\|_{L^\infty} =\frac{t}{t}\| \partial_\tau (\tau h)\| \le \lar $
where the last bound follows from the lemma~\ref{derivativesh}

\item  Finally, we deal with the derivatives respect to $x$. Firstly, observe that $tp(x,\tau)=\tau h(x,\tau)$ and thus the estimates of the derivatives in $x$ follow directly from those of $h$,  which are explicitly bounded in lemma~\ref{derivativesh}.  Hence we obtain that
\begin{equation}\label{p}  \|t p \|_{1,1}+ \|t \pa_x p \|_{1,1}+ \| \pa_\xi p_b \|_{1,0} \le \lar.
\end{equation}

 Next we look at the $x$ derivatives of $p_{b},\partial_\xi p_{b}$. We have done all the work
as in the expressions, of $T_{11},T_{12},U_1,U_2$ the only difficulty occurs when after the use of chain rule we differentiate $A$ and $c$. Thus we obtain that

\begin{equation}\label{pb}\|p_b\|_{1,1}+ \|\pa_x p_b\|_{1,1} \le \lar. \end{equation}

Since, $\frac{p}{\xi}=h(x,\tau)$ and $\varphi$ is explicit  lemma~\ref{cotakc0}, lemma~\ref{derivativesh} imply readily
the bounds for $\|p_{good}\|_{1,1} \le \lar$ and thus the claim i) follows. Claim ii), which deals with $p_{main}$, is  an straightforward consequence of \eqref{p} and \eqref{pb}.
 \end{enumerate}
\end{proof}

The following lemma is cumbersome as considering $p_+^\frac12$ instead of $p_+$ is less innocent
than it seems.  In fact  here is the only place where the existence of the  constant $c_\infty$ is required.

\begin{lemma} \label{symbolp_+} Let $2|A|+5+8\pi<\frac{B}{2}$. The symbols $p^\frac{1}{2}_+$ and $q=\pa_\xi p^\frac{1}{2}_+$ satisfy that for $0<\ep<1$ it holds that,
\begin{align*}
&\| t^{\frac12} p_+^\frac{1}{2} \|_{1,1} \le \lar,\quad ||\pa_\xi p_+^\frac{1}{2}||_{L^\infty(\R^2)}\leq \lar, \quad \|\pa_x \pa_\xi p_+^\frac{1}{2}\|_{L^\infty(\R^2)} \le \lar, \quad  \sup_{\xi\in\R}\left(\|\pa_x q \|_{H^2}+ \| \pa_{x} q \|_{\dot{H}^{-\ep}}\right) \le \lar,\\ &
||t^\frac{1}{2}\pa^2_x p_{+}^\frac{1}{2}||_{L^\infty(\R^2)}\leq \lar.
\end{align*}
\end{lemma}
\begin{proof}

We give the proof in the case $c\ge 1+\kappa$, and will explain at the end of the proof the modifications for the case $c=1$. We explain first
the $L^\infty$ bounds  then $\pa_xq \in \dot{H}^{-\ep}$ and finally how to control the $x$ derivatives of both symbols.

\begin{enumerate}
\item Since,
\begin{equation}\label{product}
tp_{+}(x,\xi)=(t+\tau)\left(\varphi(\tau)-h(x,\tau)\right)
\end{equation}
it follows from  \eqref{cotakc02}  that
\begin{equation}\label{p+infty}|p_+^\frac{1}{2}|\leq \frac{\lar}{\sqrt{t}}.\end{equation}
   Now we deal with $q=\pa_\xi  p_+^\frac{1}{2}(x,\xi)$.
By product rule for derivatives,
\begin{align}\label{chiderivative}
q(x,\xi)&=\frac{1}{2}\frac{sign(\xi)}{(1+|\xi|)^\frac{1}{2}}\left(\varphi(\tau)-h(x,\tau)\right)^\frac{1}{2}\nonumber\\
&+(1+|\xi|)^\frac{1}{2}\frac{t\sign(\xi)}{2}\left(\varphi(\tau)-h(x,\tau)\right)^{-\frac{1}{2}} \times \pa_{\tau}\left(\varphi(\tau)-h_2(\tau)\right).
\end{align}

The first term is innocent since $(\varphi-h)$ is bounded. For the second we notice,  that
\begin{equation}\label{keys}
t(1+|\xi|)^\frac12 \le t^\frac12(1+|\tau|)^\frac12, (\varphi-h)^{-\frac12} \le C (1+\tau )^{\frac12},
\end{equation}
where the second inequality comes from  lemma \ref{cotakc0}.  In addition, lemma~\ref{derivativesh} implies that $|\partial_\tau (\varphi-h)| \le \frac{\lar}{1+\tau^2}$. Thus, the desired
\begin{equation*}
|q(x,\xi)|\leq \lar.
\end{equation*}
follows.

\item That  $\pa_xq\in \dot{H}^{-\ep}$ follows from  the existence of a constant $q_\infty$ such that
$ q-q_\infty \in L^2$.
Since the only $x$ dependence of $q$ is through $A,c$ mwe will declare $q_\infty(t,\xi)=q(0,c_\infty, t,\xi)$.  Now by plugging into \eqref{chiderivative}  the bounds
 from lemmas \ref{cotakc0} amd  \ref{derivativesh},  it follows that  $|\nabla_{A,c}| q\le \lar$.
Then,  the mean value theorem applied to
$q$ as a function of $A,c$, yields  that for every $x \in \mathbb{R}$,
\[ |q(x)-q_\infty| \le \lar (|A(x)-0|+ |c(x)-c_\infty|) \]
and since $A$ is $H^3$ and $\int |c-c_\infty|^2 dx<C$ by assumption, the result follows. The proof to bound $||p^\frac12_+||_{\dot{H}^{-\ep}}$ follows similar steps.

\item Now we compute $\pa^k_x p^\frac{1}{2}_+(x,\xi)$ and $\pa^k_{x} q$.  By chain and product rule $\pa^k_x p^\frac{1}{2}_+(x,\xi)$  it is a sum of terms of the type $(1+|\xi|)^\frac12 (\varphi-h)^{\frac12-i} \Pi_{\sum \alpha _i=k} \pa^{\alpha_i} _x (\varphi- h)$ for $i=1,\ldots,k$. By \eqref{keys} and lemma \ref{derivativesh} the most singular term is

\[ (1+|\xi|)^\frac12 (\varphi-h)^{\frac12-k} |\pa_x h|^k  \le t^{-\frac12} \lar |c_x|+|A_x| \frac{(1+\tau)^k}{ (1+\tau)^k} \le  t^{-\frac12} \lar  \left(|c_x|+|A_x| \right).\]

We move to $q$.  We need to show that $ \pa_{x} q \in L^\infty$ for $p_+^\frac12 \in \Hw$ and that $\pa^k_x q \in L^2$ for $k=1,2,3$. By lemma~\ref{identities}
we only need to give  the details of the case  $\tau>1$.  Notice that when we differentiate in \eqref{chiderivative} the derivatives of the first term still remain innocent
as  $\pa_x^k\left(\varphi(\tau)-h(x,\tau)\right)^\frac{1}{2}$ has been shown to be bounded by $\sup_{1 \le i \le k}|\pa_x^i c|+|\pa_x^i A|$. For the second term, again product rule combined with lemma~\ref{cotakc0} and \ref{derivativesh} implies that the most singular term is

\[  t(1+|\xi|)^\frac12 (\varphi-h(x,\tau)^{-\frac12-k} |\pa_x h|^k \pa_{\tau}\left(\varphi(\tau)-h(x,\tau)\right) \le \lar \left(|\pa_x c|+ |\pa_x A| \right) \frac{ (1+\tau)^{k+1} } {(1+\tau)^k(1+\tau^2)}. \]

Notice  that the terms $\pa^{\alpha_i} _x (\varphi- h) $ and $\pa^{\alpha_i} _x \pa_\tau(\varphi- h)$ will be less singular respect to $\tau$  and  will
be controlled by $\lar (|\pa^{\alpha_i}_x c|+ |\pa^{\alpha_i}_x A|)$.

\end{enumerate}
Finally, we mention that in the case $c=1$, $\pa_x c=0$, thus their derivatives are bounded by powers of $\frac{1}{1+\tau^2}$ and lemma~\ref{cotakc0} implies that $(\varphi-h)^{-1} \le \frac{C}{B} (1+\tau^2)$.
The terms appearing in our various derivatives compensate exactly in the same way.

   \end{proof}

\begin{tabular}{l}
\textbf{Angel Castro} \\
 {\small Instituto de Ciencias Matem\'aticas-CSIC-UAM-UC3M-UCM}\\
  {\small Email: angel\_castro@icmat.es} \\
\\
\textbf{Diego C\'ordoba} \\
  {\small Instituto de Ciencias Matem\'aticas} \\
 {\small Consejo Superior de Investigaciones Cient\'ificas} \\
 {\small C/ Nicolas Cabrera, 13-15, 28049 Madrid, Spain} \\
  {\small Email: dcg@icmat.es} \\
\\
\textbf{Daniel Faraco} \\
  {\small Departamento de Matem\'aticas} \\
 {\small Universidad Aut\'onoma de Madrid} \\
 {\small Instituto de Ciencias Matem\'aticas-CSIC-UAM-UC3M-UCM}\\
  {\small Email: daniel.faraco@uam.es} \\

\end{tabular}

\end{document}